\documentclass[twoside,11pt]{article}

\usepackage{blindtext}

\usepackage{jmlr2e}

\usepackage{lastpage}

\ShortHeadings{Approximate full-conformal multi-task regression with reproducing kernels}{Razafindrakoto, Celisse and Lacaille}
\firstpageno{1}

\begin{document}

\title{Approximate full-conformal multi-task regression\\ 
with reproducing kernels}

\author{\name Davidson Lova Razafindrakoto
      \email davidson-lova.razafindrakoto@proton.me \\
      \addr Laboratoire SAMM\\
      Universit{\'e} Paris 1 Panth{\'e}on-Sorbonne\\
      90, rue de Tolbiac\\
      75634 PARIS CEDEX 13, FRANCE
      \AND
      \name Alain Celisse
      \email alain.celisse@univ-paris1.fr\\
      \addr Laboratoire SAMM\\
      Universit{\'e} Paris 1 Panth{\'e}on-Sorbonne\\
      90, rue de Tolbiac\\
      75634 PARIS CEDEX 13, FRANCE
      \AND
      \name J{\'e}r{\^o}me Lacaille
      \email jerome.lacaille@safrangroup.com \\
      \addr Safran Aircraft Engines\\
      Rond-Point Ren{\'e} Ravaud, R{\'e}au,\\
      77550 Moissy-Cramayel CEDEX, France}

\editor{}

\maketitle

\begin{abstract}
    Multi-task regression aims at jointly solving multiple regression problems, called tasks.
    Compared to solving each task separately, better performances can be achieved as long as the tasks are sufficiently related.
    Full-conformal prediction is a framework that formulates
    a data-dependent prediction-region containing the unknown output-vector at any prescribed confidence level.
    However, explicit computation of this prediction-region is intractable in general since it requires training infinitely many predictors.
    The present work focuses on multi-task regression
    in a Reproducing Kernel Hilbert Space (RKHS) of vector-valued functions.
    This computational issue is addressed by designing an approximating prediction-region containing the full-conformal one.
    This construction is carried out in two scenarios: $(i)$ when the inter-task covariance-matrix is known, and $(ii)$ when this matrix is estimated.
    In terms of volume, the tightness of this approximation
    is assessed theoretically by means of an upper-bound in the first scenario. It is also empirically proved to improve upon the split-conformal prediction on synthetic data in both scenarios.
\end{abstract}
\section{Introduction}
\label{sec.intro}
%
%
Multi-task regression consists in jointly solving multiple regression problems, called tasks \citep[see Section 2]{solnonMultitaskRegressionUsing2012}.
Each coordinate of the trained vector-valued predictor
corresponds to a task.
Intuitively, this vector-valued predictor is expected to outperform upon
separately training multiple scalar-valued predictors on different task as long as the different tasks are sufficiently related.
In practice, the link between tasks can be exploited by means of a dedicated regularization term exploiting existing connections between the tasks encoded by the inter-task covariance-matrix
(see \cite{zhangSurveyMultitaskLearning2021} for a survey on multi-task regression, and \cite{solnonMultitaskRegressionUsing2012} for a data dependent choice of such regularization).
Multi-task regression with a Reproducing Kernel Hilbert Space (RKHS) \citep{aronszajnTheoryReproducingKernels1950} has been detailed in the seminal work by \cite{micchelliLearningVectorvaluedFunctions2005} where each element of the RKHS is a vector-valued function.
%
%
%
%
Remarkably, any continuous vector-valued function can be uniformly approximated by an element of such an RKHS (on every compact subset of the input space) if the underlying matrix-valued reproducing kernel is universal \citep{evgeniouLearningMultipleTasks2005}.
%
%
%
%

In this context, the present goal is to design
and compute a data-dependent prediction-region
containing the unknown output-vector
at any prescribed confidence-level.
In conformal prediction, a prediction-region enjoying this guarantee was initially
formulated by \citet{vovkAlgorithmicLearningRandom2005},
and named Full-Conformal Prediction-region (\textbf{FullCP}-region).
However, the explicit computation of the \textbf{FullCP}-region requires
training as many predictors as the cardinality of the output space, which makes it impossible to compute in general.
In the (narrow) setting where trained-predictors admit
a closed-form expression as a function of the training output-vectors,
the \textbf{FullCP}-region can still be approximately recovered.
To be more specific, \citet{johnstoneExactApproximateConformal2024}
approximately recover the \textbf{FullCP}-region by computing boundary points of the \textbf{FullCP}-region along a specified finite number of directions.
However since exact recovery requires the boundary points
along the infinite number of possible directions,
parts of the \textbf{FullCP}-region may lay outside
the recovered prediction-region, inducing a lack of coverage.
The Split Conformal Prediction \citep[\textbf{SplitCP}]{papadopoulosInductiveConformalPrediction2008} was introduced to overcome this computational bottleneck by initially splitting the data and only training a single predictor from a subset of them.
The main drawback of \textbf{SplitCP} lies in the loss of information incurred by the splitting step which results in a variance inflation of the predictor and the wider (and then less informative) prediction-region
\citep[see Figure 3]{ndiayeStableConformalPrediction2022}.

By contrast, the present work rather describes a generic approximating scheme leading to a prediction-region containing the \textbf{FullCP}-region which does not require any initial data splitting and can be fully computed.
For single-task learning, such a scheme was first instantiated
by \citet{ndiayeStableConformalPrediction2022} and then by \citet{leeLeaveOneOutStableConformal2025}
with algorithmic stability-bounds as the main tool. More recently, \citet{razafindrakotoApproximateFullConformal2026} refined this approach by using influence functions.
Let us also mention that the tightness of the resulting prediction-interval
(one-dimensional region) was proved by means of a finite-sample upper-bound on the volume of the resulting region \citep[see Theorems~12 and~34]{razafindrakotoApproximateFullConformal2026}.
However, reformulating such approximations in the multi-task context remains highly challenging due to the inter-task relationships which must be preserved.
In this spirit, let us mention that \citet{messoudiCopulabasedConformalPrediction2021} and \citet{braunMultivariateStandardizedResiduals2026} formulated ellipsoidal \textbf{SplitCP}-regions by learning inter-task covariance-matrices (see \cite{dheurUnifiedComparativeStudy2025}
and \cite{braunMinimumVolumeConformal2025} for a survey of regions with other shapes).

\medskip

The main contribution of the present work is twofold:
(1) computing a prediction-region that contains the \textbf{FullCP}-region while
integrating an inter-task covariance-matrix,
and (2) providing numerical and theoretical evidences that the output prediction-region is a tight approximation to the \textbf{FullCP}-region in terms of volume.

To be more precise, Section~\ref{sec.statistical.framework} introduces
Multi-task regression within a Reproducing Kernel Hilbert Space (RKHS) of vector-valued predictors, and the conformal prediction basic concepts and notations.
Section~\ref{sec.approximate.full.conformal.perdiction}
introduces the generic approximation scheme of the \textbf{FullCP}-region and establishes the main algorithmic stability-bound used along the subsequent derivation.
This bound holds for various loss-functions (including robust ones)
and is instantiated for a particular class of matrix-valued kernels.
Section~\ref{sec.known.inter.task.covariance} illustrates how such schemes are instantiated when the inter-task covariance-matrix is known.
It gives rise to an approximate \textbf{FullCP}-region achieving the prescribed coverage.
In terms of volume, the tightness of this prediction-region is assessed by means of a finite-sample upper-bound exhibiting informative dependencies on the ambient dimension of the output space.
Section~\ref{sec.estimated.inter.task.covariance} goes further by addressing the case where the inter-task covariance-matrix is learned from the same data as the ones used for computing the predictor (no splitting step).
The tightness of the resulting prediction-region
is empirically assessed and compared with its \textbf{SplitCP}-counterpart which seems to output wider predictive confidence regions.

\section{Statistical framework}
\label{sec.statistical.framework}
This section discusses multi-task regression in the context of reproducing kernels and corresponding Reproducing Kernel Hilbert Spaces (RKHS) as well as conformal prediction.
More precisely, Section~\ref{sec.multi.task.regression.in.an.RKHS}
considers multi-task regression when the vector-valued predictor belongs to an RKHS and the vector-valued predictor is computed by numerically soving an optimization problem (representer theorem).
After which, Section~\ref{sec.conformal.prediction} formulates
the full-conformal prediction (\textbf{FullCP}) region
and discusses the coverage probability of such a region.

\subsection{Multi-task regression in an RKHS}
\label{sec.multi.task.regression.in.an.RKHS}
%
%
%
%

Let $\paren{X_{1}, Y_{1}}$,\ldots, $\paren{X_{n}, Y_{n}}$, $\paren{X_{n+1}, Y_{n+1}}$
denote $n+1$ copies of the random variable $\paren{X, Y} \in \mathcal{X} \times \mathcal{Y}$, with $\mathcal{X} \subset \mathbb{R}^{d}$ and
$\mathcal{Y} \subset \mathbb{R}^{p}$.
The objective of Multi-task kernel-regression is to formulate a predictor
$\hat{f} \in \mathcal{F}\paren{\mathcal{X}, \mathcal{Y}}$ (where $\mathcal{F}\paren{\mathcal{X}, \mathcal{Y}}$ denotes a candidate function class) from the observed data points $\paren{X_{1}, Y_{1}}$, \ldots, $\paren{X_{n}, Y_{n}}$ such that, given the new input $X_{n+1}$,
$\hat{f}$ outputs a prediction $\hat{f}\paren{X_{n+1}} \in\mathbb{R}^p$ for
the corresponding unobserved output-vector $Y_{n+1}\in\mathbb{R}^p$.

Following \cite{micchelliLearningVectorvaluedFunctions2005},
the candidate function class $\mathcal{F}\paren{\mathcal{X}, \mathcal{Y}}$ is chosen to be a Reproducing Kernel Hilbert Space \citep[RKHS]{aronszajnTheoryReproducingKernels1950}
of vector-valued functions, named $\mathcal{H}$.
%
%
%
As a Hilbert space, $\mathcal{H}$ is endowed with a scalar product $\dotprod{\bullet}{\bullet}{\mathcal{H}} : \mathcal{H} \times \mathcal{H} \to \mathbb{R}$ (inducing a norm $\|\bullet\|_{\mathcal{H}}: \mathcal{H} \to \mathbb{R}_{+}$).
Additionally, as an RKHS of vector-valued functions,
$\mathcal{H}$ admits a matrix-valued kernel denoted by
$K\paren{\bullet, \bullet}: \mathcal{X} \times \mathcal{X} \mapsto \mathbb{R}^{p \times p}$
(see \cite{micchelliKernelsMultiTask2004}, \cite{evgeniouLearningMultipleTasks2005},
\cite{micchelliLearningVectorvaluedFunctions2005},
\cite{caponnettoUniversalMultitaskKernels2008},
\cite{alvarezKernelsVectorValuedFunctions2012},
\cite{solnonMultitaskRegressionUsing2012}, and
\cite{audiffrenStabilityMultitaskKernel2013} for examples of matrix-valued kernels).
Two main assets of this matrix-valued kernel are:
(1) it allows for non-linear links between input and output prediction-vector, and
(2) it enables encoding the inter-tasks connections.

The next definition formulates the vector-valued predictor to be considered as a minimizer of a regularized empirical-risk.
\begin{definition}{\citep[{see Eq.~4.5}]{micchelliLearningVectorvaluedFunctions2005}}
\label{def.predictor}
Let $\lambda \in  \paren{0, +\infty}$ denote a regularization parameter, and $D$ be the training data set with cardinal $\abss{D} \in \mathbb{N}$.
Then, the vector-valued predictor $\hat{f}_{\lambda; D} \in \mathcal{H}$
stands for a minimizer of the regularized empirical-risk
$\widehat{\mathbf{R}}_{\lambda; D}\paren{\bullet}$ that is,
\begin{align}
    \label{eq.regularized.empirical.risk.function}
    \hat{f}_{\lambda; D}
    \in
    \argmin_{f \in \mathcal{H}}
    \widehat{\mathbf{R}}_{\lambda; D}\paren{f},
\end{align}
with, for every $f \in \mathcal{H}$,
\begin{align*}
    \widehat{\mathbf{R}}_{\lambda; D}\paren{f}
    := \frac{1}{\abss{D}}\sum_{\paren{x, y}\in D}
    \ell\paren{y, f(x)}
    + \lambda \normh{f}^2,
\end{align*}
where $\ell \paren{\bullet, \bullet} : \mathcal{Y} \times \mathcal{Y} \to \mathbb{R}$ denotes any loss-function.
\end{definition}
The loss-function involved in the regularized empirical-risk expression is not restricted to the quadratic one.
Several examples of such loss-functions are enumerated along Section~\ref{sec.stability.bounds}.
The regularization term is classically meant to control the overfitting phenomenon (provided $\lambda$ is chosen accordingly) by enforcing smoothness constraints on any solution. In the present context of multi-task regression, it also promotes solutions which account for the inter-task links which are encoded by the matrix-valued kernel (see for instance \citet{solnonMultitaskRegressionUsing2012} for an illustration of this last claim, and also Definition~\ref{eq.regularization.erf.matrix} for the explicit encoding).

\medskip

Let us now review the assumptions under which
the aforementioned predictor $\hat{f}_{\lambda; D} \in \mathcal{H}$
(see Definition~\ref{def.predictor}) is well-defined and can be numerically computed by solving an optimisation problem in the vector space $\mathbb{R}^{\abss{D}p}$.
%
%
\begin{assumption}
    There exists a constant $c_{\ell} \in \mathbb{R}$
    such that,
    \begin{align}
        \label{asm.lower.bounded.loss}
        \forall 
        \paren{y, u} \in \mathbb{R}^{p} \times \mathbb{R}^{p},
        \quad
        \ell\paren{y, u} \geq c_{\ell}.
        \tag{$c_{\ell}$-LwL}
    \end{align}
\end{assumption}

\begin{assumption}
    For every $y \in \mathcal{Y}$,
    \begin{align}
        \label{asm.conv.loss}
        u \in \mathbb{R}^{p} \mapsto \ell\paren{y, u} \in \mathbb{R}
        \mbox{ is a convex function.}
        \tag{ConvL}
    \end{align}
\end{assumption}

\begin{assumption}
    For every $y \in \mathcal{Y}$,
    \begin{align}
        \label{asm.lsc.loss}
        u \in \mathbb{R}^{p}
        \mapsto \ell\paren{y, u} \in \mathbb{R}
        \mbox{ is lower semi-continuous}.
        \tag{LscL}
    \end{align}
\end{assumption}

These three assumptions are standard for proving
existence results in convex optimization.
As an example, \eqref{asm.lower.bounded.loss} holds true for any
non-negative loss-functions with $c_{\ell} = 0$,
and \eqref{asm.conv.loss} and \eqref{asm.lsc.loss} are fulfilled by
the loss-functions discussed in Section~\ref{sec.stability.bounds}.

\begin{notation}
By convention, every vector is a column vector (by default),
and for any integer $m>0$ and $i \in \brac{1, \ldots, m}$,
$e_i := \paren{\delta_{ij}}_{j=1}^{m} \in \mathbb{R}^{m \times 1}$
denotes the canonical basis vector in $\mathbb{R}^{m \times 1}$,
where $\delta_{i j}=1$, if $i=j$ and 0 otherwise.
Moreover, for any integer $m \in \mathbb{N}^{*}$,
let $\mathrm{Sym}_{m}^{+}\paren{\mathbb{R}}$ designate
the set of positive semi-definite square matrices
in $\mathbb{R}^{m \times m}$, and 
$\mathrm{Sym}_{m}^{++}\paren{\mathbb{R}}$ 
the subset of non-singular ones.
%
%
%
%

\notparagraph{Kronecker product}
For every matrix $A \in \mathbb{R}^{m \times n}$
and every $B\in \mathbb{R}^{p \times q}$,
let $A \otimes B \in \mathbb{R}^{mp \times nq}$ denote
the Kronecker product between $A$ and $B$ that is,
\begin{align}
    \label{eq.kronecker.product}
    A \otimes B
    = \begin{bmatrix}
        a_{1, 1} B
        &
        \ldots
        &
        a_{1, n} B
        \\
        \vdots
        &
        \ddots
        &
        \vdots
        \\
        a_{m,1} B
        &
        \ldots
        &
        a_{m, n} B
    \end{bmatrix}.
\end{align}

\notparagraph{Representer Hilbert subspace}
Let $\mathbf{X} \in \mathbb{R}^{n \times d}$ denote
the $n \times d$ design matrix
resulting from stacking
the row input vectors $X_{1}^{T}, \ldots, X_{n}^{T} \in \mathbb{R}^{1 \times d}$ that is,
\begin{align}
    \label{eq.input.matrix}
    \mathbf{X}
    := \begin{bmatrix}
        X_{1}
        &
        \ldots
        &
        X_{n}
    \end{bmatrix}^{T},
\end{align}
and $\mathcal{A}_{\mathbf{X}} \subseteq \mathcal{H}$ be
the finite dimensional Hilbert subspace of $\mathcal{H}$,
given by
\begin{align}
    \label{eq.subspace}
    \mathcal{A}_{\mathbf{X}}
    &
    := \mathrm{Span}\brac{
        K\paren{\bullet, X_i} e_l \in \mathcal{H}:
        \forall i \in \brac{1, \ldots, n},
        \forall l \in \brac{1, \ldots, p}
    }\notag
    \\
    &
    =
    \brac{
        \sum_{i=1}^{n} K\paren{\bullet, X_i} \paren{e_i^{T} \otimes I_{p}} V \in \mathcal{H}:
        \forall V \in \mathbb{R}^{np}
    },
\end{align}
where, for every $i \in \brac{1, \ldots, n}$, $e_i \in \mathbb{R}^{n \times 1}$ denotes the $ith$ canonical basis vector of $\mathbb{R}^n$.

\notparagraph{Gram matrix}
Let $\mathbf{K}_{\mathbf{X}} \in \mathbb{R}^{np \times np}$
stand for the Gram matrix given by
\begin{align}
    \label{eq.gram.matrix}
    \mathbf{K}_{\mathbf{X}}
    &
    := \begin{bmatrix}
    K\paren{X_1, X_1}
    &
    K\paren{X_1, X_2}
    &
    \ldots
    &
    K\paren{X_1, X_n}
    \\
    K\paren{X_2, X_1}
    &
    K\paren{X_2, X_2}
    &
    \ldots
    &
    K\paren{X_2, X_n}
    \\
    \vdots
    &
    \vdots
    &
    \ddots
    &
    \vdots
    \\
    K\paren{X_n, X_1}
    &
    K\paren{X_n, X_2}
    &
    \ldots
    &
    K\paren{X_n, X_n}
    \end{bmatrix}.
\end{align}
\end{notation}
%
%



\medskip

The next result proves that the predictor $\hat{f}_{\lambda; D} \in \mathcal{H}$ (see Definition~\ref{def.predictor}) is well-defined.
\begin{lemma}[Well-defined predictor]
\label{lm.predictor.well.defined}

%
%
%
Assume \eqref{asm.lsc.loss},
\eqref{asm.conv.loss} and \eqref{asm.lower.bounded.loss} hold true.
Then, the regularized empirical-risk function (see Eq.~\ref{eq.regularized.empirical.risk.function})
is a lower semi-continuous $2\lambda$-strongly convex and coercive function.
Moreover,
the predictor $\hat{f}_{\lambda; D} \in \mathcal{H}$
does exist and is unique.
\end{lemma}
\begin{proof} The proof is deferred to Appendix~\ref{proof.predictor.well.defined}.
\end{proof}

\medskip

In all coming results, $\lambda \in \paren{0, +\infty}$ is a regularization parameter, and $D$ refers to the data set $\paren{X_{1}, Y_{1}}, \ldots, \paren{X_{n}, Y_{n}}$.
The so-called Representer theorem \citep{aronszajnTheoryReproducingKernels1950} applies to the regularized empirical-risk from Eq.~\eqref{eq.regularized.empirical.risk.function}.
As a result, the minimizer is expressed as an element of $\mathcal{A}_{\mathbf{X}}$, that is, $\hat{f}_{\lambda; D} \in \mathcal{A}_{\mathbf{X}}$.
\begin{lemma}
\label{lm.representer}
Assume \eqref{asm.conv.loss}, \eqref{asm.lsc.loss}
and \eqref{asm.lower.bounded.loss} hold true.
Then, the minimizer $\hat{f}_{\lambda; D} \in \mathcal{H}$
of the regularized empirical-risk function
(see Eq.~\ref{eq.regularized.empirical.risk.function})
can be decomposed as follows,
for every $x \in \mathcal{X}$,
\begin{align}
    \label{eq.representer}
    \hat{f}_{\lambda; D}
    = \sum_{i=1}^{n}
    K\paren{\bullet, X_i} \paren{e_i^{T} \otimes I_{p}} \widehat{W}_{\lambda; D},
\end{align}
for some weight matrix
$\widehat{W}_{\lambda; D} \in \mathbb{R}^{np}$.
\end{lemma}
The proof is deferred to Appendix~\ref{proof.representer}.
Owing to the multi-task setting, for all $i \in \brac{1, \ldots, n}$, the feature maps $K\paren{\bullet, X_i}$ are bounded linear-operators from $\mathbb{R}^{p}$ to $\mathcal{H}$.
Moreover the dimension of the weight vector $\widehat{W}_{\lambda; D} \in \mathbb{R}^{np}$ reflects the computational burden incurred by jointly learning $p$ tasks.

\medskip

Computing the minimizer is made possible by plugging the previous expression in the regularised empirical-risk given by Eq.~\ref{eq.regularized.empirical.risk.function}, which allows for rephrasing the optimization problem as finding the vector $W$ minimizing the next function over $\mathbb{R}^{np}$.
\begin{definition}
    For any predictor $f \in \mathcal{A}_{\mathbf{X}}$ (see Eq.~\ref{eq.subspace}),
    there exists $W \in \mathbb{R}^{np}$
    such that,
    \begin{align}
        \label{eq.regularization.erf.matrix}
        \widehat{R}_{\lambda; D}\paren{W}
        :=
        \frac{1}{n}\sum_{i=1}^{n}
        \ell\paren{Y_i, 
            \paren{e_i^{T} \otimes I_{p}}
            \mathbf{K}_{\mathbf{X}}W  
        }
        + \lambda W^{T} \mathbf{K}_{\mathbf{X}} W
        = \widehat{\mathbf{R}}_{\lambda; D}\paren{f},
    \end{align}
    where $\mathbf{K}_{\mathbf{X}} \in \mathbb{R}^{np \times np}$ is the Gram matrix (see Eq.~\ref{eq.gram.matrix}).
\end{definition}
The proof is deferred to Appendix~\ref{proof.regularization.erf.matrix}.
Similarly to the single-task setting \citep{razafindrakotoApproximateFullConformal2026},
the empirical-risk is now expressed a function of the weight-vector $W \in \mathbb{R}^{np}$.
The inter-task connection (encoded by the matrix-valued kernel) is enforced by this regularization term where the Gram matrix plays a central role.
This is made clear by \citet{solnonMultitaskRegressionUsing2012}, where
the objective was to select the best matrix-valued kernel to effectively
capture the right inter-task connection.

{%
This vector-representation of any predictor in $\mathcal{A}_{\mathbf{X}}$
provides an optimization-problem over a finite-dimensional vector-space
whose minimizer (unique over the range of the Gram matrix $\mathbf{K}_{\mathbf{X}}$ by Proposition~\ref{prop.vector.predictor.well.defined})
represents the predictor $\hat{f}_{\lambda; D}$ (see Eq.~\ref{def.predictor}).
In general, an iterative optimization-scheme is used to recover
this minimizer since
it does not admit a closed-form expression
(except for some special cases such as when the loss-function is the quadratic loss-function).
In particular,
this holds true for
the loss-function described in Proposition~\ref{prop.example.lip.loss},
when the univariate cost-function $c\paren{\bullet, \bullet} : \mathbb{R} \times \mathbb{R} \to \mathbb{R}$ is for instance
the Logcosh function \citep{salehStatisticalPropertiesLogcosh2022},
the smoothed-pinball function \citep{zheng2011gradient} or
the pseudo-Huber function \citep{charbonnier1994two}.
In fact, in the subsequent numerical experiments (see Figure~\ref{fig.coverage.stableCP}),
the Scipy's Newton-CG \citep{MinimizemethodNewtonCGSciPyV1180}
is used to recover this minimizer when the loss-function is based on the Logcosh function \citep{salehStatisticalPropertiesLogcosh2022}.

}




%
\subsection{Conformal Prediction}
\label{sec.conformal.prediction}
%
%
When not stated otherwise,
$\paren{X_1, Y_1}$, \dots, $\paren{X_{n}, Y_{n}}$, and
$\paren{X_{n+1}, Y_{n+1}}$ are assumed to be exchangeable \citep{vovkAlgorithmicLearningRandom2005}.
The purpose of conformal prediction is
to design a \emph{confidence prediction-region} $\widehat{C}_{\alpha}\paren{X_{n+1}} \subseteq \mathcal{Y}$
based on a learning algorithm applied to output a trained predictor $\hat{f} \in \mathcal{F}\paren{\mathcal{X}, \mathcal{Y}}$
such that, for every $\alpha \in \paren{0, 1}$,
\begin{align}\label{eq.confidence.prediction.region}
    \mathbb{P}\croch{ Y_{n+1} \in \widehat{C}_{\alpha}\paren{X_{n+1}} } \geq 1-\alpha.
\end{align}
The full-conformal prediction (\textbf{FullCP}) region \citep{vovkAlgorithmicLearningRandom2005},
named $\widehat{C}_{\lambda; \alpha}^{\full}\paren{X_{n+1}}$,
writes as
\begin{align}
\label{eq.conformalRegion}
    \widehat{C}_{\lambda; \alpha}^{\full}\paren{X_{n+1}}
    := 
    \left\{ 
        y \in \mathcal{Y} : 
        \widehat{\pi}_{\lambda; D}^{\full}\paren{X_{n+1}, y} > \alpha
    \right\},
\end{align}
where $\widehat{\pi}_{\lambda; D}^{\full}\paren{X_{n+1}, \bullet}$ is the so-called \emph{conformal p-value function} given by the following definition from \emph{non-conformity scores}.
\begin{definition}[Non-conformity scores and Conformal p-value]
\label{def.conformal.pvalues}

Let $\widehat{s}_{D^{y}}(\bullet, \bullet):\mathcal{Y} \times \mathcal{Y} \to \mathbb{R}$
stand for a \emph{non-conformity function}, and
$\hat{f}_{\lambda; D^y} \in \mathcal{H}$
denote a predictor (see Eq.~\ref{def.predictor}),
trained over the data set $D^y$ containing
$(X_{1}, Y_{1}), \ldots, (X_{n}, Y_{n})$ and $(X_{n+1}, y)$, for every output-value $y\in\mathcal{Y}$.
Then the non-conformity score of the couple $(X_i,Y_i)$ provided $y$ and the data set $D^y$,
is given by
\begin{align}  
\label{def.scores}  
\begin{aligned}
    S_{\lambda; D^{y}} \paren{X_i, Y_i} & := \widehat{s}_{D^{y}}\paren{Y_i, \hat{f}_{\lambda; D^y}\paren{X_i}},
    &&\mbox{if }1\leq i\leq n \\
    S_{\lambda; D^{y}} \paren{X_{i}, y} & := \widehat{s}_{D^{y}}\paren{Y_i, \hat{f}_{\lambda; D^y}\paren{X_i}},
    &&\mbox{if } i=n+1.
\end{aligned}
\tag{NCScores}  
\end{align}
Furthermore, the conformal p-value function $\widehat{\pi}_{\lambda; D}^{\full}\paren{X_{n+1}, \bullet} : \mathcal{Y} \to \croch{\frac{1}{n+1}, 1}$ is
given by, for every $y \in \mathcal{Y}$,
\begin{equation}
\label{def.conformal.pvalue.function}
    \tag{Cp-value}
        \widehat{\pi}_{\lambda; D}^{\full}\paren{X_{n+1}, y}
        := \frac{1
            +
            \sum_{i = 1}^{n}
            \mathbbm{1}\left\{
                S_{\lambda; D^{y}} \paren{X_i, Y_i}
                \geq
                S_{\lambda; D^{y}} \paren{X_{n+1}, y}
            \right\}
        }{n+1}.
    \end{equation}
\end{definition}
Starting from \eqref{def.scores},
$S_{\lambda; D^{y}}\paren{X_i, Y_i}$ quantifies how strongly the point $\paren{X_i, Y_i}$
``deviates from" the data set $D^{y}$, where large values are equivalent to large deviations.
Therefore, going back to Eq.~\eqref{eq.conformalRegion}, the \textbf{FullCP}-region contains values $y \in \mathcal{Y} \subset \mathbb{R}^p$
for which the point $\paren{X_{n+1}, y}$ presents
a relatively weak deviation from the points within the set $D^{y}$.
%
%
Provided the predictor $\hat{f}_{\lambda; D^{y}}$,
if test output-value $y \in \mathcal{Y}$ is contained within the \textbf{FullCP}-region,
then output-value $y$ is expected to be relatively similar to the prediction $\hat{f}_{\lambda; D^{y}}\paren{X_{n+1}}$
w.r.t. to the non-conformity measure $\widehat{s}_{D^{y}}\paren{\bullet, \bullet}$.
In other words, $S_{\lambda; D^{y}}\paren{X_{n+1}, y}$ is expected to be small relative to the non-conformity scores
$S_{\lambda; D^{y}}\paren{X_{1}, Y_{1}}, \ldots, S_{\lambda; D^{y}}\paren{X_{n}, Y_{n}}$, that is,
when \eqref{def.conformal.pvalue.function} is large.

Let us emphasize that compared to the previous work
\cite[Definition 1]{razafindrakotoApproximateFullConformal2026}, an additional novelty of the present one owes to considering a non-conformity function that can depend on the data set $D^{y}$.
This is illustrated in Section~\ref{sec.estimated.inter.task.covariance}
where the non-conformity function integrates an inter-task covariance-matrix estimator.

\medskip

The main motivation for considering a conformal prediction-region stems from the next property which guarantees a minimum confidence (called coverage) in a \emph{distribution-free} setting.
\begin{theorem} 
\label{thm.coverage}
\citep[in Section 2.2.5]{vovk2022}
For every confidence level $\alpha \in \left[\frac{1}{n+1}, 1\right)$,
the \textbf{FullCP}-region $\widehat{C}_{\lambda; \alpha}^{\full}\paren{X_{n+1}}$ enjoys the following coverage guarantee
\begin{equation*}
    \mathbb{P}\croch{Y_{n+1} \in \widehat{C}_{\lambda; \alpha}^{\full}\paren{X_{n+1}}} \geq 1 - \alpha.
\end{equation*}
As such the \textbf{FullCP}-region $\widehat{C}_{\lambda; \alpha}^{\full}\paren{X_{n+1}}$ is a confidence prediction-region.

Furthermore, if the non-conformity scores $S_{\lambda; D^{Y_{n+1}}} \paren{X_1, Y_1}$, \dots, $S_{\lambda; D^{Y_{n+1}}} \paren{X_{n+1}, Y_{n+1}}$ are almost surely distinct, then,
the coverage-probability is also bounded from above,
that is,
\begin{equation*}
    \mathbb{P}\croch{Y_{n+1} \in \widehat{C}_{\lambda; \alpha}^{\full}\paren{X_{n+1}}} \leq 1 - \alpha + \frac{1}{n+1}.
\end{equation*}
\end{theorem}
By \eqref{def.conformal.pvalue.function}, the confidence level $\alpha$ cannot be lower than $\frac{1}{n+1}$.
This entails that any $\alpha < \frac{1}{n+1}$ results in a non-informative \textbf{FullCP}-region $\widehat{C}_{\lambda; \alpha}^{\full}\paren{X_{n+1}}$, that is, $\widehat{C}_{\lambda; \alpha}^{\full}\paren{X_{n+1}} = \mathcal{Y}$, which trivially fulfils the upper-bound on the coverage-probability.
\section{Approximate full-conformal prediction}
\label{sec.approximate.full.conformal.perdiction}
The objective of the present section is to reformulate
the generic approximation scheme introduced in \citet[see Definition 4]{razafindrakotoApproximateFullConformal2026}.
%
%
More precisely, Section~\ref{sec.approximation.scheme} formulates
a prediction-region containing the \textbf{FullCP}-region
and recalls the notion of \emph{thickness} which quantifies the tightness of this prediction-region in terms of volume.
Since this approximation relies on algorithmic stability,
Section~\ref{sec.stability.bounds} formulates the classical algorithmic stability-bound for the predictor, and instantiates this bound for a particular class of kernels.
\subsection{Approximation scheme}
\label{sec.approximation.scheme}
%
For every test output-value $y \in \mathcal{Y}$,
let $D^{y}$ denote the data set containing
$\paren{X_1, Y_1}$, \dots, $\paren{X_{n}, Y_{n}}$ and $\paren{X_{n+1}, y}$.
The following definition formulates a prediction-region which
contains the \textbf{FullCP}-region based on upper- and lower-approximate non-conformity scores.

\begin{definition}[Approximate \textbf{FullCP}-region]
    \label{def.approximate.fullCP.region}
    For every test output-value $y \in \mathcal{Y}$,
    let $\widetilde{S}_{\lambda; D^{y}}^{\up}\paren{\bullet, \bullet}$
    (and $\widetilde{S}_{\lambda; D^{y}}^{\lo}\paren{\bullet, \bullet}$) designate an upper-approximate non-conformity score function (resp. a lower one) that is, for every $\paren{x, u} \in \mathcal{X} \times \mathcal{Y}$,
    \begin{align}
        \label{eq.bound.score}
        \widetilde{S}_{\lambda; D^{y}}^{\lo}\paren{x, u}
        \leq S_{\lambda; D^{y}}\paren{x, u}
        \leq \widetilde{S}_{\lambda; D^{y}}^{\up}\paren{x, u},
        \quad \mbox{a.s.}.
    \end{align}
    Then, let $\widetilde{C}_{\lambda; \alpha}^{\up}\paren{X_{n+1}}$ name
    an upper-approximate \textbf{FullCP}-region,
    given by
    \begin{align}
        \label{eq.generic.approx.confidence.region}
        \widetilde{C}_{\lambda; \alpha}^{\up}\paren{X_{n+1}}
        :=
        \brac{
        y \in \mathcal{Y} :
        \widetilde{\pi}_{\lambda; D}^{\up} \paren{X_{n+1}, y} > \alpha
        },
    \end{align}
    where $\widetilde{\pi}_{\lambda; D}^{\up} \paren{X_{n+1}, \bullet} : \mathcal{Y} \to \croch{\frac{1}{n+1}, 1}$
    designates the so-called upper-approximate conformal p-value function given by, for every test output-value $y \in \mathcal{Y}$,
    \begin{align*}
        \widetilde{\pi}_{\lambda; D}^{\up} \paren{X_{n+1}, y}
        := \frac{
        1 + \sum_{i=1}^{n}
        \mathbbm{1}
        \brac{
        \widetilde{S}_{\lambda; D^{y}}^{\up}\paren{X_i, Y_i}
        \geq \widetilde{S}_{\lambda; D^{y}}^{\lo}\paren{X_{n+1}, y}
        }
        }{n+1}.
    \end{align*}
    %
    %
    %
\end{definition}

Compared with \citet[see Definition 4]{razafindrakotoApproximateFullConformal2026}, the present one is a generalization. As it will be clarified in Section~\ref{sec.estimated.inter.task.covariance} (see Lemma~\ref{lm.stable.score.maha}), the present definition allows for more general types of correction (for instance multiplicative ones) than the additive one formally detailed in \citet[see Definition 4]{razafindrakotoApproximateFullConformal2026}. This difference turns out to be particularly helpful when dealing with more complex score functions \citep{ndiayeStableConformalPrediction2022,leeLeaveOneOutStableConformal2025,razafindrakotoApproximateFullConformal2026}.
%
%
%
%

\medskip

A main motivation for the previous approximation scheme
owes to the next result
which guarantees that the upper-approximate \textbf{FullCP}-region
contains the \textbf{FullCP}-region and thus, inherits its coverage guarantee.
\begin{theorem}
    \label{thm.coverage.approximate.region}
    For any control-level $\alpha \in \left[\frac{1}{n+1}, 1\right)$,
    the upper-approximate \textbf{FullCP}-region $\widetilde{C}_{\lambda; \alpha}^{\up}\paren{X_{n+1}}$
    (see Eq.~\ref{eq.generic.approx.confidence.region})
    contains the \textbf{FullCP}-region $\widehat{C}_{\lambda; \alpha}^{\full}\paren{X_{n+1}}$ that is,
    $
        \widehat{C}_{\lambda; \alpha}^{\full}\paren{X_{n+1}}
        \subseteq
        \widetilde{C}_{\lambda; \alpha}^{\up}\paren{X_{n+1}}
    $.    
    Then Theorem~\ref{thm.coverage} implies that
    \begin{align*}
        \mathbb{P}\croch{
            Y_{n+1}
            \in
            \widetilde{C}_{\lambda; \alpha}^{\up}\paren{X_{n+1}}
        }
        \geq \mathbb{P}\croch{
            Y_{n+1}
            \in
            \widehat{C}_{\lambda; \alpha}^{\full}\paren{X_{n+1}}
        }
        \geq 1 - \alpha,
    \end{align*}
    making $\widetilde{C}_{\lambda; \alpha}^{\up}\paren{X_{n+1}}$ a confidence prediction-region.
\end{theorem}
\begin{proof}
    Direct consequence of Lemma~\ref{lm.sandwiching} and Theorem~\ref{thm.coverage}.
\end{proof}

Since the upper-approximate \textbf{FullCP}-region is larger that the \textbf{FullCP}-region,
the next definition introduces the notion of \emph{thickness} which quantifies its tightness
through the volume of its symmetric difference with the \textbf{FullCP}-region.
\begin{definition}
    \label{def.thickness}
    Let $\mathrm{THK}_{\lambda; \alpha}\paren{X_{n+1}}$
    be the \emph{thickness} of the approximate \textbf{FullCP}-region
    $\widetilde{C}_{\lambda; \alpha}^{\up}\paren{X_{n+1}}$
    given by the Lebesgue measure
    of its symmetric difference with the \textbf{FullCP}-region
    $\widehat{C}_{\lambda; \alpha}^{\full}\paren{X_{n+1}}$
    that is,
    \begin{align}
        \label{eq.thickness}
        \mathrm{THK}_{\lambda; \alpha}\paren{X_{n+1}} := \leb{
            \widetilde{C}_{\lambda; \alpha}^{\up}\paren{X_{n+1}}
            \Delta
            \widehat{C}_{\lambda; \alpha}^{\full}\paren{X_{n+1}}
        } = \leb{
            \widetilde{C}_{\lambda; \alpha}^{\up}\paren{X_{n+1}}
            \setminus
            \widehat{C}_{\lambda; \alpha}^{\full}\paren{X_{n+1}}
        }.
    \end{align}
\end{definition}

In practice, the \emph{thickness} cannot be evaluated since it relies
on the intractable \textbf{FullCP}-region. To derive a computable empirical upper-bound
of the \emph{thickness}, a lower-approximate \textbf{FullCP}-region, named $\widetilde{C}_{\lambda; \alpha}^{\lo}\paren{X_{n+1}}$,
is formulated from the non-conformity scores in Eq.~\eqref{eq.score} as
\begin{align}
    \label{eq.generic.lower.approx.confidence.region}
    \widetilde{C}_{\lambda; \alpha}^{\lo}\paren{X_{n+1}}
    :=
    \brac{
    y \in \mathcal{Y} :
    \widetilde{\pi}_{\lambda; D}^{\lo} \paren{X_{n+1}, y} > \alpha
    },
\end{align}
where
$\widetilde{\pi}_{\lambda; D}^{\lo} \paren{X_{n+1}, \bullet} : \mathcal{Y} \to \croch{\frac{1}{n+1}, 1}$
denotes the lower-approximate conformal p-value function,
given by, for every test output-value $y \in \mathcal{Y}$,
\begin{align*}
    \widetilde{\pi}_{\lambda; D}^{\lo} \paren{X_{n+1}, y}
    := \frac{
    1 + \sum_{i=1}^{n}
    \mathbbm{1}
    \brac{
    \widetilde{S}_{\lambda; D^{y}}^{\lo}\paren{X_i, Y_i}
    \geq \widetilde{S}_{\lambda; D^{y}}^{\up}\paren{X_{n+1}, y}
    }
    }{n+1}.
\end{align*}

By obviously sandwiching the \textbf{FullCP}-region between the upper- and
lower-approximate \textbf{FullCP}-regions,
the subsequent result ensures that the \emph{thickness} is bounded from above by the next computable upper-bound.
\begin{lemma}[Sandwiching]
\label{lm.sandwiching}
Assuming Eq.~\eqref{eq.bound.score} hold true, $\widehat{C}_{\lambda; \alpha}^{\full}\paren{X_{n+1}}$ (\textbf{FullCP}-region) is sandwiched between
its upper $\widetilde{C}_{\lambda; \alpha}^{\up}\paren{X_{n+1}}$ and
lower $\widetilde{C}_{\lambda; \alpha}^{\lo}\paren{X_{n+1}}$ approximations that is,
\begin{align*}
    \widetilde{C}_{\lambda; \alpha}^{\lo}\paren{X_{n+1}}
    \subseteq
    \widehat{C}_{\lambda; \alpha}^{\full}\paren{X_{n+1}}
    \subseteq
    \widetilde{C}_{\lambda; \alpha}^{\up}\paren{X_{n+1}},
    \quad \mbox{a.s.}.
\end{align*}
It results that the \emph{thickness} is bounded from above
\begin{align}
\label{eq.confidence.region.gap.bound}
    \mathrm{THK}_{\lambda; \alpha}\paren{X_{n+1}} \leq
    \mathcal{V}\left(
        \widetilde{C}_{\lambda; \alpha}^{\up}\paren{X_{n+1}}
        \setminus
        \widetilde{C}_{\lambda; \alpha}^{\lo}\paren{X_{n+1}}
    \right),
    \quad \mbox{a.s.}.
\end{align}
\end{lemma}
The proof is deferred to Appendix~\ref{proof.sandwiching}.
Let us emphasize that such a proxy for the \emph{thickness}
is not available in the \textbf{SplitCP} framework.
In practice, this empirical upper-bound can be computed as long as
the upper and the lower-approximate \textbf{FullCP}-regions can.
Figures~\ref{fig.thickness.stableCP} and~\ref{fig.thickness.GlobalEllipsoidCP} display the value this upper bound with respect to different influential quantities such as the sample size $n$.

\subsection{Algorithmic stability-bounds}
\label{sec.stability.bounds}
%
%
The objective is now to reformulate the algorithmic stability-bound derived by \cite{audiffrenStabilityMultitaskKernel2013}
for our purposes.
Let us first introduce some more notation and
discuss an additional assumption.
\begin{notation}
Let $\norm{\bullet} : \mathbb{R}^{p} \to \mathbb{R}_+$ stand for
the norm induced by the Euclidean scalar product over $\mathbb{R}^{p}$, and
$\norm{\bullet}_{\mathrm{op}} : \mathbb{R}^{p \times p} \mapsto \mathbb{R}_+$,
the corresponding operator norm given by,
for every matrix $B \in \mathbb{R}^{p \times p}$,
\begin{align*}
    \norm{B}_{\mathrm{op}} :=
    \sup_{u \in \mathbb{R}^{p} \setminus \brac{0_{\mathbb{R}^{p}}}}
    \frac{\norm{B u}}{\norm{u}}.
\end{align*}    
\end{notation}
\begin{assumption}
    There exists a finite constant $\rho_p \in \paren{0, +\infty}$
    such that, for every $y \in \mathcal{Y}$,
    \begin{align}
        \label{asm.lip.loss}
        u \in \mathbb{R}^{p}
        \mapsto \ell\paren{y, u} \in \mathbb{R}
        \mbox{ is $\rho_p$-Lipschitz continuous
        w.r.t. }\norm{\bullet}.
        \tag{$\rho_p$-LipL}
    \end{align}
\end{assumption}

This assumption along with \eqref{asm.conv.loss}
are used by \citet{audiffrenStabilityMultitaskKernel2013}
to provide uniform stability-bounds in the context of multi-task kernel-regression
given an RKHS with Ridge-type regularization.
The next lemma provides
a class of loss-functions which fulfil these assumptions.
%
\begin{proposition}
    \label{prop.example.lip.loss}
    Let $c\paren{\bullet, \bullet} : \mathbb{R} \times \mathbb{R} \mapsto \mathbb{R}$
    denote a cost-function taking scalar values,
    which is convex and Lipschitz continuous w.r.t. its second argument that is, there exists a constant $\rho \in \paren{0, +\infty}$ such that, for every $a \in \mathbb{R}$,
    $c\paren{a, \bullet}$ is convex and $\rho$-Lipschitz continuous.
    Then, the loss-function $\ell\paren{\bullet, \bullet}: \mathcal{Y} \times \mathcal{Y} \mapsto \mathbb{R}$ given by
    for every $\paren{y, u} \in \mathcal{Y} \times \mathcal{Y}$,
    \begin{align*}
        \ell\paren{y, u} := \frac{1}{\sqrt{p}}\sum_{l=1}^{p}c\paren{y_l, u_l},
    \end{align*}
    fulfills \eqref{asm.conv.loss}
    and \eqref{asm.lip.loss} with $\rho_{p} = \rho$.
\end{proposition}
The proof is deferred to Appendix~\ref{proof.example.lip.loss}.
This procedure derives
a multivariate loss-function from
a univariate cost-function.
As such, one can lift a robust univariate cost-function \citep{salehStatisticalPropertiesLogcosh2022, zheng2011gradient, charbonnier1994two} into a multivariate one.
One such loss function is used in the subsequent numerical experiments (see Figure~\ref{fig.coverage.stableCP}).
Let us note that for this particular choice
of loss-function the Lipschitz-constant
does not depend on the dimension $p$ of the output space $\mathcal{Y}$.

\medskip

A key property for deriving our approximations is the algorithmic stability. In the present context, it results from the strong convexity property of the Ridge-type regularization term, which can be extended to any strongly convex regularization term \citep[see Eq.~1]{ndiayeStableConformalPrediction2022}.
%
Therefore the algorithmic stability ensures that
small differences in the training data set entails
small differences in the subsequent predictors.
The following result provides an upper-bound on the latter difference
involving explicitly known quantities such as
the Lipschitz constant of the loss-function,
the operator norm of the kernel,
and the Ridge-type regularization parameter.
\begin{lemma}[Uniform stability-bound]
    \label{lm.stability.bounds}
    Assume
    \eqref{asm.conv.loss},
    \eqref{asm.lip.loss} and
    \eqref{asm.lower.bounded.loss}
    hold true.
    Then for every test output-value $y \in \mathcal{Y}$, we get
    \begin{align*}
        \normh{
            \hat{f}_{\lambda; D^{y}}
            - \hat{f}_{\lambda^{+}; D}
        }
        \leq
        \frac{
            \rho_{p}\norm{K\paren{X_{n+1}, X_{n+1}}}_{\mathrm{op}}^{\frac{1}{2}}
        }{2 \lambda\paren{n+1}},
    \end{align*}
    with $\lambda^{+} := \frac{n+1}{n} \lambda$.
\end{lemma}
The proof is deferred to Appendix~\ref{proof.stability.bounds}.
Let us briefly mention that this result is an adaptation of
an intermediate one from \citet[see Appendix A]{audiffrenStabilityMultitaskKernel2013}.
A slight difference in our notation is that,
following \citet[see Eq.~20]{bousquetStabilityGeneralization2002},
the truncated empirical-risk, evaluated over $n$ data points,
present a normalization of $\frac{1}{n+1}$.
Whereas in our case, the normalization is exactly the inverse of the number of data points (see Definition~\ref{def.predictor}).
In order to recover their normalization, the predictor $\hat{f}_{\lambda^{+}; D}$
trained over $n$ data points integrates an inflated regularization parameter $\lambda^{+} = \frac{n+1}{n} \lambda$,
instead of $\lambda$ in $\hat{f}_{\lambda; D^{y}}$.

\medskip

\notparagraph{Example of matrix-valued kernel}
Within the subsequent numerical experiments,
the following matrix-valued kernel is considered,
for every $\paren{x, t} \in \mathcal{X} \times \mathcal{X}$,
\begin{align*}
    K\paren{x, t} := k\paren{x, t}\Gamma \in \mathrm{Sym}_{p}^{+}\paren{\mathbb{R}},
\end{align*}
where $k\paren{\bullet, \bullet}: \mathcal{X} \times \mathcal{X}$
is a scalar-valued kernel,
and $\Gamma \in \mathrm{Sym}_{p}^{++}\paren{\mathbb{R}}$
is a symmetric positive definite matrix.
Here, $k\paren{x, t}$ encodes the similarity between $x$ and $t$, while the matrix $\Gamma$ encodes the inter-task correlation. In other words, the link between the feature vectors and the inter-task connections have been disentangled. This kernel was studied by \citet[with $\Gamma = M^{-1}$]{solnonMultitaskRegressionUsing2012}
and expanded upon by \citet[with $\Gamma = T$]{liOptimalRatesRegularized2022}
for function-valued predictors.
For every $\paren{x, t} \in \mathcal{X} \times \mathcal{X}$, the operator norm of the matrix-valued kernel is given by
\begin{align*}
    \norm{K\paren{x, t}}_{\mathrm{op}} = \abss{k\paren{x, t}} \norm{\Gamma}_{\mathrm{op}}.
\end{align*}
If the scalar-valued kernel is uniformly bounded (as is the case for the RBF Gaussian kernel) then, the above-mentioned algorithmic stability-bound improves at the standard rate of $O\paren{\frac{1}{\lambda n}}$ as  $\lambda n$ grows for any fixed output dimension $p$.
This entails that $\hat{f}_{\lambda^{+}; D}$ is an increasingly better approximation to
$\hat{f}_{\lambda; D^{y}}$ for larger values of $\lambda n$.

\section{Known inter-task covariance-matrix}
\label{sec.known.inter.task.covariance}
The present section studies an
instance of the approximation scheme from Definition~\ref{def.approximate.fullCP.region}
called \textbf{StableCP} when the inter-task covariance-matrix is known.
Section~\ref{sec.known.cov.ncms}
first introduces the non-conformity scores incorporating the inter-task covariance-matrix,
and derives the corresponding upper and lower approximations.
Then, Section~\ref{sec.known.cov.explicit.computation}
details its the \textbf{StableCP}-region expression, states the guarantee it enjoys, and numerically illustrates what this guarantee looks like on a synthetic data set.
The explicit rates for its \emph{thickness} are established in Section~\ref{sec.known.cov.thickness.upper.bound}, while Section~\ref{sec.known.cov.numerical.experiments}
numerically illustrates the evolution of the \emph{thickness} and
compares  the \textbf{StableCP}- and \textbf{SplitCP}-regions (see Appendix~\ref{sec.split.cp}) in terms of their volume relative to that of the \textbf{OracleCP}-region (see Appendix~\ref{sec.oracle.cp}).
\subsection{Non-conformity scores and approximations}
\label{sec.known.cov.ncms}

Since the inter-task covariance-matrix $\Gamma \in \mathrm{Sym}_{p}^{++}\paren{\mathbb{R}}$
is presently known, the following non-conformity score $S_{\lambda; D^{y}}^{\Gamma}\paren{x, u}$
measures the quality of the prediction-vector $\hat{f}_{\lambda; D^{y}}\paren{x} \in \mathbb{R}^{p}$
as an approximation of the output-vector $u \in \mathbb{R}^{p}$
through the $\Gamma^{-1}$-Mahalanobis distance.

%
%
\begin{definition}[Mahalanobis non-conformity score]
For every $\paren{x, u} \in \mathcal{X} \times \mathcal{Y}$,
let $S_{\lambda; D^{y}}^{\Gamma}\paren{x, u}$
denote the Mahalanobis non-conformity score
of the point $\paren{x, u}$
with respect to the data set $D^{y}$ that is,
\begin{align}
    \label{eq.score}
    S_{\lambda; D^{y}}^{\Gamma}
    \paren{x, u}
    :=
    \norm{
        \Gamma^{-\frac{1}{2}}
        \paren{
            u - \hat{f}_{\lambda; D^{y}}(x)
        }
    }.
\end{align}    
\end{definition}
This non-conformity score was already considered in the context of conformal prediction
(for \textbf{FullCP} see \cite{johnstoneExactApproximateConformal2024},
for \textbf{SplitCP} see \cite{messoudiEllipsoidalConformalInference2022} and
\cite{braunMultivariateStandardizedResiduals2026}).
Closest to the present work, \citet{johnstoneExactApproximateConformal2024}
approximately recover the \textbf{FullCP}-region when the prediction-vector $\hat{f}_{\lambda; D^{y}}\paren{x}$
is an affine transformation of the test output-value $y \in \mathcal{Y}$. However, this work departs from the present one in two respects:
\begin{enumerate}
    \item since exact recovery requires computing boundary points
along an infinite number of directions,
parts of the \textbf{FullCP}-region may lay outside the recovered prediction-region, then loosing some coverage guarantee in practice.
By contrast, the upper-approximate \textbf{FullCP}-region
formulated in Definition~\ref{def.approximate.fullCP.region} (namely the \textbf{StableCP}-region) fully contains the \textbf{FullCP}-region.

    \item the explicit expression of the \textbf{FullCP}-region holds with the quadratic loss-function that is, for every $\paren{y, u} \in \mathbb{R}^{p} \times \mathbb{R}^{p}$,
$\ell\paren{y, u} = \norm{y - u}^2$, whereas the present work builds approximate \textbf{FullCP}-region with any loss-function fulfilling \eqref{asm.conv.loss}, \eqref{asm.lip.loss} and \eqref{asm.lower.bounded.loss}.
\end{enumerate}

\medskip

Harnessing the algorithmic stability-bound on the prediction (see Lemma~\ref{lm.stability.bounds}),
the next result details the expressions of the upper- and lower-approximate non-conformity scores.

\begin{lemma}
    \label{lm.stable.score.norm}
    Assume
    \eqref{asm.conv.loss},
    \eqref{asm.lip.loss} and
    \eqref{asm.lower.bounded.loss}
    hold true.
    Then, for every test output-value $y \in \mathcal{Y}$
    and for every $\paren{x, u} \in \mathcal{X} \times \mathcal{Y}$,
    \begin{align*}
        \widetilde{S}_{\lambda; D^{y}}^{\Gamma, \lo}
        \paren{x, u}
        \leq
        S_{\lambda; D^{y}}^{\Gamma}
        \paren{x, u}
        \leq
        \widetilde{S}_{\lambda; D^{y}}^{\Gamma, \up}
        \paren{x, u},
    \end{align*}
    where $\widetilde{S}_{\lambda; D^{y}}^{\Gamma, \lo}\paren{x, u}$ and
    $\widetilde{S}_{\lambda; D^{y}}^{\Gamma, \up}\paren{x, u}$
    stand for the upper- and lower-approximate non-conformity scores given by
    \begin{align*}
        \widetilde{S}_{\lambda; D^{y}}^{\Gamma, \lo}\paren{x, u}
        &
        :=
        S_{\lambda^{+}; D}^{\Gamma}\paren{x, u}
        - \widehat{\tau}_{\lambda}^{\Gamma}(x),
        \\
        \widetilde{S}_{\lambda; D^{y}}^{\Gamma, \up}\paren{x, u}
        &
        :=
        S_{\lambda^{+}; D}^{\Gamma}\paren{x, u}
        + \widehat{\tau}_{\lambda}^{\Gamma}(x),
    \end{align*}
    with $\widehat{\tau}_{\lambda}^{\Gamma}(x)$ being a score stability-bound given by
    \begin{align}
        \label{eq.score.stability.bound}
        \widehat{\tau}_{\lambda}^{\Gamma}(x)
        &
        := 
        \norm{\Gamma^{-\frac{1}{2}}K\paren{x, x} \Gamma^{-\frac{1}{2}}}_{\mathrm{op}}^{\frac{1}{2}}
        \frac{
            \rho_{p} \norm{K\paren{X_{n+1}, X_{n+1}}}_{\mathrm{op}}^{\frac{1}{2}}       
        }{2\lambda \paren{n+1}}.
    \end{align}
\end{lemma}
The proof is deferred to Appendix~\ref{proof.stable.score.norm}.
Let us mention that such an additive uniform (w.r.t. $y$) correction $\widehat{\tau}_{\lambda}^{\Gamma}(\bullet)$
is not new in the context of single task learning \citep{ndiayeStableConformalPrediction2022,leeLeaveOneOutStableConformal2025,razafindrakotoApproximateFullConformal2026}.
The specificity of the present multi-task context lies in the inter-task relationships, emphasized here by $\norm{\Gamma^{-\frac{1}{2}}K\paren{\bullet, \bullet} \Gamma^{-\frac{1}{2}}}_{\mathrm{op}}$.
This first factor in Eq.~\eqref{eq.score.stability.bound}
outlines the importance of the relationship between the matrix-valued kernel $K\paren{\bullet, \bullet}$
and the covariance-matrix $\Gamma$ of the output-vector.
To be more specific, if the matrix-valued kernel is chosen to be
$K\paren{\bullet, \bullet} := k\paren{\bullet, \bullet}\Gamma$,
then
\begin{align*}
    \norm{\Gamma^{-\frac{1}{2}}K\paren{\bullet, \bullet} \Gamma^{-\frac{1}{2}}}_{\mathrm{op}}
    = \abss{k\paren{\bullet, \bullet}},
\end{align*}
where any dependence on $\Gamma$ (and thus on the output-space dimension $p$) cancels out.
Therefore, the dimension $p$ does not impact the correction $\widehat{\tau}_{\lambda}^{\Gamma}(\bullet)$ through this term.
This suggests that the non-conformity score (see Eq.~\ref{eq.score})
is a sounding choice when the predictor incorporates the matrix-valued kernel described above.

The right-most factor in Eq.~\eqref{eq.score.stability.bound} (already discussed at the end of Section~\ref{sec.stability.bounds}) straightforwardly results from the uniform stability-bound.
If $\sup_{x\in\mathcal{X}} \norm{K(x,x)}_{\mathrm{op}}<+\infty$, then the upper bound decays at rate $O((n\lambda)^{-1})$ as $n\lambda \to +\infty$.
%
%
%
%

\subsection{Explicit computation}
\label{sec.known.cov.explicit.computation}

Incorporating the upper and lower non-conformity scores introduced in Lemma~\ref{lm.stable.score.norm}, the resulting upper- and lower-approximate \textbf{FullCP}-regions are formulated as follows.

\begin{definition}[Lower and upper \textbf{StableCP}-regions]
Let $\widetilde{C}_{\lambda; \alpha}^{\Gamma, \up} \paren{X_{n+1}}$
stand for the upper \textbf{StableCP}-region 
with the Mahalanobis non-conformity. It is given by
\begin{align}
    \label{def.approx.fcpr.norm}
    \widetilde{C}_{\lambda; \alpha}^{\Gamma, \up} \paren{X_{n+1}}
    := \brac{
    y \in \mathcal{Y}:
    \frac{
    1 + \sum_{i=1}^{n}
    \mathbbm{1}\brac{
        \widetilde{S}_{\lambda; D^{y}}^{\Gamma, \up}\paren{X_i, Y_i}
        \geq \widetilde{S}_{\lambda; D^{y}}^{\Gamma, \lo}\paren{X_{n+1}, y}
        }
    }{n+1} > \alpha
    }.
\end{align}
    
\end{definition}
In order to derive the explicit expression of the upper \textbf{StableCP}-region,
let us first introduce some more notation and then state a result on quantiles.
\begin{notation}
Let $m \in \brac{0, \ldots, n}$ denote a number of data points.
For any control-level $\alpha \in \left[\frac{1}{n+1}, \frac{m+1}{n+1}\right)$,
let $i_{n, \alpha}^{m} \in \brac{1, \ldots, m}$
be the index given by
\begin{align}
    \label{eq.index.alpha}
    i_{n, \alpha}^{m} := \ceil{\paren{n+1}\paren{1 - \alpha} - \paren{n - m}},
\end{align}
Moreover, for any given real-valued sequence $a_{1}, \ldots, a_{m}$,
let $a_{(1)}, \ldots, a_{(m)}$ stand for the sequence of elements ordered such that $a_{\paren{1}} \leq \ldots \leq a_{\paren{m}}$. 
\end{notation}
Owing to the choice of non-conformity score (see Eq.~\ref{eq.score}),
the next result details the closed-form expression of the upper \textbf{StableCP}-region.
\begin{proposition}
    \label{prop.upper.stableCP.region.fixed}
    For any control-level $\alpha \in \left[\frac{1}{n+1}, 1\right)$,
    the upper \textbf{StableCP}-region $\widetilde{C}_{\lambda; \alpha}^{\Gamma, \up}\paren{X_{n+1}}$
    is the region enclosed by the $\Gamma^{-1}$-ellipsoid centred around $\hat{f}_{\lambda^{+}; D}\paren{X_{n+1}}$
    with a radius of $\widehat{Q}_{\lambda; D^{+}}^{\Gamma, \up}(\alpha)
    + \widehat{\tau}_{\lambda}^{\Gamma}\paren{X_{n+1}}$, that is,
    \begin{align*}
        \widetilde{C}_{\lambda; \alpha}^{\Gamma, \up}\paren{X_{n+1}}
        = \brac{
            y \in \mathcal{Y} :
            \norm{
                \Gamma^{-\frac{1}{2}}\paren{
                    y - \hat{f}_{\lambda^{+}; D}\paren{X_{n+1}}
                }
            } \leq \widehat{Q}_{\lambda; D^{+}}^{\Gamma, \up}(\alpha)
            + \widehat{\tau}_{\lambda}^{\Gamma}\paren{X_{n+1}}
        },
    \end{align*}
    where $\widehat{Q}_{\lambda; D^{+}}^{\Gamma, \up}(\alpha)$ is
    given by
    \begin{align}
        \label{eq.fixed.ncs.quantile}
        \widehat{Q}_{\lambda; D^{+}}^{\Gamma, \up}(\alpha)
        :=
        \norm{
            \Gamma^{-\frac{1}{2}}
            \paren{
                Y_{\paren{i_{n, \alpha}^{n}}}
                - \hat{f}_{\lambda^{+}; D}\paren{X_{\paren{i_{n, \alpha}^{n}}}}
            }
        } + \widehat{\tau}_{\lambda}^{\Gamma}\paren{X_{\paren{i_{n, \alpha}^{n}}}}
    \end{align}
    (see Eq.~\ref{eq.score.stability.bound} for $\widehat{\tau}_{\lambda}^{\Gamma}\paren{\bullet}$
    and Eq.~\ref{eq.index.alpha} for $i_{n, \alpha}^{n}$).
\end{proposition}
The proof is deferred to Appendix~\ref{proof.upper.StableCP.region.fixed}.
Being a region enclosed by an ellipsoid, the shape of the upper \textbf{StableCP} reveals the relationship between the coordinates of the prediction-vector.
Moreover, the additive correction term $\widehat{\tau}_{\lambda}^{\Gamma}(\bullet)$ is reflected by a thickening
of the quantile value involved in the expression of the radius.
Based on this simple expression, the shape (and most importantly the volume) of
the upper \textbf{StableCP}-region can be computed exactly, which comes in handy when computing the proxy for the \emph{thickness} in Lemma~\ref{lm.sandwiching}.
In contrast, such an exact computation cannot be done for the recovered prediction-region
computed by \citet{johnstoneExactApproximateConformal2024}.

\bigskip
\notparagraph{Empirical assessment of the coverage-probability}
The following experiments aim at illustrating that
(in practice on a synthetic data set) the upper \textbf{StableCP}-region is a confidence prediction-region for any control-level $\alpha \in \left[\frac{1}{n+1}, 1\right)$, and every regularization parameter $\lambda \in \paren{0, +\infty}$.
The code is available at \url{https://github.com/Davidson-Lova/approximate_full-conformal_multi-task_kernel_regression.git}.

Following \citet{braunMultivariateStandardizedResiduals2026} for each repetition, $n+1$ independent copies of $\paren{X, Y}$ are sampled,
where $Y \sim f\paren{X} + T\paren{X}B$ (see Appendix~\ref{sec.synthetic.data.set} for details).
\sloppy For $\lambda \in \brac{10^{-3}, 10^{-2}, 10^{-1}}$
and for $\alpha$ within a grid, the upper \textbf{StableCP}-region
is computed at the location $X_{n+1}$ from a predictor trained with ``Newton-CG'' with:
(1) the vector Logcosh loss-function \citep{salehStatisticalPropertiesLogcosh2022} derived by applying Proposition~\ref{prop.example.lip.loss} ,
(2) a matrix-valued kernel of the form $K\paren{\bullet, \bullet} = k\paren{\bullet, \bullet} \Gamma$,
where the scalar-valued kernel is the scikit-learn's Laplacian kernel \citep{Laplacian_kernel},
and the inter-task covariance-matrix $\Gamma$ is given by
\begin{align*}
    \Gamma = \begin{bmatrix}
        1 & \frac{1}{2}\\
        \frac{1}{2} & 1
    \end{bmatrix}.
\end{align*}
At the end of each repetition, if the unknown output-vector $Y_{n+1}$
is contained in the upper \textbf{StableCP}-region,
the coverage-value is set to $1$,  and $0$ otherwise.
Finally, across $100$ repetitions,
for each $\lambda \in \brac{10^{-3}, 10^{-2}, 10^{-1}}$
and each $\alpha$ within a grid,
the coverage-values are averaged into the empirical coverage-probability.

\begin{figure}[H]
    \centering
    \includegraphics[width=0.8\textwidth]{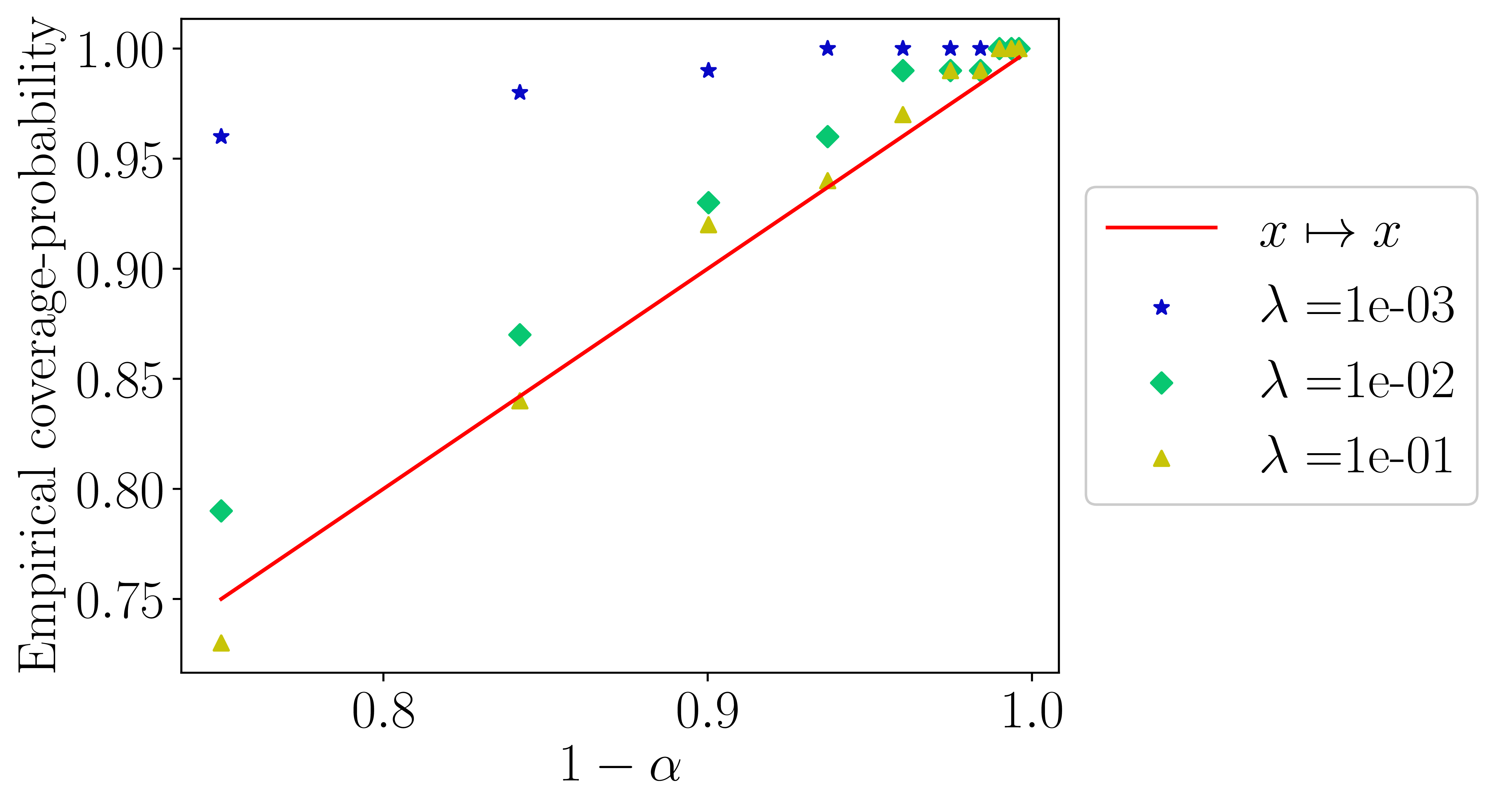}
    \caption{
        Evolution of the empirical coverage-probability of the upper \textbf{StableCP}-region,
        for $\lambda \in \brac{10^{-3}, 10^{-2}, 10^{-1}}$
        and for $\alpha$ within a logarithmic grid.
    }
    \label{fig.coverage.stableCP}
\end{figure}
The results are displayed in Figure~\ref{fig.coverage.stableCP}.
The red-line represents the desired confidence level $1 - \alpha$
over which the coverage-probability is guaranteed to lie.
The empirical coverage-probabilities are mostly (up to some variability) above the red-line,
that is, the prediction-regions mostly do empirically enjoy the desired coverage property.

Additionally, a smaller value of $\lambda$ induces, a larger empirical coverage-probability.
This is due to the additive correction $\widehat{\tau}_{\lambda}^{\Gamma}\paren{\bullet}$, and
thus the thickening of the radius of the upper \textbf{StableCP}-region (see Proposition~\ref{prop.upper.stableCP.region.fixed}), getting larger for a smaller value of $\lambda$.

\subsection{Finite sample upper-bound on the \emph{thickness}}
\label{sec.known.cov.thickness.upper.bound}
Let us first discuss the key assumptions
used in deriving a finite-sample upper-bound on the \emph{thickness} Eq.~\eqref{eq.thickness.stable.CP}.

\begin{assumption}
    There exists a constant $\kappa^{\Gamma} \in \paren{0, +\infty}$,
    such that,
    \begin{align}
        \label{asm.bounded.gamma.kernel}
        \forall x \in \mathcal{X},
        \quad
        \norm{K\paren{\bullet, x} \Gamma^{-\frac{1}{2}}}_{\mathrm{op}} =
        \norm{\Gamma^{-\frac{1}{2}} K\paren{x, x} \Gamma^{-\frac{1}{2}}}_{\mathrm{op}}^{\frac{1}{2}} \leq \kappa^{\Gamma}.
        \tag{$\kappa^{\Gamma}$-Bd$K_\Gamma$}
    \end{align}
\end{assumption}

Assumption \eqref{asm.bounded.gamma.kernel} links back to the discussion
below Lemma~\ref{lm.stable.score.norm}, where it was argued that a suitable matrix-valued kernel should satisfy that, for every $x\in\mathcal{X}$, $K(x,x)$ must have almost the same eigensapces as the ones of $\Gamma$. Here one specifies more strongly this constraint
by introducing a constant $\kappa$ governing the relationship between $K(\bullet,\bullet)$ and $\Gamma$.
For instance if $K(x,x) = k(x,x)\Gamma$, with $\sup_{x\in\mathcal{X}}k(x,x)\leq \kappa<+\infty$, then it would result that \eqref{asm.bounded.gamma.kernel} holds with $\kappa = \kappa^\Gamma$. In this case, let us also notice that the constant $\kappa^\Gamma$ does not depend on the output-space dimension $p$.
%
%
%
%
%
%

\begin{assumption}
    There exists a constant $C_{\ell} \in \paren{0, +\infty}$,
    such that,
    \begin{align}
        \label{asm.bounded.expected.loss}
        \mathbf{R}_{0}\paren{0} = \mathbb{E}\croch{\ell\paren{Y, 0}} \leq C_{\ell},
        \tag{$C_{\ell}$-BdEL}
    \end{align}
    where $\mathbf{R}_{0}\paren{\bullet} : \mathcal{H} \mapsto \mathbb{R}$ denotes the risk function given for every $f \in \mathcal{H}$ by
    \begin{align*}
        \mathbf{R}_{0}\paren{f} := \mathbb{E}\croch{\ell\paren{Y, f\paren{X}}}.
    \end{align*}

\end{assumption}

\begin{assumption}
    There exists a constant $C_{\mathcal{Y}}\paren{p} \in \paren{0, +\infty}$,
    such that,
    \begin{align}
        \label{asm.bounded.gamma.y}
        \forall y \in \mathcal{Y},
        \quad
        \norm{\Gamma^{-\frac{1}{2}}y} \leq C_{\mathcal{Y}}\paren{p}.
        \tag{$C_{\mathcal{Y}}\paren{p}$-BdY}
    \end{align}
\end{assumption}
These three assumptions ensure that the risk function
is not constant and equal to $+\infty$.
A similar hypothesis (formulated without the matrix $\Gamma$)
was formulated by \citet[see Hypothesis 5]{audiffrenStabilityMultitaskKernel2013}
in the multi-task setting to derive an algorithmic stability-bound with a quadratic loss.
An analogous assumption was also made by \citet{bousquetStabilityGeneralization2002} in the single task setting (see the comments following Definition~19).

\begin{assumption}
    The risk function $\mathbf{R}_{0}\paren{\bullet}$
    admits a minimizer over
    the hypothesis space $\mathcal{H}$ that is,
    \begin{align}
        \label{asm.risk.minimizer.attained}
        \brac{
            f \in \mathcal{H} :
            \mathbf{R}_{0}\paren{f} = \inf_{g\in \mathcal{H}} \mathbf{R}_{0}\paren{g}
        } \neq \emptyset.
        \tag{MinR}
    \end{align} 
\end{assumption}
\eqref{asm.risk.minimizer.attained} ensures that
there exists a function in $\mathcal{H}$, whose norm
is greater than that of the population counterpart $f_{\lambda}$ (see Lemma~\ref{lm.reg.risk.minimizer})
of the predictor $\hat{f}_{\lambda; D}$ (see Definition~\ref{def.predictor}).
Among all such functions, let us note $f_{\mathcal{H}} \in \mathcal{H}$
the one with the smallest norm that is,
\begin{align*}
    f_{\mathcal{H}} \in \argmin_{
        \substack{
            f \in \mathcal{H}\\
            \mathbf{R}_{0}\paren{f} = \inf_{g\in \mathcal{H}} \mathbf{R}_{0}\paren{g}
        }
    } \norm{f}_{\mathcal{H}}.
\end{align*}

\begin{assumption}
    There exists a constant $C_{\mathcal{H}} \in \paren{0, +\infty}$, such that,
    \begin{align}
        \label{asm.not.source.condition}
        \normh{f_{\mathcal{H}}} \leq C_{\mathcal{H}}.
        \tag{$C_\mathcal{H}$-SRC}
    \end{align}
\end{assumption}
\eqref{asm.not.source.condition} is a consequence
of a classical source condition in a well-specified setting.
Such source condition was formulated by \citet[see SRC]{liOptimalSobolevNorm2024}
in the context of vector-valued regularized kernel least-square regression.
When their regularity parameter $\beta$ is greater than $1$,
then (by continuous embedding) their source condition (SRC) implies \eqref{asm.not.source.condition}.

\begin{assumption}[Polynomial decay]
Let us note $\mu_{1}\paren{\Gamma} \geq \ldots \geq \mu_{p}\paren{\Gamma}>0$
the eigenvalues of the matrix $\Gamma \in \mathrm{Sym}_{p}^{++}\paren{\mathbb{R}}$.
Let assume there exist an exponent $\gamma \in \paren{1, +\infty}$ and
a constant $C_{\Gamma} \in \paren{0, +\infty}$
invariant w.r.t. the size $p$ of $\Gamma$
(which is also the dimension of the output space $\mathcal{Y} \subseteq \mathbb{R}^{p}$)
such that
\begin{align}
    \label{asm.eigenvalue.decay}
    \forall \ell \in \brac{1, \ldots, p},
    \quad
    \mu_{\ell}\paren{\Gamma}
    \leq C_{\Gamma} \ell^{-\gamma}.
    \tag{$\gamma$-EVD$\Gamma$}
\end{align}    
\end{assumption}
{Training a predictor that jointly solves multiple tasks
(instead of building one predictor per task)
relies on the assumption that the tasks are strongly related.
Therefore \eqref{asm.eigenvalue.decay} encodes the strength of the inter-task link through
the decay rate of the eigenvalues of $\Gamma$.
Larger values of $\gamma$ means a stronger link between the tasks.
Such a polynomial decay assumption is not new.
For instance, \citet[see Theorem 1]{steinwartOptimalRatesRegularized2009}
states a similar assumption for the integral operator associated with a scalar-valued kernel.
}

\medskip

Now equipped with the above assumptions, we are now in position to derive an upper-bound for the \emph{thickness}, namely $\mathrm{THK}_{\lambda; \alpha}^{\Gamma}\paren{X_{n+1}}$,
of the upper \textbf{StableCP}-region. It is 
given by the Lebesgue measure of its symmetric difference with
the \textbf{FullCP}-region that is,
\begin{definition}[Thickness of the upper-\textbf{StableCP}-region]
\begin{align}
    \label{eq.thickness.stable.CP}
    \mathrm{THK}_{\lambda; \alpha}^{\Gamma}\paren{X_{n+1}}
    := \leb{
        \widetilde{C}_{\lambda; \alpha}^{\Gamma, \up}\paren{X_{n+1}}
        \setminus
        \widehat{C}_{\lambda; \alpha}^{\Gamma, \full}\paren{X_{n+1}}
    }.
\end{align}    
\end{definition}
For deriving an upper bound on the thickness, a key step consists in constructing a lower \textbf{StableCP}-region,
denoted by $\widetilde{C}_{\lambda; \alpha}^{\Gamma, \lo}\paren{X_{n+1}}$ and
given by
\begin{align*}
    \widetilde{C}_{\lambda; \alpha}^{\Gamma, \lo}\paren{X_{n+1}}
    := \brac{
        y \in \mathcal{Y}:
        \frac{1
        + \sum_{i=1}^{n}
        \mathbbm{1}\brac{
            \widetilde{S}_{\lambda; D^{y}}^{\Gamma, \lo}\paren{X_i, Y_i}
            \geq \widetilde{S}_{\lambda; D^{y}}^{\Gamma, \up}\paren{X_{n+1}, y}
        }}{n+1}>\alpha
    }.
\end{align*}
This so-called lower \textbf{StableCP}-region is a lower approximation to the \textbf{FullCP}-region that is,
\begin{proposition}
\begin{align*}
    \widetilde{C}_{\lambda; \alpha}^{\Gamma, \lo}\paren{X_{n+1}} \subset \widehat{C}_{\lambda; \alpha}^{\Gamma, \full}\paren{X_{n+1}}.
\end{align*}
Furthermore, the \textbf{FullCP}-region can be sandwiched between its lower and upper \textbf{StableCP}-regions as follows.
\begin{align*}
    \widetilde{C}_{\lambda; \alpha}^{\Gamma, \lo}\paren{X_{n+1}}
    \subset \widehat{C}_{\lambda; \alpha}^{\Gamma, \full}\paren{X_{n+1}}
    \subset \widetilde{C}_{\lambda; \alpha}^{\Gamma, \up}\paren{X_{n+1}}.
\end{align*}
\end{proposition}
As will be clarified in what follows, this sandwich inequality turns out to be useful for upper bounding the thickness.

\medskip

An important step towards deriving this upper bound is the next result, which details the closed-form expression of the lower \textbf{StableCP}-region
as a region enclosed by an ellipsoid.
\begin{lemma}
    \label{lm.lower.stableCP.region.fixed}
    For any control-level $\alpha \in \left[\frac{1}{n+1}, 1\right)$,
    the lower \textbf{StableCP}-region $\widetilde{C}_{\lambda; \alpha}^{\Gamma, \lo}\paren{X_{n+1}}$
    is the region enclosed by the $\Gamma^{-1}$-ellipsoid centred around $\hat{f}_{\lambda^{+}; D}\paren{X_{n+1}}$
    with a radius of $\croch{
        \widehat{Q}_{\lambda; D^{+}}^{\Gamma, \lo}(\alpha)
        - \widehat{\tau}_{\lambda}^{\Gamma}\paren{X_{n+1}}
    }_{+}$,
    that is,
    \begin{align*}
        \widetilde{C}_{\lambda; \alpha}^{\Gamma, \lo}\paren{X_{n+1}}
        = \brac{
            y \in \mathcal{Y} :
            \norm{
                \Gamma^{-\frac{1}{2}}\paren{
                    y - \hat{f}_{\lambda^{+}; D}\paren{X_{n+1}}
                }
            } \leq \widehat{Q}_{\lambda; D^{+}}^{\Gamma, \lo}(\alpha)
            - \widehat{\tau}_{\lambda}^{\Gamma}\paren{X_{n+1}}
        },
    \end{align*}
    where $\widehat{Q}_{\lambda; D^{+}}^{\Gamma, \lo}(\alpha)$
    is given by
    \begin{align}
        \label{eq.fixed.lower.ncs.quantile}
        \widehat{Q}_{\lambda; D^{+}}^{\Gamma, \lo}(\alpha)
        :=
        \norm{
            \Gamma^{-\frac{1}{2}}
            \paren{
                Y_{\paren{i_{n, \alpha}^{n}}}
                - \hat{f}_{\lambda^{+}; D}\paren{X_{\paren{i_{n, \alpha}^{n}}}}
            }
        } - \widehat{\tau}_{\lambda}^{\Gamma}\paren{X_{\paren{i_{n, \alpha}^{n}}}}
    \end{align}
    (see Eq.~\ref{eq.score.stability.bound} for $\widehat{\tau}_{\lambda}^{\Gamma}\paren{\bullet}$
    and Eq.~\ref{eq.index.alpha} for $i_{n, \alpha}^{n}$).
\end{lemma}
Proof is analogue to Proposition~\ref{prop.upper.stableCP.region.fixed}.
Comparing the expressions of the upper \textbf{StableCP}-region (Proposition~\ref{prop.upper.stableCP.region.fixed}) and the present lower one,
the only differences lie in the sign in front of the correction $\widehat{\tau}_{\lambda}^{\Gamma}\paren{\bullet}$.
This suggests that the smaller the correction,
the smaller the difference between these two prediction-regions,
and thus the smaller the \emph{thickness}.
\medskip

The next result formalizes this intuition
by providing an upper-bound on $\mathrm{THK}_{\lambda; \alpha}^{\Gamma}\paren{X_{n+1}}$.
%
%
\begin{lemma}[Coarse upper-bound on the \emph{thickness}]
    \label{lm.thickness.empirical.bound.maha}
    Assume~\eqref{asm.conv.loss},
    \eqref{asm.lip.loss} and
    \eqref{asm.lower.bounded.loss} hold true.
    Then, the \emph{thickness} $\mathrm{THK}_{\lambda; \alpha}^{\Gamma}\paren{X_{n+1}}$
    of the upper \textbf{StableCP}-region
    $\widetilde{C}_{\lambda; \alpha}^{\Gamma, \up}\paren{X_{n+1}}$
    is bounded from above by
    \begin{align*}
        \mathrm{THK}_{\lambda; \alpha}^{\Gamma}\paren{X_{n+1}}
        &\leq 
        2 \frac{
            \rho_{p} \norm{\Gamma}_{\mathrm{op}}^{\frac{1}{2}} \paren{\widehat{\kappa}^{\Gamma}}^{2}
        }{\lambda\paren{n+1}}
        \abss{\mathrm{det}\paren{\Gamma^{\frac{1}{2}}}}
        \frac{\pi^{\frac{p}{2}}}{\paren{\frac{p}{2}}!}
        p \paren{
            \widehat{Q}_{\lambda; D}^{\Gamma}(\alpha)
            + \frac{
                \rho_{p} \norm{\Gamma}_{\mathrm{op}}^{\frac{1}{2}} \paren{\widehat{\kappa}^{\Gamma}}^{2}
            }{\lambda\paren{n+1}}
        }^{p-1},
    \end{align*}
    where $\paren{\widehat{\kappa}^{\Gamma}}^{2}$ denotes
    the kernel operator norm bound given by
    \begin{align}
        \label{eq.kernel.norm.bound}
        \paren{\widehat{\kappa}^{\Gamma}}^{2}
        :=
        & 
        \max_{i \in \brac{1, \ldots, n+1}}
        \norm{\Gamma^{-\frac{1}{2}}K\paren{X_i, X_i}\Gamma^{-\frac{1}{2}}}_{\mathrm{op}},
    \end{align}
    and $\widehat{Q}_{\lambda; D}^{\Gamma}(\alpha)$
    stands for the non-conformity score quantile given by
    \begin{align}
        \label{eq.score.quantile}
        \widehat{Q}_{\lambda; D}^{\Gamma}(\alpha)
        := \norm{\Gamma^{-\frac{1}{2}}\paren{Y_{\paren{i_{n, \alpha}^{n}}} - \hat{f}_{\lambda^{+}; D}\paren{X_{\paren{i_{n, \alpha}^{n}}}}}}.
    \end{align}
    %
    %
    %
\end{lemma}
The proof is deferred to Appendix~\ref{proof.thickness.empirical.bound.maha}.
In line with the intuition, the first factor is an upper-bound
on the correction term $\widehat{\tau}_{\lambda}^{\Gamma}\paren{\bullet}$ (see Eq.~\ref{eq.score.stability.bound}), highlighting that the volume straightforwardly depends on the strength of the correction.
Moreover, the quantile value $\widehat{Q}_{\lambda; D}^{\Gamma}(\alpha)$
(see Eq.~\ref{eq.score.quantile}) appears in the upper-bound on the \emph{thickness} as long as $p\geq 2$ in contrast with the single-task setting
\citep[see Theorem 12]{razafindrakotoApproximateFullConformal2026}.
This quantile is the leading term within the right-most brackets of the upper bound on the thickness.
Regarding the dependence of this upper bound on the output-space dimension $p$, it appears that as long as the factor $\abss{\mathrm{det}\paren{\Gamma^{\frac{1}{2}}}}
\frac{\pi^{\frac{p}{2}}}{\paren{\frac{p}{2}}!}p$ balances the brackets involving the quantile $\widehat{Q}_{\lambda; D}^{\Gamma}(\alpha)$, the whole upper-bound does not necessarily worsen as $p$ increases.

\medskip

{To clarify this observation,
let us first deal with the randomness of the quantile.
The next result states that the population counterpart $f_{\lambda} \in \mathcal{H}$ of the predictor
$\hat{f}_{\lambda; D} \in \mathcal{H}$ (see Definition~\ref{def.predictor}) is well-defined.
}
\begin{lemma}
    \label{lm.reg.risk.minimizer}
    Assume \eqref{asm.conv.loss}, \eqref{asm.lsc.loss}
    and \eqref{asm.lower.bounded.loss} hold true.
    Then, for every regularization parameter $\lambda \in \paren{0, +\infty}$,
    the regularized risk function
    $\mathbf{R}_{\lambda}\paren{\bullet} : \mathcal{H} \to \mathbb{R}$
    given by, for every $f \in \mathcal{H}$,
    \begin{align*}
        \mathbf{R}_{\lambda}\paren{f}
        := \mathbf{R}_{0}\paren{f} + \lambda\normh{f}^2,
    \end{align*}
    admits a unique minimizer over the hypothesis space $\mathcal{H}$.
    Let us note said minimizer $f_{\lambda} \in \mathcal{H}$.
\end{lemma}
\begin{proof}
    One can apply a similar reasoning as Lemma~\ref{lm.predictor.well.defined}.
\end{proof}

{Finally, integrating the deviation of $\hat{f}_{\lambda^{+};D}$ around $f_{\lambda}$,
the next result provides an upper-bound on the \emph{thickness}
(see Eq.~\ref{eq.thickness.stable.CP}) which is tighter for larger values of $\lambda n$.}
\begin{theorem}
    \label{thm.thickness.upper.bound}
    %
    %
    %
    Assume \eqref{asm.conv.loss}, \eqref{asm.lip.loss}, \eqref{asm.lower.bounded.loss},
    \eqref{asm.bounded.expected.loss}, \eqref{asm.bounded.gamma.kernel},
    \eqref{asm.risk.minimizer.attained},\\ \eqref{asm.bounded.gamma.y},
    \eqref{asm.eigenvalue.decay} and \eqref{asm.not.source.condition} hold true.
    For every risk level $\delta \in \paren{0, 1}$,
    with probability greater than $1 - \delta$,
    \begin{align*}
        \mathrm{THK}_{\lambda; \alpha}^{\Gamma}\paren{X_{n+1}}
        \leq
        \frac{
            \rho_{p} \paren{\kappa^{\Gamma}}^{2}
        }{\lambda \paren{n+1}}
        a^{\Gamma}\paren{p}
        b_{\lambda; n; \mathcal{H}}^{\Gamma, \gamma} \paren{p}
        p^{-\frac{\gamma+1}{2}p},
    \end{align*}
    where the terms $a^{\Gamma}\paren{p}$
    and $b_{\lambda; n; \mathcal{H}}^{\Gamma, \gamma} \paren{p}$
    are given by
    \begin{align}
        \label{eq.constant.term}
        a^{\Gamma}\paren{p}
        &
        := 2C_{\Gamma}^{\frac{1}{2}}
        2^{-\frac{\gamma}{4}}
        \pi^{-\frac{\gamma}{4} - \frac{1}{2}}
        e^{-\frac{\gamma}{24p + 2}
        - \frac{1}{6p + 1}}
        \\
        b_{\lambda; n; \mathcal{H}}^{\Gamma, \gamma} \paren{p}
        &
        := p^{\frac{1}{2} - \frac{\gamma}{4}}
        \paren{
            \paren{2 \pi C_{\Gamma}}^{\frac{p}{2}}
            e^{\frac{\gamma+1}{2}p}
        }\notag
        \\
        &
        \quad
        \times
        \paren{
            C_{\mathcal{Y}}\paren{p}
            + \kappa^{\Gamma} C_{\mathcal{H}}
            + \frac{
                \rho_{p} \paren{\kappa^{\Gamma}}^{2}
            }{
                \lambda \sqrt{n}
            }C_{\Gamma}^{\frac{1}{2}}
            \croch{
                \frac{3}{2}
                \frac{\sqrt{n}}{n+1}
                +
                2^{\frac{3}{2}}
                \paren{
                    \frac{\gamma}{\gamma - 1}
                }^{\frac{1}{2}}
                + \sqrt{2 \log \frac{1}{\delta}}
            }
        }^{p-1}\notag.
    \end{align}
Furthermore, these quantities satisfy that
    \begin{align*}
        \lim_{p \to +\infty} a^{\Gamma}\paren{p} < \infty,
    \end{align*}
    and if there exists a constant $C<\infty$ independent of $p$ such that $C_{\mathcal{Y}}\paren{p} \leq C p^{t}$, with $2t - 1 < \gamma$, then
    \begin{align*}
        \lim_{p \to \infty} b_{\lambda; n; \mathcal{H}}^{\Gamma, \gamma} \paren{p}
        p^{-\frac{\gamma+1}{2}p} = 0.
    \end{align*}    
\end{theorem}
The proof is deferred to Appendix~\ref{proof.thickness.upper.bound}.
{For a fixed output dimension $p \geq 2$,
the upper bound improves at the rate of $O\paren{\frac{1}{\lambda n}}$,
if $b_{\lambda; n; \mathcal{H}}^{\Gamma, \gamma} \paren{p}$ is consistent, that is,
if $\frac{1}{\lambda \sqrt{n}} = O(1)$.
This rate was reported by \citet[see Theorem 12]{razafindrakotoApproximateFullConformal2026}
in the single-task setting with $p = 1$.
Conversely, in the multi-task setting with larger values of $p$,
the second factor $a^{\Gamma}\paren{p}$ converges,
and the third factor $b_{\lambda; \theta}^{\Gamma, \gamma} p^{-\frac{\gamma + 1}{2}}$,
converges to 0 if $\gamma > 2t - 1$, where
$\sup_{y \in \mathcal{Y}} \norm{\Gamma^{-\frac{1}{2}}y} \leq C_{\mathcal{Y}}\paren{p} \leq Cp^{t}$.
Since $\gamma$ reflects the strength of the inter-task relatedness,
this means that stronger links between the tasks correspond to larger values of $\gamma$, which allows for larger values of $t$ without worsening the final convergence rate of the thickness.

\subsection{Numerical experiments}
\label{sec.known.cov.numerical.experiments}

Within the subsequent experiments and for the remainder of the present work,
the data generating distribution and the predictor are the ones detailed in the illustration at the end of Section~\ref{sec.known.cov.explicit.computation}.

\bigskip

\notparagraph{Upper bound on the \emph{thickness} and sample size}
The present experiments aim at illustrating how the upper-bound derived in Theorem~\ref{thm.thickness.upper.bound} depends on the sample size $n$ (with $\lambda = 10^{-2}$).
To be more specific, for sample sizes on a logarithmic grid,
the empirical upper-bound $\leb{\widetilde{C}_{\lambda; \alpha}^{\Gamma, \mathrm{up}}\paren{X_{n+1}} \setminus \widetilde{C}_{\lambda; \alpha}^{\Gamma, \lo}\paren{X_{n+1}}}$ 
on the \emph{thickness} $\thicc{\Gamma}$ is computed from 20 repetitions.
%
%
Figure~\ref{fig.thickness.stableCP} displays the boxplots corresponding to the values of this upper bound recorded for each sample size $n$.
Then, a linear regression is performed to estimate the slope of the straight line (corresponding to the exponent of $n$).
%
%
The dashed-red-line reports the resulting straight line. 
\begin{figure}[H]
    \centering
    \includegraphics[width=0.8\textwidth]{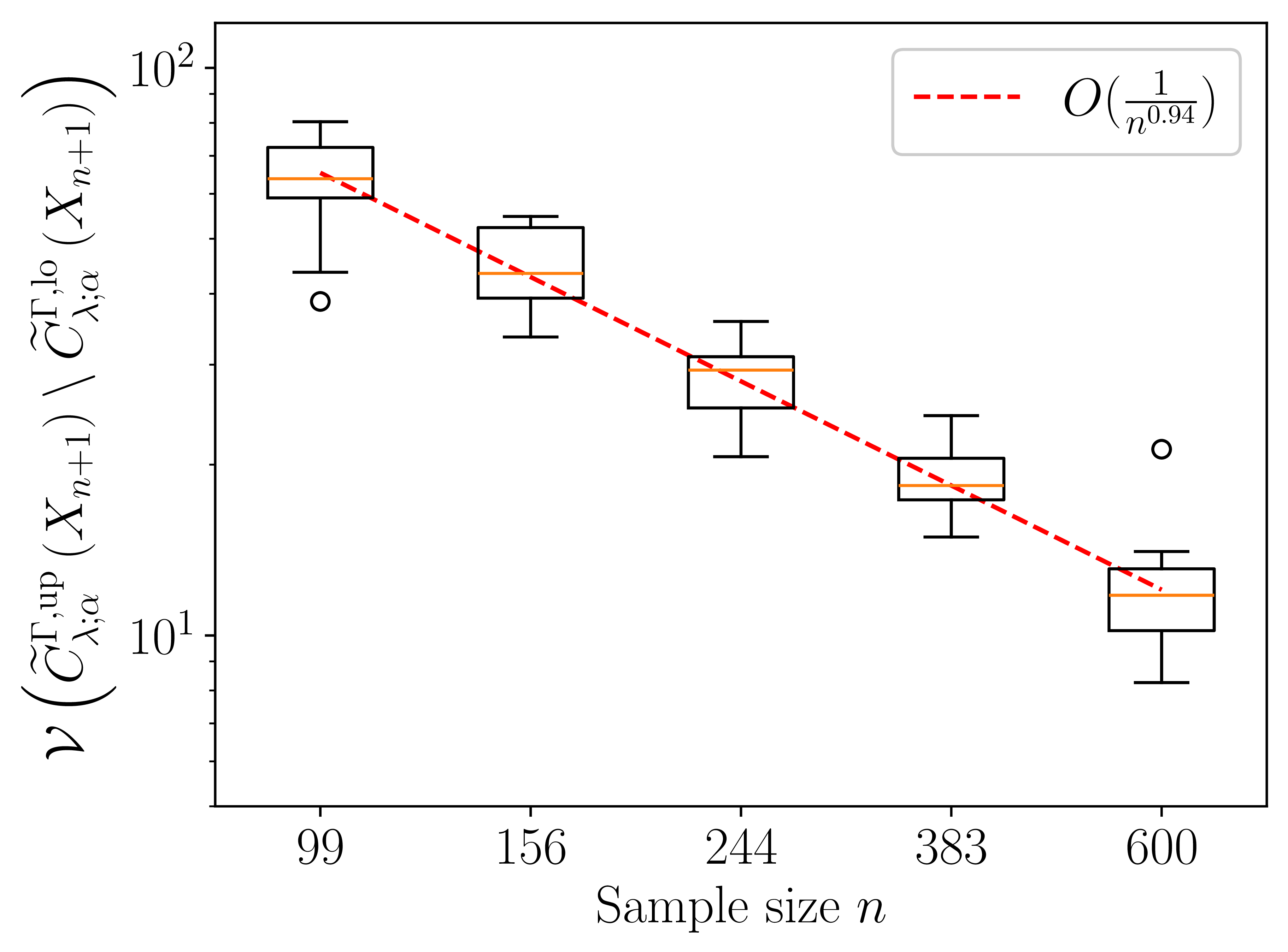}
    \caption{
        Evolution of the computable empirical upper-bound $\leb{\widetilde{C}_{\lambda; \alpha}^{\Gamma, \mathrm{up}}\paren{X_{n+1}} \setminus \widetilde{C}_{\lambda; \alpha}^{\Gamma, \lo}\paren{X_{n+1}}}$
        for the \emph{thickness} $\thicc{\Gamma}$ across 20 repetitions for $\alpha = 0.1$ for fixed $\lambda = 10^{-2}$.
    }
    \label{fig.thickness.stableCP}
\end{figure}

On average, the empirical upper-bound
$\leb{\widetilde{C}_{\lambda; \alpha}^{\Gamma, \mathrm{up}}\paren{X_{n+1}} \setminus \widetilde{C}_{\lambda; \alpha}^{\Gamma, \lo}\paren{X_{n+1}}}$ gets smaller as the training sample size increases.
The estimated exponent is close to the theoretical one which is 1 for a fixed $\lambda$ (see Theorem~\ref{thm.thickness.upper.bound}).  
This suggests that the theoretical upper bound derived in Theorem~\ref{thm.thickness.upper.bound} on the thickness is reasonably tight. Let us emphasize that this already holds over a range of moderate values of $n$. (We refer interested readers to Appendix~\ref{sec.evolution.thickness.fixed} for assessing the behavior of the upper bound with $\lambda$ of the order of $O(\frac{1}{\sqrt{n}})$.)

\bigskip

\notparagraph{Comparison between \textbf{StableCP} and \textbf{SplitCP}}
It is also desirable to draw a comparison between the upper \textbf{StableCP}-region (advocated in the present paper as a means to approximate the \textbf{FullCP}-region) and the \textbf{SplitCP}-region (see Appendix~\ref{sec.split.cp} for a detailed description). 
The volume of the considered region is divided by the one of an \textbf{OracleCP}-region which exploits the knowledge of $Y_{n+1}$ (see Appendix~\ref{sec.oracle.cp} for details).
For this comparison, the volume is computed from the sample size $n = 500$.

In practice, the value of the regularization parameter $\lambda$ from Definition~\ref{def.predictor} is chosen by minimizing a
penalized criterion leading to a reliable predictor.
Therefore each procedure incorporates its own regularization parameter value, which varies with the training sample size.
To be more specific, \textbf{SplitCP} incorporates $\widehat{\lambda}_{n_{\mathrm{train}}}$, while \textbf{StableCP} incorporates $\widehat{\lambda}_{n}$.
For the coverage guarantee to hold, these regularization parameters are computed from the same procedure used to choose: $(1)$ $\widehat{\lambda}_{n}$ from a set $D^{\prime}$ of cardinality $n$, and $(2)$ $\widehat{\lambda}_{n_{\mathrm{train}}}$ from a subset $D^{\prime}_{\mathrm{train}} \subset D^{\prime}$ of cardinality $n_{\mathrm{train}}$.

Let $\Lambda$ denote a set of candidate regularization parameter values.
A regularization parameter $\widehat{\lambda}_{n}$ is chosen
as the minimizer of the subsequent penalized criterion that is,
\begin{align*}
    \widehat{\lambda}_{n}
    \in \argmin_{\lambda \in \Lambda}
    \brac{
        \widehat{\mathbf{R}}_{0, D^{\prime}}\paren{\hat{f}_{\lambda; D^{\prime}}}
        + \widehat{\mathrm{pen}}_{D^{\prime}}\paren{\lambda}
    }
\end{align*}
involving $D^{\prime}$, independent of
$\paren{X_{1}, Y_{1}}, \ldots, \paren{X_{n+1}, Y_{n+1}}$,
where for every $\lambda \in \Lambda$,
the penalization term $\mathrm{pen}\paren{\lambda}$
is set to be
\begin{align*}
    \widehat{\mathrm{pen}}_{D^{\prime}}\paren{\lambda}
    := \frac{\rho_{p}^2}{\lambda n}
    \croch{
        \frac{1}{n}\sum_{i=1}^{n}
        \norm{K\paren{X_i^{\prime}, X_i^{\prime}}}_{\mathrm{op}}^{\frac{1}{2}}
    }.
\end{align*}
The idea is that of the structural risk minimization (SRM) procedure \citep[see Section 4.6.1]{bachLearningTheoryFirst}, that is, minimizing an upper-bound on the risk.
Compared to \citep[see Section 4.6.1]{bachLearningTheoryFirst}, the high probability control have been dropped since
the objective is only to get a $\lambda_{n}$
with about the right order of magnitude.
This also corresponds to a control of the true risk in expectation (see Proposition~\ref{prop.generalization.upper.bound})
}, as established by the following result adapted from an intermediate result in \citet[see the proof of Theorem 12]{bousquetStabilityGeneralization2002}.

\begin{proposition}
    \label{prop.generalization.upper.bound}
    Under \eqref{asm.conv.loss},
    \eqref{asm.lip.loss} and \eqref{asm.lower.bounded.loss},
    \begin{align*}
        \mathbb{E}\croch{
            \mathbf{R}_{0}\paren{\hat{f}_{\lambda; D}}
            - \widehat{\mathbf{R}}_{0; D}\paren{\hat{f}_{\lambda; D}}
        } \leq
        \frac{
            \rho_{p}^2
        }{\lambda n}\mathbb{E}\croch{\norm{K\paren{X, X}}_{\mathrm{op}}^{\frac{1}{2}}}.
    \end{align*}
\end{proposition}
The proof is deferred to Appendix~\ref{proof.generalization.upper.bound}.

\medskip

By contrast, the ideal regularization parameter $\widehat{\lambda}^{\mathrm{ideal}}$ would be the one for which the resulting predictor minimizes the risk that is,
\begin{align*}
    \widehat{\lambda}_{\mathrm{ideal}} := \argmin_{\lambda \in \Lambda} \mathbf{R}_{0}\paren{\hat{f}_{\lambda; D}}.
\end{align*}
However, since $\mathbf{R}_{0}\paren{\hat{f}_{\lambda; D}}$ is intractable,
the above penalized criterion provides a computable proxy
to $\mathbf{R}_{0}\paren{\hat{f}_{\lambda; D}}$.
\begin{figure}[H]
    \centering
    \includegraphics[width=\textwidth]{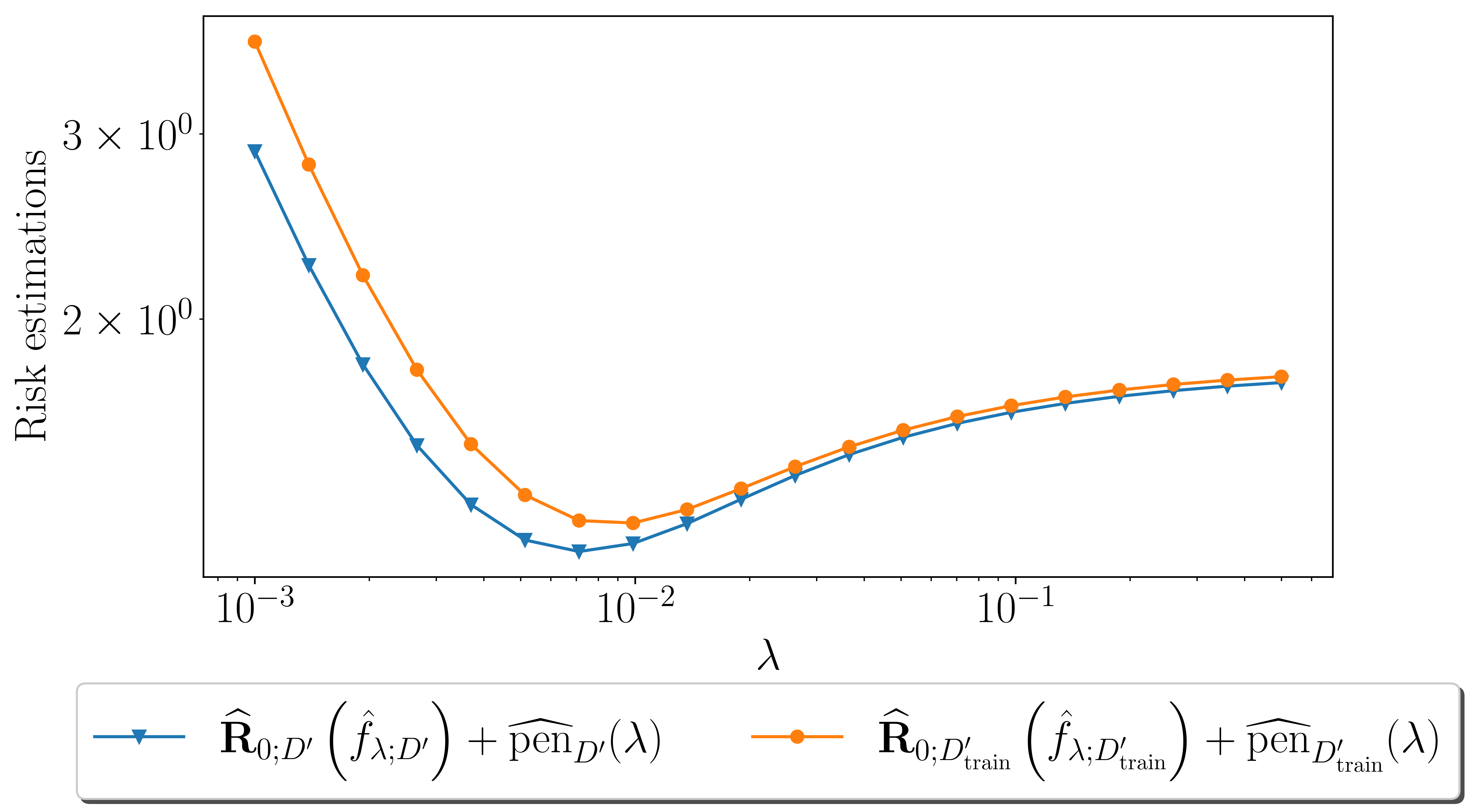}
    \caption{
        Evolution of the penalized criterion for $D^{\prime}$ and $D^{\prime}_{\mathrm{train}}$
        with $\abss{D^{\prime}} = n = 500$ and $\abss{D^{\prime}_{\mathrm{train}}} = 375$.
    }
    \label{fig.risk estimation}
\end{figure}
Figure~\ref{fig.risk estimation} displays the values of the estimated risks as a function of $\lambda$. Since the lower the better, one observes that both curves smoothly varies as functions of $\lambda$. The blue one (corresponding to the upper \textbf{StableCP}-region) is minimized at $\hat \lambda_n \approx 0.007$,
while the orange one (corresponding to the \textbf{SplitCP}-region) is minimized at $\hat \lambda_{n_{\mathrm{train}}} \approx 0.009$.

{
Now for each repetition and each procedure (\textbf{SplitCP} and \textbf{StableCP}),
the ratio between the volume of the resulting prediction-region
with that of the \textbf{OracleCP}-region (see Appendix~\ref{sec.oracle.cp}) is computed.
Then, each boxplot displays the values of this ratio across 100 repetitions.
Additionally, $\overline{cov}$ reports the empirical coverage-probability that is,  the empirical proportion of output-value that are contained within the prediction-regions across these 100 repetitions.
Finally, $\overline{T}$ measures the average computation time (relative to that of the \textbf{OracleCP}-region).
}

\begin{figure}[H]
    \centering
    \includegraphics[width=\textwidth]{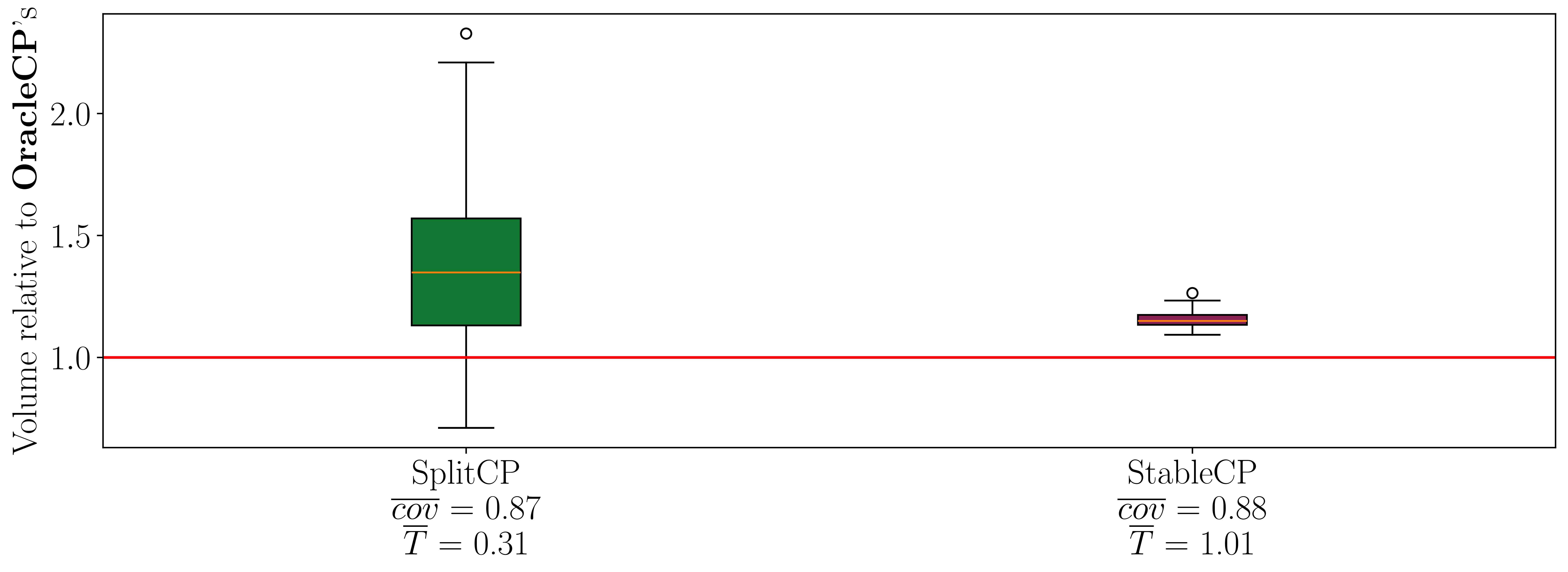}
    \caption{
        Comparison between the \textbf{StableCP}-regions (in red)
        and the \textbf{SplitCP}-regions (in green) in terms of
        volume relative to that of the \textbf{OracleCP}-regions
        (see Appendix~\ref{sec.oracle.cp}) for $\alpha = 0.1$
        with $\abss{D} = n = 500$ and $\abss{D_{\mathrm{train}}} = 375$.
    }
    \label{fig.volume.stableCP}
\end{figure}

{%
Figure~\ref{fig.volume.stableCP} displays the boxplots corresponding to the \textbf{SplitCP}-region (left-green) and the (right-red) \textbf{StableCP}-region.
On average the \textbf{StableCP}-region is smaller  than
the \textbf{SplitCP}-region in terms of volume relative to that of the \textbf{OracleCP}-region.
The variability of the volumes of the \textbf{SplitCP}-regions is also
higher than that of the \textbf{StableCP}-regions.
In practice, this higher variability implies that
the volume of the \textbf{SplitCP}-region can be much larger
than that of the \textbf{StableCP}-region, which one would like to avoid.

In terms of empirical coverage-probability, despite its smaller volume (on average),
the \textbf{StableCP}-region empirically presents a larger empirical coverage-probability
compared to the \textbf{SplitCP}-region.
Regarding the computation time, the reported values indicate that on average,
\textbf{StableCP}-region has almost the same cost as  \textbf{OracleCP}-region, which also relies on considering all the $n+1$ observations. On the contrary, the \textbf{SplitCP}-region is at least 3 times less expensive than the \textbf{StableCP}-one, which is due to the smaller cardinality of the training set it is based on.
}

\section{Estimated inter-task covariance-matrix}
\label{sec.estimated.inter.task.covariance}
The present section addresses an instance of
the approximation scheme introduced in Definition~\ref{def.approximate.fullCP.region},
called \textbf{G-EllipsoidCP}, where the inter-task covariance-matrix is unknown and has to be estimated.
More precisely,
Section~\ref{sec.global.covariance.ncms.approximations}
introduces the inter-task covariance-estimator along with the resulting non-conformity scores,
and derives the corresponding upper- and lower-approximate non-conformity scores.
Then, Section~\ref{sec.global.covariance.explicit.computation}
explains the guarantee enjoyed by the upper \textbf{G-EllipsoidCP}-region, while providing its explicit expression. Numerical experiments also highlight the good empirical coverage guarantee the upper \textbf{G-EllipsoidCP}-region also enjoys on a synthetic data set.
Finally, Section~\ref{sec.global.covariance.numerical.experiments} numerically explores the the tightness of the upper \textbf{G-EllipsoidCP}-region
as a function of the training sample size, and
compares its volume with that of the \textbf{SplitCP}-region on a synthetic data set (see Appendix~\ref{sec.split.cp}).

\subsection{Non-conformity scores and approximations}
\label{sec.global.covariance.ncms.approximations}
Let us first introduce the inter-task covariance-matrix estimator
and the corresponding non-conformity score.

\bigskip
\notparagraph{Inter-task covariance-matrix estimator}
Let $a \in \paren{0, 1}$ denote a regularization parameter,
and $y \in \mathcal{Y}$ a test output-value.
Then, $\widehat{\Gamma}_{a; D^{y}} \in \mathbb{R}^{p \times p}$
refers to the $a$-regularized inter-task covariance-matrix based on $D^{y}$, which is given by
\begin{align}
    \label{eq.estimated.output.covariance}
    \widehat{\Gamma}_{a; D^{y}}
    := \frac{1}{n+1} y y^{T}
    + \frac{1}{n+1}
    \sum_{i=1}^{n}
    Y_{i} Y_{i}^{T}
    - \widehat{\mu}_{D^{y}}
    \widehat{\mu}_{D^{y}}^{T}
    + a I_{p},
\end{align}
where $\widehat{\mu}_{D^{y}}$ stands for the average
of the output-vectors in $D^{y}$ that is,
\begin{align*}
    \widehat{\mu}_{D^{y}}
    := \frac{1}{n+1}y
    + \frac{1}{n+1}\sum_{i=1}^{n}Y_i.
\end{align*}
Moreover,
let $\widehat{\Gamma}_{a^{+}} \in \mathrm{Sym}_{p}^{++}\paren{\mathbb{R}}$
denote the matrix given by,
\begin{align}
    \label{eq.global.covariance}
    \widehat{\Gamma}_{a^{+}} := \widehat{\Gamma}_{\frac{a\paren{n+1}}{n}; D}.
\end{align}
In the \textbf{SplitCP} setting, such an estimated covariance matrix was
considered by \citet{messoudiEllipsoidalConformalInference2022} and later by \citet{braunMultivariateStandardizedResiduals2026}.
Building up on that, the present covariance-matrix estimator makes sense for two reasons: $(1)$ up to the regularization term, it is the maximum likelihood estimator under a Gaussian distributional assumption, and $(2)$ it enjoys a simple closed-form expression (Eq.~\ref{eq.estimated.output.covariance}).
However other covariance-matrix estimators could be considered as well although this is left for a future work.

\notparagraph{Estimated Mahalanobis non-conformity score}
For every $\paren{x, u} \in \mathcal{X} \times \mathcal{Y}$,
let $S_{\lambda; D^{y}}^{\widehat{\Gamma}_{a}}\paren{x, u}$ denote
the estimated Mahalanobis non-conformity score
of $\paren{x, u}$ with respect to the data set $D^{y}$. It is defined as
\begin{align*}
    S_{\lambda; D^{y}}^{\widehat{\Gamma}_{a}}\paren{x, u}
    :=
    \norm{
        \widehat{\Gamma}_{a; D^{y}}^{-\frac{1}{2}}\paren{
            u - \hat{f}_{\lambda; D^{y}}(x)
        }
    },
\end{align*}
where $\widehat{\Gamma}_{a; D^{y}}$ is the covariance-matrix given by Eq.~\eqref{eq.estimated.output.covariance}.

Harnessing the algorithmic stability-bounds and
the Sherman-Woodbury Morrison formula,
the next result details the expressions
of the upper- and lower-approximate non-conformity scores induced by this choice of score.
\begin{lemma}[Upper- and lower-non-conformity score approximations]
    \label{lm.stable.score.maha}
    Assume
    \eqref{asm.conv.loss},
    \eqref{asm.lip.loss} and
    \eqref{asm.lower.bounded.loss}
    hold true.
    Then, for every $\paren{x, u} \in \mathcal{X} \times \mathcal{Y}$,
    \begin{align*}
        \widetilde{S}_{\lambda; D^{y}}^{\widehat{\Gamma}_{a}, \lo}
        \paren{x, u}
        \leq
        S_{\lambda; D^{y}}^{\widehat{\Gamma}_{a}}
        \paren{x, u}
        \leq
        \widetilde{S}_{\lambda; D^{y}}^{\widehat{\Gamma}_{a}, \up}
        \paren{x, u},
    \end{align*}
    where $\widetilde{S}_{\lambda; D^{y}}^{\widehat{\Gamma}_{a}, \lo}
    \paren{x, u}$ designates the lower-approximate non-conformity score,
    defined as
    \begin{align*}
        \widetilde{S}_{\lambda; D^{y}}^{\widehat{\Gamma}_{a}, \lo}
        \paren{x, u}
        := \frac{
            \paren{
                \frac{n+1}{n}
            }^{\frac{1}{2}}
            \croch{
                \norm{
                    \widehat{\Gamma}_{a^{+}}^{-\frac{1}{2}}
                    \paren{u - \hat{f}_{\lambda^{+}; D}(x)}
                } - \widehat{\tau}_{\lambda}^{\widehat{\Gamma}_{a^{+}}}(x)
            }
        }{
            \paren{
                1 + \frac{1}{n+1}
                \paren{
                    \norm{
                        \widehat{\Gamma}_{a^{+}}^{-\frac{1}{2}}
                        \paren{y - \hat{f}_{\lambda^{+}; D}\paren{X_{n+1}}}
                    } + \widehat{t}_{\lambda; a}
                }^{2}
            }^{\frac{1}{2}}
        },
    \end{align*}
    and $\widetilde{S}_{\lambda; D^{y}}^{\widehat{\Gamma}_{a}, \up}
    \paren{x, u}$, the upper-approximate non-conformity score,
    defined as
    \begin{align}
        \label{eq.upper.global.score}
        \widetilde{S}_{\lambda; D^{y}}^{\widehat{\Gamma}_{a}, \up}
        \paren{x, u} :=
        \paren{
            \frac{n+1}{n}
        }^{\frac{1}{2}}
        \croch{
            \norm{
                \widehat{\Gamma}_{a^{+}}^{-\frac{1}{2}}
                \paren{u - \hat{f}_{\lambda^{+}; D}(x)}
            } + \widehat{\tau}_{\lambda}^{\widehat{\Gamma}_{a^{+}}}(x)
        },
    \end{align}
    and $\widehat{\tau}_{\lambda}^{\widehat{\Gamma}_{a^{+}}}(x)$
    the stability-bound, defined as
    \begin{align}
        \label{eq.global.score.stability.bound}
        \widehat{\tau}_{\lambda}^{\widehat{\Gamma}_{a^{+}}}(x)
        := \norm{
            \widehat{\Gamma}_{a^{+}}^{-\frac{1}{2}}
            K\paren{x, x}
            \widehat{\Gamma}_{a^{+}}^{-\frac{1}{2}}
        }_{\mathrm{op}}^{\frac{1}{2}}
        \frac{\rho_p\norm{K\paren{X_{n+1}, X_{n+1}}}_{\mathrm{op}}^{\frac{1}{2}}
        }{2\lambda (n+1)},
    \end{align}
    and the deviation $\widehat{t}_{\lambda; a}$ is given by
    \begin{align}
        \label{eq.correction}
        \widehat{t}_{\lambda; a}
        := \norm{
            \widehat{\Gamma}_{a^{+}}^{-\frac{1}{2}} \paren{
                \hat{f}_{\lambda^{+}; D}\paren{X_{n+1}} - \widehat{\mu}_{D}
            }
        }.
    \end{align}
\end{lemma}
The proof is deferred to Appendix~\ref{proof.stable.score.maha}.
Let us emphasize that, one main difference compared to similar results from \citet{ndiayeStableConformalPrediction2022,leeLeaveOneOutStableConformal2025,razafindrakotoApproximateFullConformal2026}, the lower-non-conformity score approximation involves not only an additive correction, but also a multiplicative one as well as a denominator term depending on $y$.

Concerning the additive correction in Eq.~\eqref{eq.global.score.stability.bound}, the first term involves
the inter-task covariance-matrix estimator combined with the influence of the matrix-valued $K(x,x)$. This again emphasizes the importance
of specifying a matrix-value kernel that is able to balance the inverse of the estimated covariance-matrix (estimator).
Provided that the inter-task covariance-matrix converges,
similarly to the ``known inter-task covariance-matrix'' setting,
with a matrix-value kernel such that $K\paren{\bullet, \bullet} = k\paren{\bullet, \bullet} \Gamma$ (with $k\paren{\bullet, \bullet}$ is a scalar-valued bounded kernel), then the score approximations improve at a rate of $O\paren{\frac{1}{\lambda n}}$.

\subsection{Explicit approximate \textbf{G-EllipsoidCP}-regions computation}
\label{sec.global.covariance.explicit.computation}
From the above upper and lower non-conformity scores from Lemma~\ref{lm.stable.score.maha}, the resulting
upper- and lower-approximate FullCP are formulated as follows.

\bigskip
\notparagraph{\textbf{G-EllipsoidCP}-region}
Let $\widetilde{C}_{\lambda; \alpha}^{\widehat{\Gamma}_{a}, \up} \paren{X_{n+1}}$
denote the upper \textbf{G-EllipsoidCP}-region
with the estimated Mahalanobis non-conformity, given by
\begin{align}
    \label{def.approx.fcpr.global.maha}
    \widetilde{C}_{\lambda; \alpha}^{\widehat{\Gamma}_{a}, \up} \paren{X_{n+1}}
    := \brac{
    y \in \mathcal{Y}:
    \frac{
    1 + \sum_{i=1}^{n}
    \mathbbm{1}
    \brac{
    \widetilde{S}_{\lambda; D^{y}}^{\widehat{\Gamma}_{a}, \up}\paren{X_i, Y_i}
    \geq \widetilde{S}_{\lambda; D^{y}}^{\widehat{\Gamma}_{a}, \lo}\paren{X_{n+1}, y}
    }
    }{n+1} > \alpha
    }.
\end{align}
Conversely,
let $\widehat{C}_{\lambda; \alpha}^{\widehat{\Gamma}_{a}, \full}\paren{X_{n+1}}$
designate the \textbf{FullCP}-region, given by
\begin{align*}
    \widehat{C}_{\lambda; \alpha}^{\widehat{\Gamma}_{a}, \full}\paren{X_{n+1}}
    := \brac{
        y \in \mathcal{Y}:
        \frac{1
        + \sum_{i=1}^{n}
        \mathbbm{1}\brac{
            S_{\lambda; D^{y}}^{\widehat{\Gamma}_{a}}\paren{X_i, Y_i}
            \geq S_{\lambda; D^{y}}^{\widehat{\Gamma}_{a}}\paren{X_{n+1}, y}
        }}{n+1}>\alpha
    }.
\end{align*}

As an instance of the approximate scheme in Definition~\ref{def.approximate.fullCP.region},
the next result establishes the coverage guarantee that the upper \textbf{G-EllipsoidCP}-region enjoys.

\begin{corollary}
    \label{cor.coverage.ellipsoid}
    Assume \eqref{asm.conv.loss}, \eqref{asm.lip.loss} and \eqref{asm.lower.bounded.loss} hold true.
    Then, for any control-level $\alpha \in \left[\frac{1}{n+1}, 1\right)$,
    the upper \textbf{G-EllipsoidCP}-region $\widetilde{C}_{\lambda; \alpha}^{\widehat{\Gamma}_{a}, \up} \paren{X_{n+1}}$
    contains the \textbf{FullCP}-region $\widehat{C}_{\lambda; \alpha}^{\widehat{\Gamma}_{a}, \full} \paren{X_{n+1}}$, that is
    $%
        \widehat{C}_{\lambda; \alpha}^{\widehat{\Gamma}_{a}, \full} \paren{X_{n+1}}
        \subseteq
        \widetilde{C}_{\lambda; \alpha}^{\widehat{\Gamma}_{a}, \up}\paren{X_{n+1}}
    $.
    As a result, the coverage-probability of $\widetilde{C}_{\lambda; \alpha}^{\widehat{\Gamma}_{a}} \paren{X_{n+1}}$
    is bounded from below,
    \begin{align*}
        \mathbb{P}\croch{
        Y_{n+1}
        \in
        \widetilde{C}_{\lambda; \alpha}^{\widehat{\Gamma}_{a}, \up}\paren{X_{n+1}}
        }
        \geq \mathbb{P}\croch{
        Y_{n+1}
        \in
        \widehat{C}_{\lambda; \alpha}^{\widehat{\Gamma}_{a}, \full}\paren{X_{n+1}}
        }
        \geq 1 - \alpha,
    \end{align*}
    making $\widetilde{C}_{\lambda; \alpha}^{\widehat{\Gamma}_{a}, \up}\paren{X_{n+1}}$ a confidence prediction-region.
\end{corollary}
\begin{proof}
    Direct application of Lemma~\ref{lm.stable.score.maha}
    and Theorem~\ref{thm.coverage.approximate.region}.
\end{proof}

Owing to the choice of non-conformity scores in Lemma~\ref{lm.stable.score.maha},
the next proposition details the closed-form expression of the upper \textbf{G-EllipsoidCP}-region.
\begin{proposition}
    \label{prop.expression.upper.GlobalEllipsoidCP.region}
    For any control-level $\alpha \in \left[\frac{1}{n+1}, 1\right)$,
    the upper \textbf{G-EllipsoidCP}-region
    $\widetilde{C}_{\lambda; \alpha}^{\widehat{\Gamma}_{a}, \up}\paren{X_{n+1}}$
    is the region enclosed by a
    $\widehat{\Gamma}_{a^{+}}^{-1}$-ellipsoid
    centred at $\hat{f}_{\lambda^{+}; D}\paren{X_{n+1}}$
    with a radius of $\widetilde{Q}_{\lambda; D^{+}}^{\widehat{\Gamma}_{a}, \up}(\alpha)$, that is,
    \begin{align*}
        \widetilde{C}_{\lambda; \alpha}^{\widehat{\Gamma}_{a}, \up}\paren{X_{n+1}}
        =
        \brac{
            y \in \mathcal{Y}:
                \norm{
                \widehat{\Gamma}_{a^{+}}^{-\frac{1}{2}}
                \paren{y - \hat{f}_{\lambda^{+}; D}\paren{X_{n+1}}}
            } \leq \widetilde{Q}_{\lambda; D^{+}}^{\widehat{\Gamma}_{a}, \up}(\alpha)
        },
    \end{align*}
    where the radius $\widetilde{Q}_{\lambda; D^{+}}^{\widehat{\Gamma}_{a}, \up}(\alpha)$ is given by
    \begin{align*}
        \widetilde{Q}_{\lambda; D^{+}}^{\widehat{\Gamma}_{a}, \up}(\alpha)
        := \paren{
            L_{
                \frac{1}{n+1}
            }^{\widehat{t}_{\lambda; a}}
        }^{-1}
        \paren{
            \widehat{Q}_{\lambda; D^{+}}^{\widehat{\Gamma}_{a}, \up}(\alpha);
            \widehat{\tau}_{\lambda}^{\widehat{\Gamma}_{a^{+}}}\paren{X_{n+1}}
        },
    \end{align*}
    (see Eq.~\ref{eq.global.covariance}
    for $\widehat{\Gamma}_{a^{+}}$,
    Eq.~\ref{eq.global.score.stability.bound}
    for $\widehat{\tau}_{\lambda}^{\widehat{\Gamma}_{a^{+}}}\paren{\bullet}$,
    Eq.~\ref{eq.correction} for $\widehat{t}_{\lambda; a}$,
    Eq.~\ref{eq.lower.transformation.function} for $L_{w}^{t}\paren{\bullet; \tau}$)
    provided that the quantile value $\widehat{Q}_{\lambda; D^{+}}^{\widehat{\Gamma}_{a}, \up}(\alpha)$,
    given below, is smaller than $\sqrt{n+1}$,
    that is,
    \begin{align}
        \label{eq.upper.quantile.global.covariance}
        \widehat{Q}_{\lambda; D^{+}}^{\widehat{\Gamma}_{a}, \up}(\alpha)
        := \croch{
            \norm{
                \widehat{\Gamma}_{a^{+}}^{-\frac{1}{2}}
                \paren{
                    Y_{\paren{i_{n, \alpha}^{n}}}
                    - \hat{f}_{\lambda^{+}; D}\paren{X_{\paren{i_{n, \alpha}^{n}}}}
                }
            } + \widehat{\tau}_{\lambda}^{\widehat{\Gamma}_{a^{+}}}\paren{X_{\paren{i_{n, \alpha}^{n}}}}
        } < \sqrt{n+1}
    \end{align}
    (see Eq.~\ref{eq.index.alpha} for $i_{n, \alpha}^{n}$).
    Otherwise, $\widetilde{C}_{\lambda; \alpha}^{\widehat{\Gamma}_{a}}\paren{X_{n+1}} = \mathcal{Y}$.
\end{proposition}
The proof is deferred to Appendix~\ref{proof.expression.upper.GlobalEllipsoidCP.region}.
The shape of the upper \textbf{G-EllipsoidCP}
encodes the influence of the estimated inter-task connections.
Moreover, the upper-correction function involved in the radius of the ellipsoid, $\paren{
    L_{
        \frac{1}{n+1}
    }^{\widehat{t}_{\lambda; a}}
}^{-1}
\paren{
    \bullet;
    \widehat{\tau}_{\lambda}^{\widehat{\Gamma}_{a^{+}}}\paren{X_{n+1}}
}$ is an increasing function, and becomes smaller as $\widehat{\tau}_{\lambda}^{\widehat{\Gamma}_{a^{+}}}\paren{X_{n+1}}$ decreases
by Lemma~\ref{lm.variation.lower.transformation.function}.
This suggests that for larger values of $\lambda n$ (see Eq.~\eqref{eq.global.score.stability.bound}), the radius of the ellipsoid would be less thickened.

\bigskip

\notparagraph{Evaluating the empirical coverage-probability}
On a synthetic data set, the empirical coverage guarantee enjoyed by the upper \textbf{G-EllipsoidCP}-region is displayed on Figure~\ref{fig.coverage.GlobalEllipsoidCP}.
All experimental parameters are the same as the ones detailed at the end of Section~\ref{sec.known.cov.explicit.computation}.
\begin{figure}[H]
    \centering
    \includegraphics[width=0.80\textwidth]{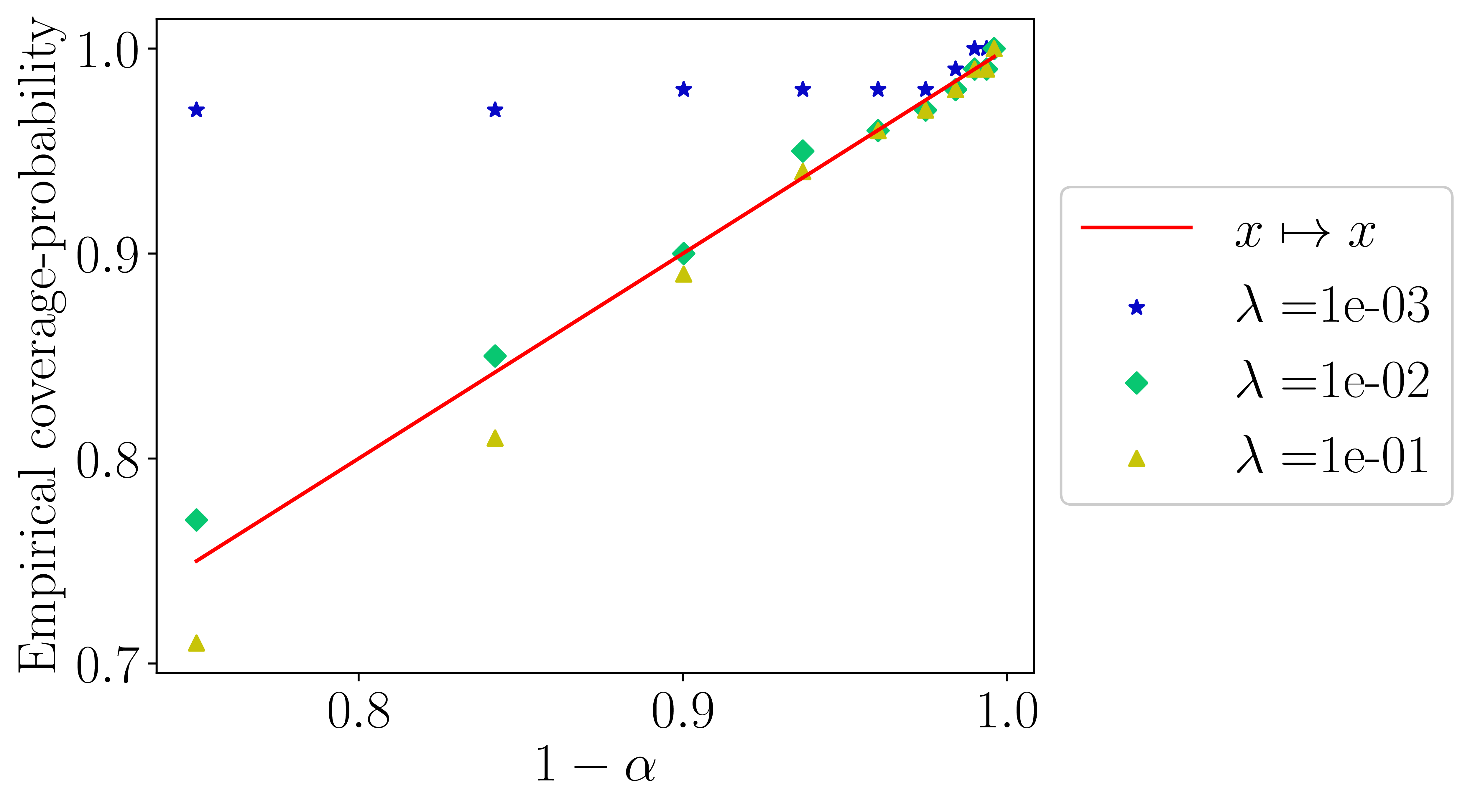}
    \caption{
        Evolution of the empirical coverage-probability of the upper \textbf{G-EllipsoidCP}-region,
        for $\lambda \in \brac{10^{-3}, 10^{-2}, 10^{-1}}$
        and for $\alpha$ within a grid.
    }
    \label{fig.coverage.GlobalEllipsoidCP}
\end{figure}
The red-line represents the desired control-level $1 - \alpha$
over which the theoretical coverage-probability is guaranteed to remain (see Corollary~\ref{cor.coverage.ellipsoid}).
The empirical coverage-probabilities are mostly above the red-line, which illustrates that the prediction-regions does empirically enjoy the desired control property.
Nevertheless, a few values are located under the red-line,
with larger deviations for lower values of $1 - \alpha$.
This results the randomness induced by the empirical coverage-probability. The empirical coverage-probability is a Binomial random variable $\mathrm{Bin}\paren{100, 1 - \alpha}$ divided by $100$.
Since its variance is equal to $\frac{\alpha\paren{1-\alpha}}{100}$,
its variance becomes higher for $\alpha$ closer to $\frac{1}{2}$.
Similarly to the ``known inter-task covariance-matrix'' setting (see Section~\ref{sec.known.cov.explicit.computation}), smaller values of $\lambda$ yields larger values empirical coverage-probabilities due
to higher thickening of the radius of the prediction-regions.

\subsection{Numerical experiments}
\label{sec.global.covariance.numerical.experiments}
The next empirical experiments aim at illustrating
the performances of the upper \textbf{G-EllipsoidCP}-region
in terms of \emph{thickness} as a function the sample size $n$, and through a comparison to the \textbf{SplitCP}-region in terms of their respective volumes.

As already explained in Section~\ref{sec.known.cov.thickness.upper.bound},
a key ingredient for deriving a computable proxy for the \emph{thickness} is the lower \textbf{G-EllipsoidCP}-region.
As such, let $\widetilde{C}_{\lambda; \alpha}^{\widehat{\Gamma}_{a}, \lo}\paren{X_{n+1}}$ denote the lower \textbf{G-EllipsoidCP}-region given by
\begin{align*}
    \widetilde{C}_{\lambda; \alpha}^{\widehat{\Gamma}_{a}, \lo}\paren{X_{n+1}}
    := \brac{
        y \in \mathcal{Y}:
        \frac{1
        + \sum_{i=1}^{n}
        \mathbbm{1}\brac{
            \widetilde{S}_{\lambda; D^{y}}^{\widehat{\Gamma}_{a}, \lo}\paren{X_i, Y_i}
            \geq \widetilde{S}_{\lambda; D^{y}}^{\widehat{\Gamma}_{a}, \up}\paren{X_{n+1}, y}
        }}{n+1}>\alpha
    }.
\end{align*}
The next result provides the closed-form expression
for the lower \textbf{G-EllipsoidCP}-region.
\begin{proposition}
    \label{prop.expresion.lower.GlobalEllipsoidCP.region}
    For any control-level $\alpha \in \left[\frac{1}{n+1}, 1\right)$,
    the lower \textbf{G-EllipsoidCP}-region
    $\widetilde{C}_{\lambda; \alpha}^{\widehat{\Gamma}_{a}, \lo}\paren{X_{n+1}}$
    is the region enclosed by a
    $\widehat{\Gamma}_{a^{+}}^{-\frac{1}{2}}$-ellipsoid
    centred at $\hat{f}_{\lambda^{+}; D}\paren{X_{n+1}}$
    with a radius of $\widetilde{Q}_{\lambda; D^{+}}^{\widehat{\Gamma}_{a}, \lo}(\alpha)$ given by
    \begin{align*}
        \widetilde{C}_{\lambda; \alpha}^{\widehat{\Gamma}_{a}, \lo}\paren{X_{n+1}}
        = \brac{
            y \in \mathcal{Y}:
            \norm{
                \widehat{\Gamma}_{a^{+}}^{-\frac{1}{2}}
                \paren{y - \hat{f}_{\lambda^{+}; D}\paren{X_{n+1}}}
            } \leq \widetilde{Q}_{\lambda; D^{+}}^{\widehat{\Gamma}_{a}, \lo}(\alpha)
        },
    \end{align*}
    where the radius is
    \begin{align*}
        \widetilde{Q}_{\lambda; D^{+}}^{\widehat{\Gamma}_{a}, \lo}(\alpha)
        := 
        \paren{
            U_{\frac{1}{n+1}; 1}^{\widehat{t}_{\lambda; a}}
        }^{-1}\paren{
            \widehat{Q}_{\lambda; D^{+}}^{\widehat{\Gamma}_{a}, \lo}(\alpha);
            \widehat{\tau}_{\lambda}^{\widehat{\Gamma}_{a^{+}}}\paren{X_{n+1}}
        },
    \end{align*}
    (see Eq.~\ref{eq.global.covariance}
    for $\widehat{\Gamma}_{a^{+}}$,
    Eq.~\ref{eq.global.score.stability.bound}
    for $\widehat{\tau}_{\lambda}^{\widehat{\Gamma}_{a^{+}}}\paren{\bullet}$,
    Eq.~\ref{eq.correction} for $\widehat{t}_{\lambda; a}$,
    and Eq.~\ref{eq.upper.transformation.function} for $U_{w; c}^{t}\paren{\bullet; \tau}$)
    and the quantile value $\widehat{Q}_{\lambda; D^{+}}^{\widehat{\Gamma}_{a}, \lo}(\alpha)$ (bellow) is greater than $\widehat{\tau}_{\lambda}^{\widehat{\Gamma}_{a^{+}}}\paren{X_{n+1}}\paren{1 + \frac{\widehat{t}_{\lambda; a}^2}{n+1}}^{\frac{1}{2}}$
    that is,
    \begin{align}
        \label{eq.lower.quantile.global.covariance}
        \widehat{Q}_{\lambda; D^{+}}^{\widehat{\Gamma}_{a}, \lo}(\alpha)
        &
        :=
        \norm{
            \widehat{\Gamma}_{a^{+}}^{-\frac{1}{2}}\paren{
                Y_{\paren{i_{n, \alpha}^{n}}}
                - \hat{f}_{\lambda^{+}; D}\paren{X_{\paren{i_{n, \alpha}^{n}}}}
            }
        } - \widehat{\tau}_{\lambda}^{\widehat{\Gamma}_{a^{+}}}\paren{X_{\paren{i_{n, \alpha}^{n}}}}
        \\
        &
        \geq \widehat{\tau}_{\lambda}^{\widehat{\Gamma}_{a^{+}}}\paren{X_{n+1}}\paren{1 + \frac{\widehat{t}_{\lambda; a}^2}{n+1}}^{\frac{1}{2}},
        \notag
    \end{align}
    (see Eq.~\ref{eq.index.alpha} for $i_{n, \alpha}^{n}$).
    Otherwise, $\widetilde{C}_{\lambda; \alpha}^{\widehat{\Gamma}_{a}}\paren{X_{n+1}} = \emptyset$.
\end{proposition}
The proof is deferred to Appendix~\ref{proof.expresion.lower.GlobalEllipsoidCP.region}.
Comparing the expression of the upper and lower \textbf{G-EllipsoidCP}-regions (See Proposition~\ref{prop.expression.upper.GlobalEllipsoidCP.region}),
two differences emerge.
First in Eq.~\ref{eq.lower.quantile.global.covariance}, the sign before the correction $\widehat{\tau}_{\lambda}^{\Gamma}\paren{\bullet}$
in the initial quantile value is now negative compared to Eq.~\eqref{eq.upper.quantile.global.covariance}.
Second, the lower-correction function
$\paren{
    U_{\frac{1}{n+1}; 1}^{\widehat{t}_{\lambda; a}}
}^{-1}\paren{
    \bullet;
    \widehat{\tau}_{\lambda}^{\widehat{\Gamma}_{a^{+}}}\paren{X_{n+1}}
}$
replaces the upper one in the expression of the radius.
By Lemma~\ref{lm.variation.lower.transformation.function},
this lower-correction function is an increasing function, and it gets larger for smaller values of $\widehat{\tau}_{\lambda}^{\widehat{\Gamma}_{a^{+}}}\paren{X_{n+1}}$, which implies that the radius of the lower \textbf{G-EllipsoidCP}-region is increased for larger values of $\lambda n$.
This means that the larger $\lambda n$, the smaller the gap between these two lower and upper \textbf{G-EllipsoidCP}-regions,
and thus the smaller the \emph{thickness} (see also the experiment below for an empirical evidence of this conclusion).
%
%

\bigskip

\notparagraph{Evolution of the \emph{thickness} w.r.t. the sample size $n$} 
Figure~\ref{fig.thickness.GlobalEllipsoidCP} illustrate how the upper bound $\leb{\widetilde{C}_{\lambda; \alpha}^{\widehat{\Gamma}_{a}, \mathrm{up}}\paren{X_{n+1}} \setminus \widetilde{C}_{\lambda; \alpha}^{\widehat{\Gamma}_{a}, \lo}\paren{X_{n+1}}}$ on the \emph{thickness} $\mathrm{THK}_{\lambda; \alpha}^{\widehat{\Gamma}_{a}}\paren{X_{n+1}}$
does depend on the sample size $n$.
All the simulation setup is the same as the one previously described in Section~\ref{sec.known.cov.numerical.experiments}.
Each boxplot (corresponding to a training-sample size) reports the
values of this upper-bound across 20 repetitions.

\begin{figure}[H]
    \centering
    \includegraphics[width=0.80\textwidth]{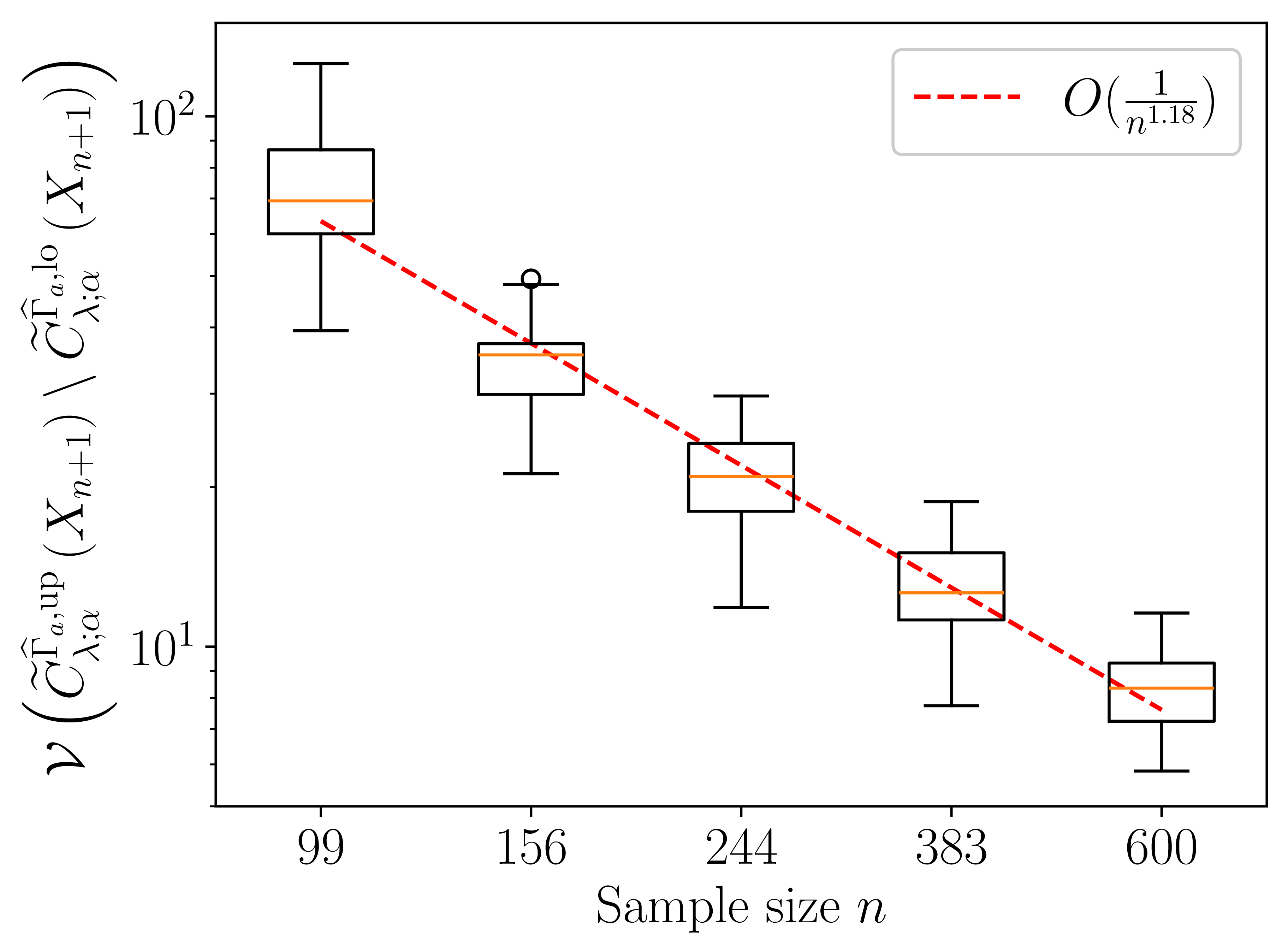}
    \caption{
        Evolution of the computable empirical upper-bound $\leb{\widetilde{C}_{\lambda; \alpha}^{\widehat{\Gamma}_{a}, \mathrm{up}}\paren{X_{n+1}} \setminus \widetilde{C}_{\lambda; \alpha}^{\widehat{\Gamma}_{a}, \lo}\paren{X_{n+1}}}$
        for the \emph{thickness} $\thicc{\widehat{\Gamma}_{a}}$ across 20 repetitions for $\alpha = 0.1$ and $\lambda = 10^{-2}$.
    }
    \label{fig.thickness.GlobalEllipsoidCP}
\end{figure}
As in Figure~\ref{fig.thickness.stableCP},
the dashed red-line highlights the rate of improvement of the empirical upper-bound as $n$ grows ($\lambda$ is fixed).
On average, the upper-bound gets smaller as the training-sample size grows.
The rate of improvement of the upper-bound is close to 1 (estimated at 1.18), which is consistent with the one of Figure~\ref{fig.thickness.stableCP} (up to the variability of the small number of replicates) formerly established with a known inter-task covariance.
That is to say, integrating a global inter-task covariance estimator
does not seem to affect too strongly the rate of improvement of the upper-bound for $\lambda$ fixed
(see Appendix~\ref{sec.evolution.thickness.global} for $\lambda \propto \frac{1}{\sqrt{n}}$).

\bigskip

\notparagraph{Comparison between \textbf{G-EllipsoidCP} and \textbf{SplitCP}}
Keeping the same experimental setup as the one described in Section~\ref{sec.known.cov.numerical.experiments}, the volume of the upper \textbf{G-EllipsoidCP}-region is now compared with the one of the \textbf{SplitCP}-region. 
Once more, let us emphasize that the values reported in Figure~\ref{fig.volume.GlobalEllipsoidCP} correspond to the ratio between the volume of the considered region divided by the volume of the \textbf{OracleCP}-region.

\begin{figure}[H]
    \centering
    \includegraphics[width=\textwidth]{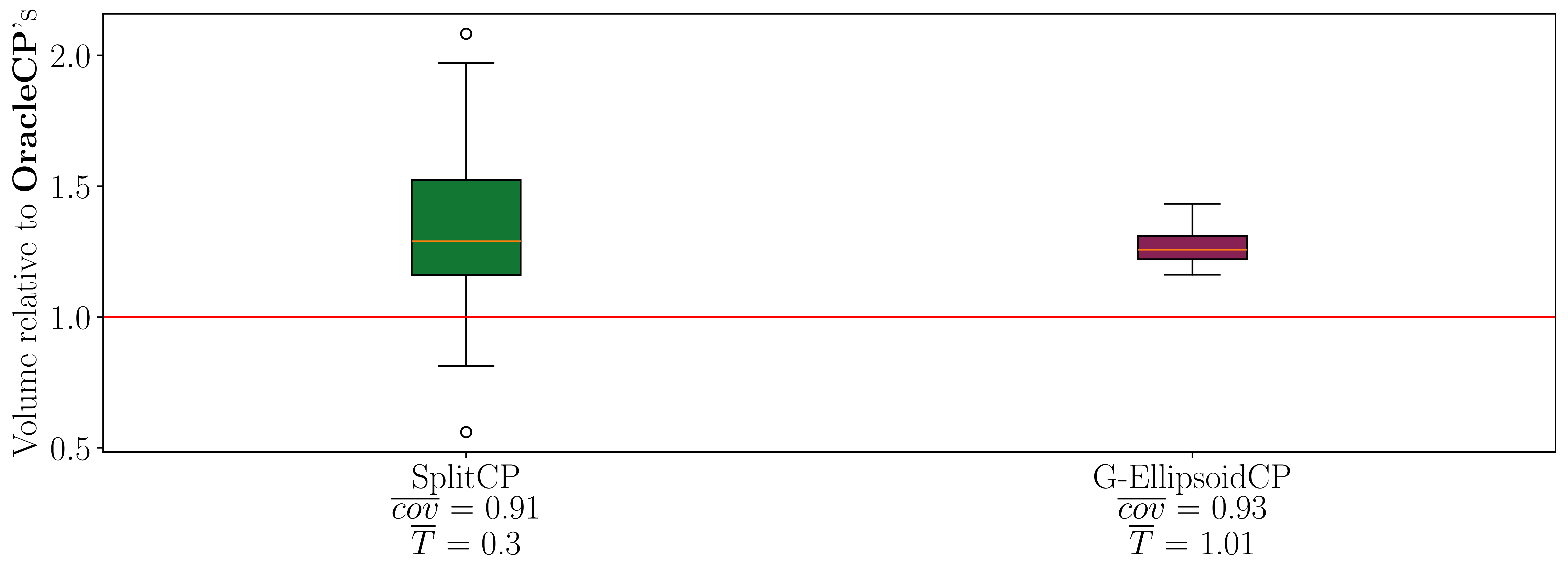}
    \caption{
        Comparison between the \textbf{G-EllipsoidCP}-regions (in red)
        and the \textbf{SplitCP}-regions (in green) in terms of
        volume relative to that of the \textbf{OracleCP}-regions
        (see Appendix~\ref{sec.oracle.cp}) for $\alpha = 0.1$
        with $\abss{D} = n = 500$ and $\abss{D_{\mathrm{train}}} = 375$.
    }
    \label{fig.volume.GlobalEllipsoidCP}
\end{figure}

According to Figure~\ref{fig.volume.GlobalEllipsoidCP}, the average (relative) volume of the \textbf{G-EllipsoidCP}-region is slightly smaller than the one of the \textbf{SplitCP}-region.
But the most striking phenomenon is the 
significantly larger variability of the volume of the \textbf{SplitCP}-regions compared to that of the \textbf{G-EllipsoidCP}-regions. 
Therefore introducing an estimated inter-task covariance does not seem to change the observation that the \textbf{SplitCP}-region can be much larger (and then strongly less informative) than the \textbf{G-EllipsoidCP}-region.

Despite its smaller volume (on average), the \textbf{G-EllipsoidCP}-region empirically exhibits a larger empirical coverage-probability
compared to the \textbf{SplitCP}-region.

Finally, the comparison of the reported computation times leads to the same conclusion as the one already drawn for a known inter-task covariance matrix:
$(1)$ the \textbf{G-EllipsoidCP}-region has almost the same computational cost as the one of the \textbf{OracleCP}-region, $(2)$ computing the \textbf{SplitCP}-region is half the price of computing the \textbf{OracleCP}-region. But the latter comes with a loss of informativeness which induces the higher variability (observed in the above boxplots).

\newpage
\appendix
\section{Concerning the statistical framework}
\label{sec.appendix.stat.framework}
This section contains the proofs of the results in Section~\ref{sec.statistical.framework}.

\subsection{Proof of Lemma~\ref{lm.predictor.well.defined}}
\label{proof.predictor.well.defined}
\begin{proof}
\notparagraph{Lower semi-continuity}
Let $\paren{x, y} \in D$.
By \citet[see Lemma 3]{audiffrenStabilityMultitaskKernel2013},
for every $x \in \mathcal{X}$,
the evaluation linear-function
$L_x : \mathcal{H} \to \mathcal{Y},$
$f \mapsto L_x\paren{f} = f(x)$
is bounded, thus continuous.
It follows that under  \eqref{asm.lsc.loss},
the following function
\begin{align*}
    f \in \mathcal{H} \mapsto \ell\paren{y, f(x)}
    = \croch{\ell\paren{y, \bullet}}\paren{f(x)}
    = \croch{\ell\paren{y, \bullet} \circ L_x} \paren{f} \in \mathbb{R}
\end{align*}
is lower semi-continuous,
since it is a composition of
a lower semi-continuous function and
a continuous function.
Therefore,
the regularized empirical-risk function (see Eq.~\ref{eq.regularized.empirical.risk.function})
is lower semi-continuous since
the regularization term,
being the norm squared, is continuous, and
the empirical-risk term
is a linear combination of lower semi-continuous functions.

\bigskip
\notparagraph{Convexity}
Let $\paren{x, y} \in D$. Under \eqref{asm.conv.loss},
the following function
\begin{align*}
    f \in \mathcal{H} \mapsto \ell\paren{y, f(x)}
    = \croch{\ell\paren{y, \bullet} \circ L_x} \paren{f} \in \mathbb{R}
\end{align*}
is convex since it is the composition between a convex function
and a linear function.
Therefore, the regularized empirical-risk function (see Eq.~\ref{eq.regularized.empirical.risk.function}) is $2\lambda$-strongly convex,
since the regularization term is the $\lambda$ times the norm squared,
and  the empirical-risk term is a linear combination with positive
coefficient of convex functions.

\bigskip
\notparagraph{Coercivity}
Under \eqref{asm.lower.bounded.loss},
for every $f \in \mathcal{H}$,
$\widehat{\mathbf{R}}_{\lambda; D}\paren{f} \geq c_{\ell} + \lambda \normh{f}^2$,
where the lower bound is a coercive function of $f$.

\bigskip
\notparagraph{Conclusion}
By \citet[]{alexanderianOptimizationInfinitedimensionalHilbert2019}~(see Corollary 5.6),
the regularized empirical-risk function admits
a unique minimizer in $\mathcal{H}$.
Let us note said minimizer $\hat{f}_{\lambda; D}$.
\end{proof}

\subsection{Proof of Lemma~\ref{lm.representer}}
\label{proof.representer}
\begin{proof}
    Under \eqref{asm.conv.loss}, \eqref{asm.lsc.loss} and \eqref{asm.lower.bounded.loss},
    Lemma~\ref{lm.predictor.well.defined}
    states that
    the regularized empirical-risk function
    $\widehat{\mathbf{R}}_{\lambda; D}\paren{\bullet}$
    (see Eq.~\ref{eq.regularized.empirical.risk.function})
    admits a unique minimizer $\hat{f}_{\lambda; D} \in \mathcal{H}$.
    Additionally, for every function $f \in \mathcal{H}$, and
    every index $i \in \brac{1, \ldots, n}$, and
    every coordinate $k \in \brac{1, \ldots, p}$,
    \begin{align*}
        \croch{f\paren{X_i}}_k
        = \dotprod{e_k}{f\paren{X_i}}{}
        = \dotprod{K\paren{\bullet, X_i} e_k}{f}{\mathcal{H}},
    \end{align*}
    where $e_k \in \mathbb{R}^{p \times 1}$, and
    the last equality follows from the reproducing property of
    the hypothesis space $\mathcal{H}$
    \citep[see Definition 2.1]{micchelliLearningVectorvaluedFunctions2005}.
    It follows that
    \begin{align*}
        f\paren{X_i}
        = \sum_{k = 1}^{p}
        \dotprod{K\paren{\bullet, X_i} e_{k}}{f}{\mathcal{H}} e_{k}.
    \end{align*}
    Since the hypothesis space $\mathcal{H}$
    is a Hilbert space, and since
    its subspace $\mathcal{A}_{\mathbf{X}}$
    (see Eq.~\ref{eq.subspace})
    has a finite dimension,
    thus, $\mathcal{H} = \mathcal{A}_{\mathbf{X}} \oplus \mathcal{A}_{\mathbf{X}}^{\perp}$.
    That is to say, for every $f \in \mathcal{H}$,
    there exist unique functions $\mathrm{proj}_{\mathcal{A}_{\mathbf{X}}}\paren{f} \in \mathcal{A}_{\mathbf{X}}$ and
    $\mathrm{proj}_{\mathcal{A}_{\mathbf{X}}^{\perp}}\paren{f} \in \mathcal{A}_{\mathbf{X}}^{\perp}$
    such that
    $f = \mathrm{proj}_{\mathcal{A}_{\mathbf{X}}}\paren{f} + \mathrm{proj}_{\mathcal{A}_{\mathbf{X}}^{\perp}}\paren{f}$.
    It follows that, for every $i \in \brac{1, \ldots, n}$,
    \begin{align*}
        f\paren{X_i}
        &
        = \sum_{k = 1}^{p}
        \dotprod{K\paren{\bullet, X_i} e_{k}}{f}{\mathcal{H}} e_{k}
        = \sum_{k = 1}^{p}
        \dotprod{K\paren{\bullet, X_i} e_{k}}{\mathrm{proj}_{\mathcal{A}_{\mathbf{X}}}\paren{f}}{\mathcal{H}} e_{k}
        \\
        &
        = \croch{\mathrm{proj}_{\mathcal{A}_{\mathbf{X}}}\paren{f}}\paren{X_i}.
    \end{align*}
    Moreover, 
    since $\mathrm{proj}_{\mathcal{A}_{\mathbf{X}}}\paren{f} \perp \mathrm{proj}_{\mathcal{A}_{\mathbf{X}}^{\perp}}\paren{f}$,
    \begin{align*}
        \norm{f}_{\mathcal{H}}^2
        &
        = \norm{
            \mathrm{proj}_{\mathcal{A}_{\mathbf{X}}}\paren{f}
        }_{\mathcal{H}}^2
        + \norm{
            \mathrm{proj}_{\mathcal{A}_{\mathbf{X}}^{\perp}}\paren{f}
        }_{\mathcal{H}}^2
        + 2 \dotprod{
            \mathrm{proj}_{\mathcal{A}_{\mathbf{X}}}\paren{f}
        }{
            \mathrm{proj}_{\mathcal{A}_{\mathbf{X}}^{\perp}}\paren{f}
        }{\mathcal{H}}
        \\
        &
        = \norm{
            \mathrm{proj}_{\mathcal{A}_{\mathbf{X}}}\paren{f}
        }_{\mathcal{H}}^2
        + \norm{
            \mathrm{proj}_{\mathcal{A}_{\mathbf{X}}^{\perp}}\paren{f}
        }_{\mathcal{H}}^2.
    \end{align*}
    Then, the regularized empirical-risk
    $\widehat{\mathbf{R}}_{\lambda; D}\paren{\hat{f}_{\lambda; D}}$ of
    the predictor $\hat{f}_{\lambda; D} \in \mathcal{H}$ is given by
    \begin{align*}
        \widehat{\mathbf{R}}_{\lambda; D}\paren{\hat{f}_{\lambda; D}}
        &
        = \frac{1}{n}\sum_{i=1}^{n}
        \ell\paren{
            Y_i, f\paren{X_i}
        } + \lambda \normh{f}^2
        \\
        &
        =
        \frac{1}{n}\sum_{i=1}^{n}
        \ell\paren{
            Y_i, \croch{\mathrm{proj}_{\mathcal{A}_{\mathbf{X}}}\paren{\hat{f}_{\lambda; D}}}\paren{X_i}
        }
        + \lambda \norm{
            \mathrm{proj}_{\mathcal{A}_{\mathbf{X}}}\paren{\hat{f}_{\lambda; D}}
        }_{\mathcal{H}}^2
        \\
        &
        \quad
        + \lambda \norm{
            \mathrm{proj}_{\mathcal{A}_{\mathbf{X}}^{\perp}}\paren{\hat{f}_{\lambda; D}}
        }_{\mathcal{H}}^2
        \\
        &
        =
        \widehat{\mathbf{R}}_{\lambda; D}\paren{
            \mathrm{proj}_{\mathcal{A}_{\mathbf{X}}}\paren{\hat{f}_{\lambda; D}}
        } + \lambda \norm{
            \mathrm{proj}_{\mathcal{A}_{\mathbf{X}}^{\perp}}\paren{\hat{f}_{\lambda; D}}
        }_{\mathcal{H}}^2
        \\
        &
        \geq
        \widehat{\mathbf{R}}_{\lambda; D}\paren{
            \mathrm{proj}_{\mathcal{A}_{\mathbf{X}}}\paren{\hat{f}_{\lambda; D}}
        }
        \\
        &
        \geq 
        \widehat{\mathbf{R}}_{\lambda; D}\paren{\hat{f}_{\lambda; D}},
    \end{align*}
    where the last inequality holds true since $\hat{f}_{\lambda; D} \in \mathcal{H}$
    is the minimizer of $\widehat{\mathbf{R}}_{\lambda; D}\paren{\bullet}$.
    Thus,
    \begin{align*}
        \widehat{\mathbf{R}}_{\lambda; D}\paren{\hat{f}_{\lambda; D}}
        = \widehat{\mathbf{R}}_{\lambda; D}\paren{
            \mathrm{proj}_{\mathcal{A}_{\mathbf{X}}}\paren{\hat{f}_{\lambda; D}}
        },
    \end{align*}
    and by uniqueness of the minimizer,
    $\hat{f}_{\lambda; D} = \mathrm{proj}_{\mathcal{A}_{\mathbf{X}}}\paren{\hat{f}_{\lambda; D}} \in \mathcal{A}_{\mathbf{X}}$.
    Therefore, by definition of $\mathcal{A}_{\mathbf{X}}$ (see Eq.~\ref{eq.subspace}),
    there exists a weight matrix $\widehat{W}_{\lambda; D} \in \mathbb{R}^{np}$
    such that, for every $x \in \mathcal{X}$,
    \begin{align*}
        \hat{f}_{\lambda; D}
        = \sum_{i = 1}^{n} K\paren{\bullet, X_i}\paren{e_i^{T} \otimes I_{p}}\widehat{W}_{\lambda; D}.
    \end{align*}
\end{proof}

\subsection{Proof of Eq.~\ref{eq.regularization.erf.matrix}}
\label{proof.regularization.erf.matrix}
\begin{proof}
Let $f \in \mathcal{A}_{\mathbf{X}}$.
Then, there exists a matrix $W \in \mathbb{R}^{np}$
such that,
\begin{align*}
    f = \sum_{j = 1}^{n} K\paren{\bullet, X_j} \paren{e_j^{T} \otimes I_{p}} W.
\end{align*}
Thus, for every $i \in \brac{1, \ldots, n}$,
\begin{align*}
    f\paren{X_i}
    &
    = \sum_{j = 1}^{n} K\paren{X_i, X_j} \paren{e_j^{T} \otimes I_{p}} W
    = \sum_{j = 1}^{n} \paren{e_i^{T} \otimes I_{p}} \mathbf{K}_{\mathbf{X}} \paren{e_{j} \otimes I_{p}}\paren{e_j^{T} \otimes I_{p}} W
    \\
    &
    = \paren{e_i^{T} \otimes I_{p}} \mathbf{K}_{\mathbf{X}} \paren{\sum_{j=1}^{n}\paren{e_{j} \otimes I_{p}}\paren{e_j^{T} \otimes I_{p}}} W
    = \paren{e_i^{T} \otimes I_{p}} \mathbf{K}_{\mathbf{X}} W,
\end{align*}
where $\mathbf{K}_{\mathbf{X}} \in \mathbb{R}^{np \times np}$ is
the Gram matrix (see Eq.~\ref{eq.gram.matrix}).
Moreover,
\begin{align*}
    \normh{f}^2
    &
    = \dotprod{f}{f}{\mathcal{H}}
    \\
    &
    =
    \dotprod{
        \sum_{j = 1}^{n} K\paren{\bullet, X_j} \paren{e_j^{T} \otimes I_{p}} W
    }{
        \sum_{j = 1}^{n} K\paren{\bullet, X_j} \paren{e_j^{T} \otimes I_{p}} W
    }{\mathcal{H}}
    \\
    &
    =
    \sum_{j = 1}^{n}
    \sum_{i = 1}^{n}
    \dotprod{
        K\paren{\bullet, X_j} \paren{e_j^{T} \otimes I_{p}} W
    }{
        K\paren{\bullet, X_i} \paren{e_i^{T} \otimes I_{p}} W
    }{\mathcal{H}}
    \\
    &
    =
    \sum_{l = 1}^{n}
    \sum_{j = 1}^{n}
    \dotprod{
        \paren{e_i^{T} \otimes I_{p}} W
    }{
        K\paren{X_i, X_j} 
        \paren{e_j^{T} \otimes I_{p}} W
    }{}
    \\
    &
    =
    \sum_{l = 1}^{n}
    \sum_{j = 1}^{n}
    W^{T}
    \paren{e_{i} \otimes I_{p}} K\paren{X_i, X_j}
    \paren{e_j^{T} \otimes I_{p}} W
    \\
    &
    =
    W^{T}
    \paren{
        \sum_{l = 1}^{n}
        \sum_{j = 1}^{n}
        \paren{e_{i} \otimes I_{p}} K\paren{X_i, X_j}
        \paren{e_j^{T} \otimes I_{p}}
    } W
    = W^{T} \mathbf{K}_{\mathbf{X}} W,
\end{align*}
where the third inequality follows from \citet[see Proposition 2.1.a]{micchelliLearningVectorvaluedFunctions2005}.
Therefore, the regularized empirical-risk is given by,
for every $f \in \mathcal{H}$,
\begin{align*}
    \widehat{\mathbf{R}}_{\lambda; D}\paren{f}
    &
    = \frac{1}{n}\sum_{i=1}^{n}
    \ell\paren{Y_i, f\paren{X_i}}
    + \lambda \normh{f}^2
    \\
    &
    =
    \frac{1}{n}\sum_{i=1}^{n}
    \ell\paren{Y_i, 
        \paren{e_i^{T} \otimes I_{p}} \mathbf{K}_{\mathbf{X}}W
    }
    + \lambda W^{T} \mathbf{K}_{\mathbf{X}} W
    =: \widehat{R}_{\lambda; D}\paren{W}.
\end{align*}
\end{proof}
%
%
%

\begin{proposition}
    \label{prop.vector.predictor.well.defined}
    Assume \eqref{asm.conv.loss}, \eqref{asm.lsc.loss} and \eqref{asm.lower.bounded.loss} hold true.
    Then, the regularized empirical-risk vector-function
    $\widehat{R}_{\lambda;D}\paren{\bullet} : \mathbb{R}^{np} \to \mathbb{R}$
    (see Eq.~\ref{eq.regularization.erf.matrix})
    admits minimizers over $\mathbb{R}^{np}$.
    Moreover, it admits a unique minimizer $\widehat{W}_{\lambda; D}$ over
    $\mathrm{range}\paren{\mathbf{K}_{\mathbf{X}}} \subseteq \mathbb{R}^{np}$
    where $\mathbf{K}_{\mathbf{X}} \in \mathbb{R}^{np \times np}$ is the Gram matrix (see Eq.~\ref{eq.gram.matrix}).
\end{proposition}
\begin{proof}
\notparagraph{Existence}
Under \eqref{asm.conv.loss}, \eqref{asm.lsc.loss} and \eqref{asm.lower.bounded.loss},
Lemma~\ref{lm.predictor.well.defined}
states that
the regularized empirical-risk function
$\widehat{\mathbf{R}}_{\lambda; D}\paren{\bullet}$
admits a unique minimizer $\hat{f}_{\lambda; D} \in \mathcal{H}$.
Additionally, Lemma~\ref{lm.representer} states
that $\hat{f}_{\lambda; D} \in \mathcal{H} \in \mathcal{A}_{\mathbf{X}}$
(see Eq.~\ref{eq.subspace}).
It follows that
\begin{align*}
    \min_{f \in \mathcal{H}}
    \widehat{\mathbf{R}}_{\lambda; D}\paren{f}
    =
    \min_{f \in \mathcal{A}_{\mathbf{X}}}
    \widehat{\mathbf{R}}_{\lambda; D}\paren{f}
    =
    \min_{W \in \mathbb{R}^{np}}
    \widehat{R}_{\lambda; D}\paren{W},
\end{align*}
where $\widehat{\mathbf{R}}_{\lambda; D}\paren{\bullet}$
given by Eq.~\eqref{eq.regularized.empirical.risk.function}
and $\widehat{R}_{\lambda}\paren{\bullet}$ in Eq.~\eqref{eq.regularization.erf.matrix}.
Since the former admits a minimum such that
$\widehat{\mathbf{R}}_{\lambda; D}\paren{\hat{f}_{\lambda; D}} = \widehat{R}_{\lambda; D}\paren{\widehat{W}_{\lambda; D}}$
for some $\widehat{W}_{\lambda; D} \in \mathbb{R}^{np}$,
it follows that $\widehat{R}_{\lambda; D}\paren{\bullet}$ admits a minimizer. 

\bigskip
\notparagraph{Uniqueness}
Since the Gram matrix $\mathbf{K}_{\mathbf{X}} \in \mathrm{Sym}_{np}^{+}\paren{\mathbb{R}}$
\citep[see Proposition 2.1.c]{micchelliLearningVectorvaluedFunctions2005},
under \eqref{asm.conv.loss},
$\widehat{R}_{\lambda; D}\paren{\bullet}$
is $\lambda$-strongly convex over the range of
the Gram matrix $\mathbf{K}_{\mathbf{X}} \in \mathbb{R}^{np \times np}$.
Further, under \eqref{asm.lower.bounded.loss},
this function is also coercive
over the range of the Gram matrix $\mathbf{K}_{\mathbf{X}}$.
Therefore, it admits a unique minimizer
over the range of the Gram matrix $\mathbf{K}_{\mathbf{X}}$.
\end{proof}

\section{Concerning the approximation scheme}
This sections lists the proofs of the results in Section~\ref{sec.approximate.full.conformal.perdiction}.

\subsection{Preliminary property}
\begin{lemma}[Minimum gap]
    \label{lm.minimum.gap}
    Let $\mathcal{G}$ designate a Hilbert space with a norm $\norm{\bullet}_{\mathcal{G}}$,
    and $R : \mathcal{G} \to \mathbb{R}$,
    a $2M$-strongly convex function, and
    $f^{\star}$, a minimizer of $R$.
    Then, for every $f \in \mathcal{G}$,
    \begin{align*}
        M\norm{f - f^\star}_{\mathcal{G}}^2
        \leq
        R(f)
        -
        R(f^\star).
    \end{align*}
\end{lemma}
\begin{proof}
    The characterization of a $2M$-strongly convex function
    implies the following, for every $t \in (0, 1]$,
    \begin{align*}
        R\paren{
            tf + (1 - t)f^{\star}
        }
        \leq
        tR(f)
        + (1 - t) R(f^{\star})
        - M t (1 - t)\norm{f - f^{\star}}_{\mathcal{G}}^2.
    \end{align*}
    It follows that,
    \begin{align*}
        R\paren{
            f^{\star}
            + t \paren{f - f^{\star}}
        }
        -
        R(f^{\star})
        \leq
        tR(f)
        -
        tR(f^{\star})
        - M t (1 - t)\norm{f - f^{\star}}_{\mathcal{G}}^2.
    \end{align*}
    Since $f^{\star}$ is a minimizer,
    then, the left-hand side is non-negative,
    and dividing both sides by $t$,
    \begin{align*}
        0
        \leq
        R(f)
        -
        R(f^{\star})
        - M (1 - t)\norm{f - f^{\star}}_{\mathcal{G}}^2.
    \end{align*}
    Putting the third term on the other side of the inequality
    \begin{align*}
        M (1 - t)\norm{f - f^{\star}}_{\mathcal{G}}^2
        \leq
        R(f)
        -
        R(f^{\star}).
    \end{align*}
    Finally, taking the supremum over all values of $t \in (0, 1]$
    on the left-hand side,
    \begin{align*}
        M\norm{f - f^{\star}}_{\mathcal{G}}^2
        \leq
        R(f)
        -
        R(f^{\star}).
    \end{align*}
\end{proof}

\subsection{Proof of Lemma~\ref{lm.sandwiching}}
\label{proof.sandwiching}
\begin{proof}
\notparagraph{p-value function}
Assuming Eq.~\eqref{eq.bound.score},
for every $y \in \mathcal{Y}$, and
every $i \in \brac{1, \ldots, n}$,
the following implications hold, almost surely,
\begin{align*}
\begin{aligned}
    &&\widetilde{S}_{\lambda; D^{y}}^{\lo} \paren{X_{i}, Y_{i}}
    &\geq \widetilde{S}_{\lambda; D^{y}}^{\up} \paren{X_{n+1}, y}
    \\
    \Longrightarrow
    &&S_{\lambda; D^{y}} \paren{X_i, Y_i} &\geq S_{\lambda; D^{y}} \paren{X_{n+1}, y}
    \\
    \Longrightarrow
    &&\widetilde{S}_{\lambda; D^{y}}^{\up} \paren{X_{i}, Y_{i}}
    &\geq \widetilde{S}_{\lambda; D^{y}}^{\lo} \paren{X_{n+1}, y}.
\end{aligned}
\end{align*}
By taking the indicator, almost surely,
\begin{align*}
    &\mathbbm{1}\brac{
        \widetilde{S}_{\lambda; D^{y}}^{\lo} \paren{X_{i}, Y_{i}}
        \geq \widetilde{S}_{\lambda; D^{y}}^{\up} \paren{X_{n+1}, y}
    }
    \\
    \leq
    &
    \mathbbm{1}\brac{S_{\lambda; D^{y}} \paren{X_i, Y_i} \geq S_{\lambda; D^{y}} \paren{X_{n+1}, y}},
    \\
    \leq
    &
    \mathbbm{1}\brac{
        \widetilde{S}_{\lambda; D^{y}}^{\lo} \paren{X_{i}, Y_{i}}
        \geq \widetilde{S}_{\lambda; D^{y}}^{\up} \paren{X_{n+1}, y}
    }.
\end{align*}
Finally, the sandwiching of the full-conformal p-value function follows
from  summing over $i \in \brac{1, \ldots, n}$, 
adding one, and dividing by $n+1$,
that is, almost surely,
\begin{align*}
    0 \leq \lfcpv{}{y} \leq \fcpv{y} \leq \ufcpv{}{y} \leq 1.
\end{align*}

\notparagraph{Prediction-regions and \emph{thickness}}
Let $\alpha \in \left[\frac{1}{n+1}, 1\right)$ designate a control-level 
and $y \in \mathcal{Y}$, a test output-value. The above sandwiching
of the full-conformal p-value function entails, almost surely,
\begin{align*}
    \lfcpv{}{y} > \alpha
    \Rightarrow
    \fcpv{y} > \alpha
    \Rightarrow
    \ufcpv{}{y} > \alpha.
\end{align*}
As a consequence of the definition of the prediction-regions,
almost surely,
\begin{align*}
    \lfcpr{}
    \subseteq
    \fcpr
    \subseteq
    \ufcpr{}.
\end{align*}
Moreover, the \emph{thickness} of the upper-approximate \textbf{FullCP}-region is bounded from above, almost surely,
\begin{align*}
    \mathrm{THK}_{\lambda; \alpha}\paren{X_{n+1}} 
    &=
    \mathcal{V}\left(
        \ufcpr{} \Delta \fcpr
    \right)
    =\mathcal{V}\left(
        \ufcpr{} \setminus \fcpr
    \right)
    \\
    &
    \leq
    \mathcal{V}\left(
        \ufcpr{} \setminus \lfcpr{}
    \right)
    .
\end{align*}
\end{proof}

\subsection{Proof of Proposition~\ref{prop.example.lip.loss}}
\label{proof.example.lip.loss}
\begin{proof}
Let $y \in \mathcal{Y}$
stand for an output-value.
For every $u \in \mathcal{Y}$
\begin{align*}
    \ell\paren{y, u}
    = \frac{1}{\sqrt{p}}
    \sum_{l=1}^{p} c\paren{y_l, u_l}
    = \frac{1}{\sqrt{p}}
    \sum_{l=1}^{p} \paren{
        c\paren{y_l, \bullet} \circ
        \dotprod{e_l}{\bullet}{}
    }\paren{u},
\end{align*}
where, for every $l \in \brac{1, \ldots, p}$,
$y_{l} \in \mathbb{R}$ is the coordinate of
index $l$ of the vector $y$.
Since for every $l \in \brac{1, \ldots, p}$,
$c\paren{y_l, \bullet} : \mathbb{R} \to \mathbb{R}$ is convex,
then by composition with a linear function, $u \in \mathcal{Y} \mapsto c\paren{y_{l}, u_{l}}$ is convex.
Since the sum of convex functions is itself convex,
it follows that $\ell\paren{y, \bullet}$ is convex.
For every $u, v \in \mathcal{Y}$
\begin{align*}
    \abss{
        \ell\paren{y, u}
        - \ell\paren{y, v}
    }
    = \abss{
    \frac{1}{\sqrt{p}}
    \sum_{l=1}^{p}
    \croch{
        c\paren{y_l, u_l}
        - c\paren{y_l, v_l}
    }
    }
    &
    \leq
    \frac{1}{\sqrt{p}}
    \sum_{l=1}^{p}
    \abss{
        c\paren{y_l, u_l}
        - c\paren{y_l, v_l}
    }
    \\
    &
    \leq
    \rho
    \frac{1}{\sqrt{p}}
    \sum_{l=1}^{p}
    \abss{
        u_l - v_l
    }
    \\
    &
    \leq \rho \norm{u - v},
\end{align*}
where the last inequality follows from
the order between the $\ell^{1}$-norm and the Euclidian norm.
Thus, $\ell\paren{y, \bullet}$ is $\rho$-Lipschitz continuous.
\end{proof}

\subsection{Proof of Lemma~\ref{lm.stability.bounds}}
\label{proof.stability.bounds}
\begin{proof}
    Let $y \in \mathcal{Y}$ denote a test output-value.
    Since \eqref{asm.lip.loss} implies \eqref{asm.lsc.loss},
    then, under \eqref{asm.conv.loss}, \eqref{asm.lip.loss}
    and \eqref{asm.lower.bounded.loss},
    Lemma~\ref{lm.predictor.well.defined} states that
    the regularized empirical-risk functions $\widehat{\mathbf{R}}_{\lambda; D^{y}}\paren{\bullet}$
    and $\widehat{\mathbf{R}}_{\lambda^{+}; D}\paren{\bullet}$ (see Eq.~\ref{eq.regularized.empirical.risk.function}),
    are respectively $2\lambda$ and $2\lambda^{+}$-strongly convex,
    and respectively admit unique minimizers $\hat{f}_{\lambda; D^{y}}$ and $\hat{f}_{\lambda^{+}; D}$.
    Since $\hat{f}_{\lambda; D^{y}}$
    is a minimizer of the $2 \lambda$-strongly convex function
    $\widehat{\mathbf{R}}_{\lambda; D^{y}}\paren{\bullet}$,
    by Lemma~\ref{lm.minimum.gap},
    \begin{align*}
        &
        \lambda
        \normh{
        \hat{f}_{\lambda; D^{y}}
        - \hat{f}_{\lambda^{+}; D}
        }^2
        \\
        &
        \leq
        \widehat{\mathbf{R}}_{\lambda; D^{y}}\paren{
            \hat{f}_{\lambda^{+}; D}
        } - \widehat{\mathbf{R}}_{\lambda; D^{y}}\paren{
        \hat{f}_{\lambda; D^{y}}
        }
        \\
        &
        \leq
        \frac{n}{n+1}\widehat{\mathbf{R}}_{0; D}\paren{
            \hat{f}_{\lambda^{+}; D}
        } + \lambda \normh{\hat{f}_{\lambda^{+}; D}}^2
        - \frac{n}{n+1} \widehat{\mathbf{R}}_{0; D}\paren{
            \hat{f}_{\lambda; D^{y}}
        } - \lambda \normh{\hat{f}_{\lambda; D^{y}}}^2
        \\
        &
        \quad
        +
        \frac{1}{n+1}
        \croch{
            \ell\paren{y, \hat{f}_{\lambda^{+}; D}\paren{X_{n+1}}}
            - \ell\paren{y, \hat{f}_{\lambda; D^{y}}\paren{X_{n+1}}}
        }.
    \end{align*}
    Moreover, since $\hat{f}_{\lambda^{+}; D}$
    is a minimizer of the $2 \lambda^{+}$-strongly convex function
    $\widehat{\mathbf{R}}_{\lambda^{+}; D}\paren{\bullet}$,
    by Lemma~\ref{lm.minimum.gap},
    \begin{align*}
        &
        \lambda^{+}
        \normh{
            \hat{f}_{\lambda; D^{y}}
            - \hat{f}_{\lambda^{+}; D}
        }^2
        \leq
        \widehat{\mathbf{R}}_{0; D}\paren{\hat{f}_{\lambda; D^{y}}}
        + \lambda^{+} \normh{\hat{f}_{\lambda; D^{y}}}^2
        - \widehat{\mathbf{R}}_{0; D}\paren{\hat{f}_{\lambda^{+}; D}}
        - \lambda^{+} \normh{\hat{f}_{\lambda^{+}; D}}^2
    \end{align*}
    Since $\lambda^{+} = \frac{n+1}{n}\lambda$,
    multiplying both sides by $-\frac{n}{n+1}$ entails
    \begin{align*}
        &
        \frac{n}{n+1}\widehat{\mathbf{R}}_{0; D}\paren{
            \hat{f}_{\lambda^{+}; D}
        } + \lambda \normh{\hat{f}_{\lambda^{+}; D}}^2
        - \frac{n}{n+1} \widehat{\mathbf{R}}_{0; D}\paren{
            \hat{f}_{\lambda; D^{y}}
        } - \lambda \normh{\hat{f}_{\lambda; D^{y}}}^2
        \leq
        - \lambda \normh{
            \hat{f}_{\lambda; D^{y}}
            - \hat{f}_{\lambda^{+}; D}
        }^2.
    \end{align*}
    By plugging this into the first sets of inequalities,
    \begin{align*}
        &
        \lambda
        \normh{
        \hat{f}_{\lambda; D^{y}}
        - \hat{f}_{\lambda^{+}; D}
        }^2
        \\
        &
        \leq - \lambda\normh{
            \hat{f}_{\lambda; D^{y}}
            - \hat{f}_{\lambda^{+}; D}
        }^2
        +
        \frac{1}{n+1}
        \croch{
            \ell\paren{y, \hat{f}_{\lambda^{+}; D}\paren{X_{n+1}}}
            - \ell\paren{y, \hat{f}_{\lambda; D^{y}}\paren{X_{n+1}}}
        }.
    \end{align*}
    This implies what follows,
    \begin{align*}
        \normh{
        \hat{f}_{\lambda; D^{y}}
        - \hat{f}_{\lambda^{+}; D}
        }^2
        &
        \leq \frac{1}{2\lambda\paren{n+1}}
        \croch{
            \ell\paren{y, \hat{f}_{\lambda^{+}; D}\paren{X_{n+1}}}
            - \ell\paren{y, \hat{f}_{\lambda; D^{y}}\paren{X_{n+1}}}
        }.
    \end{align*}
    Under \eqref{asm.lip.loss} and
    from \citet[see Lemma 3]{audiffrenStabilityMultitaskKernel2013},
    \begin{align*}
        \ell\paren{
            y,
            \hat{f}_{\lambda^{+}; D}\paren{X_{n+1}}
        }
        -
        \ell\paren{
        y,
        \hat{f}_{\lambda; D^{y}}\paren{X_{n+1}}
        }
         &
        \leq
        \rho_p
        \norm{
        \hat{f}_{\lambda^{+}; D}\paren{X_{n+1}}
        - \hat{f}_{\lambda; D^{y}}\paren{X_{n+1}}
        }
        \\
         &
        \leq
        \rho_p
        \norm{
        \croch{
        \hat{f}_{\lambda^{+}; D}
        - \hat{f}_{\lambda; D^{y}}
        } \paren{X_{n+1}}
        }
        \\
         &
        \leq
        \rho_p
        \norm{K\paren{X_{n+1}, X_{n+1}}}_{\mathrm{op}}^{\frac{1}{2}}
        \normh{
            \hat{f}_{\lambda^{+}; D}
            - \hat{f}_{\lambda; D^{y}}
        }.
    \end{align*}
    By conjoining the above this inequality
    with the one before, and dividing both sides by
    $\normh{
        \hat{f}_{\lambda^{+}; D}
        - \hat{f}_{\lambda; D^{y}}
    }$,
    \begin{align*}
        \normh{
            \hat{f}_{\lambda; D^{y}}
            - \hat{f}_{\lambda^{+}; D}
        }
        &
        \leq
        \frac{
            \rho_{p}\norm{K\paren{X_{n+1}, X_{n+1}}}_{\mathrm{op}}^{\frac{1}{2}}
        }{2 \lambda\paren{n+1}}.
    \end{align*}
\end{proof}

\section{Concerning the known inter-task covariance-matrix case}
This section details the proofs of the results in Section~\ref{sec.known.inter.task.covariance}.

\subsection{Preliminary properties}
\begin{lemma}
\label{lm.lambda.kernel.operator.norm}
For every $x \in \mathcal{X}$,
\begin{align*}
    \norm{K\paren{\bullet, x} \Gamma^{-\frac{1}{2}}}_{\mathrm{op}}
    := \sup_{y \in \mathbb{R}^{p}\setminus \brac{0_{\mathbb{R}^{p}}}}
    \frac{\norm{K\paren{\bullet, x} \Gamma^{-\frac{1}{2}} y}_{\mathcal{H}}}{\norm{y}}
    = 
    \norm{
        \Gamma^{-\frac{1}{2}}
        K\paren{x, x}
        \Gamma^{-\frac{1}{2}}
    }_{\mathrm{op}}^{\frac{1}{2}}.
\end{align*}
\end{lemma}
\begin{proof}
Let $x \in \mathcal{X}$ name an input vector, and
$y \in \mathbb{R}^p \setminus \brac{0_{\mathbb{R}^{p}}}$, an output-value.
On the one hand, 
\begin{align*}
    \norm{K\paren{\bullet, x} \Gamma^{-\frac{1}{2}} y}_{\mathcal{H}}^{2}
    &
    = \dotprod{
        K\paren{\bullet, x} \Gamma^{-\frac{1}{2}} y
    }{
        K\paren{\bullet, x} \Gamma^{-\frac{1}{2}} y
    }{\mathcal{H}}
    \\
    &
    =
    \dotprod{
        \Gamma^{-\frac{1}{2}}y
    }{
        K\paren{x, x}\Gamma^{-\frac{1}{2}}y
    }{}
    \\
    &
    =
    \dotprod{
        y
    }{
        \Gamma^{-\frac{1}{2}} K\paren{x, x}\Gamma^{-\frac{1}{2}}y
    }{}
    \\
    &
    \leq
    \norm{\Gamma^{-\frac{1}{2}} K\paren{x, x}\Gamma^{-\frac{1}{2}}}_{\mathrm{op}}
    \norm{y}^2,
\end{align*}
where the second equality follows from
\citet[in Proposition 2.1(a)]{micchelliLearningVectorvaluedFunctions2005}, and
the third equality holds true since $\Gamma^{-\frac{1}{2}}$ is symmetric, and
the inequality follows from the definition of the operator normalized of
a symmetric semi-definite matrix, which is equal to its spectral radius.
Dividing both sides by $\norm{y}^2$ and taking the square root yields
\begin{align*}
    \frac{\norm{K\paren{\bullet, x} \Gamma^{-\frac{1}{2}} y}_{\mathcal{H}}}{\norm{y}}
    \leq \norm{\Gamma^{-\frac{1}{2}} K\paren{x, x}\Gamma^{-\frac{1}{2}}}_{\mathrm{op}}^{\frac{1}{2}}.
\end{align*}
On the other hand,
if $\norm{\Gamma^{-\frac{1}{2}} K\paren{x, x}\Gamma^{-\frac{1}{2}} y} = 0$,
then,
\begin{align*}
    \frac{\norm{\Gamma^{-\frac{1}{2}} K\paren{x, x}\Gamma^{-\frac{1}{2}} y}}{\norm{y}} = 0 \leq 
    \norm{K\paren{\bullet, x} \Gamma^{-\frac{1}{2}}}_{\mathrm{op}}^2.
\end{align*}
Otherwise, that is, if $\norm{\Gamma^{-\frac{1}{2}} K\paren{x, x}\Gamma^{-\frac{1}{2}} y} \neq 0$,
then,
\begin{align*}
    \norm{\Gamma^{-\frac{1}{2}} K\paren{x, x}\Gamma^{-\frac{1}{2}} y}^2
    &
    = \dotprod{
        \Gamma^{-\frac{1}{2}} K\paren{x, x}\Gamma^{-\frac{1}{2}} y
    }{
        \Gamma^{-\frac{1}{2}} K\paren{x, x}\Gamma^{-\frac{1}{2}} y
    }{}
    \\
    &
    = \dotprod{
        \Gamma^{-\frac{1}{2}} \Gamma^{-\frac{1}{2}} K\paren{x, x}\Gamma^{-\frac{1}{2}} y
    }{
        K\paren{x, x}\Gamma^{-\frac{1}{2}} y
    }{}
    \\
    &
    = \dotprod{
        K\paren{\bullet, x} \Gamma^{-\frac{1}{2}} y
    }{
        K\paren{\bullet, x} \Gamma^{-\frac{1}{2}} \Gamma^{-\frac{1}{2}} K\paren{x, x}\Gamma^{-\frac{1}{2}} y
    }{\mathcal{H}}
    \\
    &
    \leq
    \normh{K\paren{\bullet, x} \Gamma^{-\frac{1}{2}} y}
    \normh{K\paren{\bullet, x} \Gamma^{-\frac{1}{2}} \Gamma^{-\frac{1}{2}} K\paren{x, x}\Gamma^{-\frac{1}{2}} y}
    \\
    &
    \leq
    \norm{K\paren{\bullet, x} \Gamma^{-\frac{1}{2}}}_{\mathrm{op}}
    \norm{y}
    \norm{K\paren{\bullet, x} \Gamma^{-\frac{1}{2}}}_{\mathrm{op}}
    \norm{\Gamma^{-\frac{1}{2}} K\paren{x, x}\Gamma^{-\frac{1}{2}} y},
\end{align*}
where the third equality follows from
\citet[in Proposition 2.1(a)]{micchelliLearningVectorvaluedFunctions2005},
and the first inequality follows from the Cauchy-Schwarz inequality, and
the last inequality follows from the definition of the operator norm.
Dividing both sides by $\norm{y} \norm{\Gamma^{-\frac{1}{2}} K\paren{x, x}\Gamma^{-\frac{1}{2}} y}$
yields
\begin{align*}
    \frac{\norm{\Gamma^{-\frac{1}{2}} K\paren{x, x}\Gamma^{-\frac{1}{2}} y}}{\norm{y}} \leq 
    \norm{K\paren{\bullet, x} \Gamma^{-\frac{1}{2}}}_{\mathrm{op}}^2.
\end{align*}
One concludes using the definition of the operator norm.
\end{proof}

\begin{lemma}
    \label{lm.gamma.prediction.bound}
    For every predictor $f \in \mathcal{H}$, and every input value $x \in \mathcal{X}$,
    \begin{align*}
        \norm{\Gamma^{-\frac{1}{2}} f(x)}
        \leq \norm{\Gamma^{-\frac{1}{2}} K\paren{x, x} \Gamma^{-\frac{1}{2}}}_{\mathrm{op}}^{\frac{1}{2}}
        \norm{f}_{\mathcal{H}}.
    \end{align*}
\end{lemma}
\begin{proof}
    Let $f \in \mathcal{H}$ stand for a predictor, and
    $x \in \mathcal{X}$, an input value.
    If $\norm{\Gamma^{-\frac{1}{2}} f(x)} = 0$, then the inequality trivially holds.
    Otherwise,
    \begin{align*}
        \norm{\Gamma^{-\frac{1}{2}}f(x)}^2
        &
        = f(x)^T \Gamma^{-\frac{1}{2}} \Gamma^{-\frac{1}{2}} f(x)
        \\
        &
        = \dotprod{\Gamma^{-\frac{1}{2}} \Gamma^{-\frac{1}{2}} f(x)}{f(x)}{}
        \\
        &
        = \dotprod{K\paren{\bullet, x} \Gamma^{-\frac{1}{2}} \Gamma^{-\frac{1}{2}} f(x)}{f}{\mathcal{H}}
        \\
        &
        \leq
        \normh{K\paren{\bullet, x} \Gamma^{-\frac{1}{2}} \Gamma^{-\frac{1}{2}} f(x)}
        \normh{f}
        \\
        &
        \leq
        \norm{K\paren{\bullet, x} \Gamma^{-\frac{1}{2}}}_{\mathrm{op}}
        \norm{\Gamma^{-\frac{1}{2}} f(x)}
        \normh{f}
        \\
        &
        \leq 
        \norm{\Gamma^{-\frac{1}{2}} K\paren{x, x} \Gamma^{-\frac{1}{2}}}_{\mathrm{op}}^{\frac{1}{2}}
        \norm{\Gamma^{-\frac{1}{2}} f(x)}
        \normh{f},
    \end{align*}
    where the third equality follows from the reproducing property of
    the hypothesis space $\mathcal{H}$, and
    the first inequality follows from the Cauchy-Schwarz inequality, and
    the second inequality follows form the definition of the operator norm, and
    the last inequality follows from Lemma~\ref{lm.lambda.kernel.operator.norm}.
    One concludes by dividing both sides by $\norm{\Gamma^{-\frac{1}{2}} f(x)}$.
\end{proof}
Let us adopt the convention stating that
\begin{align*}
    \sup \emptyset = -\infty.
\end{align*}
\begin{lemma}
    \label{lm.quantile}
    Let $a_{1}, \ldots, a_{m} \in \mathbb{R}$
    name some points, 
    \begin{align*}
        \sup \brac{
            q \in \mathbb{R} :
            \frac{1 + \sum_{i=1}^{m} \mathbbm{1}\brac{a_{i} \geq q}}{n+1} > \alpha
        } =
        \begin{cases}
        a_{\paren{i_{n, \alpha}^{m}}}&\mbox{if }\alpha \in \left[\frac{1}{n+1}, \frac{m+1}{n+1}\right),\\
        - \infty&\mbox{if }\alpha \in \left[\frac{m+1}{n+1}, 1\right).
        \end{cases}
    \end{align*}
\end{lemma}
\begin{proof}
    Let us notice that,
    for any $q \in \mathbb{R}$,
    \begin{align*}
        &
        \frac{1 + \sum_{i=1}^{m} \mathbbm{1}\brac{a_{i} \geq q}}{n+1} > \alpha
        \\
        &
        \Longleftrightarrow
        \sum_{i=1}^{m} \mathbbm{1}\brac{a_{i} \geq q} > \paren{n+1}\alpha - 1
        \\
        &
        \Longleftrightarrow
        \sum_{i=1}^{m} \mathbbm{1}\brac{a_{i} < q} < \paren{n+1}\paren{1 - \alpha} - \paren{n - m}.
    \end{align*}
    If $\alpha \in \left[\frac{m+1}{n+1}, 1\right)$, then
    \begin{align*}
        \paren{n+1}\paren{1 - \alpha} - \paren{n - m}
        &
        = n + 1 - \alpha \paren{n+1} + m - n
        \\
        &
        = m + 1 - \alpha \paren{n+1}
        \\
        &
        =
        \paren{n+1} \paren{\frac{m+1}{n+1} - \alpha} \leq 0.
    \end{align*}
    This entails that
    \begin{align*}
        &
        \sup \brac{
            q \in \mathbb{R} :
            \frac{1 + \sum_{i=1}^{m} \mathbbm{1}\brac{a_{i} \geq q}}{n+1} > \alpha
        }
        \\
        &
        =
        \sup \brac{
            q \in \mathbb{R} :
            \sum_{i=1}^{m} \mathbbm{1}\brac{a_{i} < q} < \paren{n+1}\paren{1 - \alpha} - \paren{n - m}
        }
        \\
        &
        \leq
        \sup \brac{
            q \in \mathbb{R} :
            \sum_{i=1}^{m} \mathbbm{1}\brac{a_{i} < q} < 0
        } = \sup \emptyset = - \infty.
    \end{align*}
    Let us then consider the case where $\alpha \in \left[\frac{1}{n+1}, \frac{m+1}{n+1}\right)$.
    Since the function $q \mapsto \sum_{i=1}^{m} \mathbbm{1}\brac{a_{i} < q}$
    is left continuous, 
    and since the set is defined by a strict inequality,
    then the supremum is attained at some value $a \in \brac{a_i : i \in \brac{1, \ldots, m}}$.
    \begin{align*}
        \sum_{i = 1}^{m}\mathbbm{1}\brac{a_{i} < a_{\paren{i_{n, \alpha}^{m}}}}
        &
        =
        \sum_{i = 1}^{m}\mathbbm{1}\brac{a_{\paren{i}} < a_{\paren{i_{n, \alpha}^{m}}}}
        \\
        &
        =
        \sum_{i = 1}^{i_{n, \alpha}^{m}}\mathbbm{1}\brac{a_{\paren{i}} < a_{\paren{i_{n, \alpha}^{m}}}}
        \\
        &
        = \sum_{i = 1}^{i_{n, \alpha}^{m}}\mathbbm{1}\brac{a_{\paren{i}} \leq a_{\paren{i_{n, \alpha}^{m}}}}
        - \sum_{i = 1}^{i_{n, \alpha}^{m}}\mathbbm{1}\brac{a_{\paren{i}} = a_{\paren{i_{n, \alpha}^{m}}}}
        \\
        &
        =
        \ceil{\paren{n+1}\paren{1 - \alpha} - \paren{n - m}}
        - \sum_{i = 1}^{i_{n, \alpha}^{n}}\mathbbm{1}\brac{a_{\paren{i}} = a_{\paren{i_{n, \alpha}^{n}}}}
        \\
        &
        \leq
        \ceil{\paren{n+1}\paren{1 - \alpha} - \paren{n - m}} - 1
        < \paren{n+1}\paren{1 - \alpha} - \paren{n - m},
    \end{align*}
    where the second equality follows from the fact that
    $i > i_{n, \alpha}^{m} \implies a_{\paren{i}} \geq a_{\paren{i_{n, \alpha}^{m}}}$, and
    the fourth inequality holds by counting and the definition of $i_{n, \alpha}^{m}$
    (see Eq.~\ref{eq.index.alpha}),
    and the last, by the fact that
    for $i = i_{n, \alpha}^{m}$,
    $a_{\paren{i}} = a_{\paren{i_{n, \alpha}^{n}}}$.
    Thus, $a_{\paren{i_{n, \alpha}^{n}}}$ is a candidate value for the supremum.
    Let us now show that any term greater than $a_{\paren{i_{n, \alpha}^{n}}}$
    cannot be a supremum, making $a_{\paren{i_{n, \alpha}^{n}}}$ the supremum.
    Let $j \in \brac{1, \ldots, n}$ be an index, such that,
    $a_{\paren{i_{n, \alpha}^{n}}} < a_{\paren{j}}$.
    It follows that $i_{n, \alpha}^{n} < j$ and
    \begin{align*}
        \sum_{i=1}^{m}\mathbbm{1}\brac{a_i < a_{\paren{j}}}
        &
        =
        \sum_{i=1}^{m}\mathbbm{1}\brac{a_{\paren{i}} < a_{\paren{j}}}
        \\
        &
        =
        \sum_{i=1}^{i_{n, \alpha}^{m}}\mathbbm{1}\brac{a_{\paren{i}} < a_{\paren{j}}}
        + \sum_{i=i_{n, \alpha}^{m}}^{j}\mathbbm{1}\brac{a_{\paren{i}} < a_{\paren{j}}}
        \\
        &
        \geq
        \sum_{i=1}^{i_{n, \alpha}^{m}}\mathbbm{1}\brac{a_{\paren{i}} \leq a_{\paren{i_{n, \alpha}^{m}}}}
        \\
        &
        \geq \ceil{\paren{n+1}\paren{1 - \alpha} - \paren{n - m}}
        \geq \paren{n+1}\paren{1 - \alpha} - \paren{n - m},
    \end{align*}
    where the second equality follows from $i>j \implies a_{\paren{i}} \geq a_{\paren{j}}$, and
    the first inequality follows from the fact that $a_{\paren{i}} \leq a_{\paren{i_{n, \alpha}^{m}}} \implies a_{\paren{i}} < a_{\paren{j}}$
    since $a_{\paren{i_{n, \alpha}^{m}}} < a_{\paren{j}}$, and the second one from counting.
    Therefore, $a_{\paren{i_{n, \alpha}^{m}}}$ must be the supremum.
\end{proof}

\subsection{Proof of Lemma~\ref{lm.stable.score.norm}}
\label{proof.stable.score.norm}
\begin{proof}
    Let $y \in \mathcal{Y}$ name a test output-value, and
    $\paren{x, u} \in \mathcal{X} \times \mathcal{Y}$ a data point.
    Since \eqref{asm.lip.loss} implies \eqref{asm.lsc.loss},
    then, under \eqref{asm.conv.loss}, \eqref{asm.lip.loss}
    and \eqref{asm.lower.bounded.loss},
    Lemma~\ref{lm.predictor.well.defined} states that
    $\hat{f}_{\lambda; D^{y}}$ and $\hat{f}_{\lambda^{+}; D}$ are well-defined.
    Then, under \eqref{asm.conv.loss}, \eqref{asm.lip.loss} and \eqref{asm.lower.bounded.loss},
    \begin{align*}
        \abss{
            S_{\lambda; D^{y}}^{\Gamma}
            \paren{x, u}
            - S_{\lambda^{+}; D}^{\Gamma}
            \paren{x, u}
        }
        &
        =
        \abss{
            \norm{
                \Gamma^{-\frac{1}{2}}
                \paren{
                    u - \hat{f}_{\lambda; D^{y}}(x)
                }
            }
            - \norm{
                \Gamma^{-\frac{1}{2}}
                \paren{
                    u - \hat{f}_{\lambda^{+}; D}(x)
                }
            }
        }
        \\
        &
        \leq
        \norm{
            \Gamma^{-\frac{1}{2}}
            \paren{
                \hat{f}_{\lambda; D^{y}}(x)
            - \hat{f}_{\lambda^{+}; D}(x)
            } 
        }
        \\
        &
        \leq
        \norm{
            \croch{
                \Gamma^{-\frac{1}{2}}
                \paren{
                    \hat{f}_{\lambda; D^{y}}
                    - \hat{f}_{\lambda^{+}; D}
                }
            }(x)
        }
        \\
        &
        \leq
        \norm{
            \Gamma^{-\frac{1}{2}} K\paren{x,x}\Gamma^{-\frac{1}{2}}
        }_{\mathrm{op}}^{\frac{1}{2}}
        \normh{
            \hat{f}_{\lambda; D^{y}}
            - \hat{f}_{\lambda^{+}; D}
        }
        \\
        &
        \leq \widehat{\tau}_{\lambda}^{\Gamma}(x),
    \end{align*}
    where the first inequality follows from the reverse triangle inequality, and
    the third inequality follows from Lemma~\ref{lm.gamma.prediction.bound}, and
    the last inequality follows from Lemma~\ref{lm.stability.bounds}
    (see Eq.~\ref{eq.score.stability.bound} for $\widehat{\tau}_{\lambda}^{\Gamma}(x)$).
    One can conclude by applying the triangle inequality.
\end{proof}

\subsection{Proof of Proposition~\ref{prop.upper.stableCP.region.fixed}}
\label{proof.upper.StableCP.region.fixed}
\begin{proof}
    Let $\alpha \in \left[\frac{1}{n+1}, 1\right)$
    designate a control-level, and
    $y \in \mathcal{Y}$,
    a test output-value.
    \begin{align*}
        &
        y \in \widetilde{C}_{\lambda; \alpha}^{\Gamma, \up}\paren{X_{n+1}}
        \\
        &
        \Longleftrightarrow
        \frac{
            1 + \sum_{i=1}^{n}
            \mathbbm{1}\brac{
                \widetilde{S}_{\lambda; D^{y}}^{\Gamma, \up} \paren{X_i, Y_i}
                \geq \widetilde{S}_{\lambda; D^{y}}^{\Gamma, \lo} \paren{X_{n+1}, y}
            }
        }{n+1}
        > \alpha
        \\
        &
        \Longleftrightarrow
        S_{\lambda; D}^{\Gamma} \paren{X_{n+1}, y}
        - \widehat{\tau}_{\lambda}^{\Gamma}\paren{X_{n+1}}
        \leq
        S_{\lambda; D}^{\Gamma} \paren{X_{\paren{i_{n, \alpha}^{n}}}, Y_{\paren{i_{n, \alpha}^{n}}}}
        + \widehat{\tau}_{\lambda}^{\Gamma}\paren{X_{\paren{i_{n, \alpha}^{n}}}}
        \\
        &
        \Longleftrightarrow
        \norm{
            \Gamma^{-\frac{1}{2}}\paren{
                y - \hat{f}_{\lambda^{+}; D}\paren{X_{n+1}}
            }
        }
        \leq
        \widehat{Q}_{\lambda; D^{+}}^{\Gamma, \up}(\alpha)
        + \widehat{\tau}_{\lambda}^{\Gamma}\paren{X_{n+1}},
    \end{align*}
    where the second equivalence follows from Lemma~\ref{lm.quantile} with $m = n$, and
    $\widehat{Q}_{\lambda; D^{+}}^{\Gamma, \up}(\alpha)$ is given by Eq.~\eqref{eq.fixed.ncs.quantile},
    and $D^{+}$, the data set containing
    $\paren{X_1, Y_1}$, \dots, $\paren{X_{n}, Y_{n}}$ and only
    the test input $X_{n+1}$, that is,
    \begin{align}
        \label{eq.half.augmented.data.set}
        D^{+} := \brac{
            \paren{X_{n}, Y_{n}},
            \ldots, \paren{X_{n}, Y_{n}},
            X_{n+1}
        }.
    \end{align}
\end{proof}

\subsection{Concerning the synthetic data set}
\label{sec.synthetic.data.set}
The data generating distribution
is from \citet[Appendix C.3]{braunMultivariateStandardizedResiduals2026}.
Let us detail the experimental parameters.
The input dimension is 2, and so is the output dimension, that is $d = p = 2$.
The perturbation is chosen to be Gaussian, and the number of anchors is set to $2$, that is, $K=2$.
When an inter-task covariance-matrix is involved, the regularization parameter $a$ is set to $10^{-6}$.

\section{Concerning the upper-bound on the thickness in the known inter-task covariance-matrix case}
This section provides the proofs of the results in Section~\ref{sec.known.cov.thickness.upper.bound}.

\subsection{Preliminary properties}
\begin{lemma}
    \label{lm.regularization.stability}
    Assume \eqref{asm.conv.loss}, \eqref{asm.lip.loss}, \eqref{asm.lower.bounded.loss} hold true.
    Then, the next upper-bound holds true,
    \begin{align*}
        \normh{\hat{f}_{\lambda^{+}; D} - \hat{f}_{\lambda; D}}
        \leq \frac{\rho_{p} \paren{\frac{1}{n}\sum_{i=1}^{n} \norm{K\paren{X_i, X_i}}_{\mathrm{op}}^{\frac{1}{2}}}}{2\lambda\paren{n+1}},
    \end{align*}
    where $\lambda^{+} := \frac{n+1}{n} \lambda$.
\end{lemma}
\begin{proof}
    The following proof takes bits of reasoning from \citet[see Proof of Theorem 2]{leeLeaveOneOutStableConformal2025}.
    Since \eqref{asm.lip.loss} implies \eqref{asm.lsc.loss},
    then, under \eqref{asm.conv.loss}, \eqref{asm.lip.loss}
    and \eqref{asm.lower.bounded.loss},
    Lemma~\ref{lm.predictor.well.defined} states that
    the regularized empirical-risk functions $\widehat{\mathbf{R}}_{\lambda; D}\paren{\bullet}$
    and $\widehat{\mathbf{R}}_{\lambda^{+}; D}\paren{\bullet}$,
    are respectively $2\lambda$ and $2\lambda^{+}$-strongly convex,
    and respectively admit unique minimizers $\hat{f}_{\lambda; D}$ and $\hat{f}_{\lambda^{+}; D}$.
    Since $\hat{f}_{\lambda; D}$
    is a minimizer of the $2 \lambda$-strongly convex function
    $\widehat{\mathbf{R}}_{\lambda; D}\paren{\bullet}$,
    by Lemma~\ref{lm.minimum.gap},
    \begin{align*}
        &
        \lambda
        \normh{
        \hat{f}_{\lambda; D}
        - \hat{f}_{\lambda^{+}; D}
        }^2
        \\
        &
        \leq
        \widehat{\mathbf{R}}_{\lambda; D}\paren{
            \hat{f}_{\lambda^{+}; D}
        } - \widehat{\mathbf{R}}_{\lambda; D}\paren{
        \hat{f}_{\lambda; D}
        }
        \\
        &
        \leq
        \widehat{\mathbf{R}}_{0; D}\paren{
            \hat{f}_{\lambda^{+}; D}
        } + \lambda \normh{\hat{f}_{\lambda^{+}; D}}^2
        - \widehat{\mathbf{R}}_{0; D}\paren{
            \hat{f}_{\lambda; D}
        } - \lambda \normh{\hat{f}_{\lambda; D}}^2.
    \end{align*}
    Moreover, since $\hat{f}_{\lambda^{+}; D}$
    is a minimizer of the $2 \lambda^{+}$-strongly convex function
    $\widehat{\mathbf{R}}_{\lambda^{+}; D}\paren{\bullet}$,
    by Lemma~\ref{lm.minimum.gap},
    \begin{align*}
        &
        \lambda^{+}
        \normh{
            \hat{f}_{\lambda; D}
            - \hat{f}_{\lambda^{+}; D}
        }^2
        \leq
        \widehat{\mathbf{R}}_{0; D}\paren{\hat{f}_{\lambda; D}}
        + \lambda^{+} \normh{\hat{f}_{\lambda; D}}^2
        - \widehat{\mathbf{R}}_{0; D}\paren{\hat{f}_{\lambda^{+}; D}}
        - \lambda^{+} \normh{\hat{f}_{\lambda^{+}; D}}^2
    \end{align*}
    Since $\lambda^{+} = \frac{n+1}{n}\lambda$,
    multiplying both sides by $-\frac{n}{n+1}$ entails
    \begin{align*}
        &
        \frac{n}{n+1}\widehat{\mathbf{R}}_{0; D}\paren{
            \hat{f}_{\lambda^{+}; D}
        } + \lambda \normh{\hat{f}_{\lambda^{+}; D}}^2
        - \frac{n}{n+1} \widehat{\mathbf{R}}_{0; D}\paren{
            \hat{f}_{\lambda; D}
        } - \lambda \normh{\hat{f}_{\lambda; D}}^2
        \leq
        - \lambda \normh{
            \hat{f}_{\lambda; D}
            - \hat{f}_{\lambda^{+}; D}
        }^2.
    \end{align*}
    Summing up with the first sets of inequalities yields
    \begin{align*}
        2 \lambda
        \normh{
        \hat{f}_{\lambda; D}
        - \hat{f}_{\lambda^{+}; D}
        }^2
        \leq \frac{1}{n+1}\paren{
            \widehat{\mathbf{R}}_{0; D}\paren{
                \hat{f}_{\lambda^{+}; D}
            } - \widehat{\mathbf{R}}_{0; D}\paren{
                \hat{f}_{\lambda; D}
            }
        }.
    \end{align*}
    Under \eqref{asm.lip.loss} and
    from \citet[see Lemma 3]{audiffrenStabilityMultitaskKernel2013},
    \begin{align*}
        \widehat{\mathbf{R}}_{0; D}\paren{
            \hat{f}_{\lambda^{+}; D}
        } - \widehat{\mathbf{R}}_{0; D}\paren{
            \hat{f}_{\lambda; D}
        }
        &
        = \frac{1}{n+1}
        \sum_{i=1}^{n}
        \croch{
            \ell\paren{Y_i, \hat{f}_{\lambda^{+}; D}\paren{X_i}}
            - \ell\paren{Y_i, \hat{f}_{\lambda; D}\paren{X_i}}
        }
        \\
        &
        \leq
        \rho_{p} \frac{1}{n+1}
        \sum_{i=1}^{n}
        \norm{\croch{\hat{f}_{\lambda^{+}; D} - \hat{f}_{\lambda; D}}\paren{X_i}}
        \\
        &
        \leq \rho_{p} \frac{1}{n+1}
        \sum_{i=1}^{n}
        \norm{K\paren{X_i, X_i}}_{\mathrm{op}}^{\frac{1}{2}}
        \normh{\hat{f}_{\lambda^{+}; D} - \hat{f}_{\lambda; D}}.
    \end{align*}
    Conjoining this inequality with the one before,
    and then dividing both sides by $2 \lambda \normh{\hat{f}_{\lambda^{+}; D} - \hat{f}_{\lambda; D}}$
    yields the desired result.
\end{proof}
\begin{lemma}
    \label{lm.risk.function.proper}
    Assume \eqref{asm.conv.loss} and \eqref{asm.lip.loss} hold true.
    Then, the risk function $\mathbf{R}_{0}\paren{\bullet} : \mathcal{H} \mapsto \mathbb{R}$
    is a continuous convex function.
    Additionally, assuming \eqref{asm.bounded.expected.loss} and
    \eqref{asm.bounded.gamma.kernel} also hold true,
    then the risk function $\mathbf{R}_{0}\paren{\bullet} : \mathcal{H} \mapsto \mathbb{R}$
    is proper and
    \begin{align*}
        \mathrm{Dom}\paren{\mathbf{R}_{0}} := \brac{f \in \mathcal{H} : \mathbf{R}_{0}\paren{f}<+\infty} = \mathcal{H}.
    \end{align*}
\end{lemma}
\begin{proof}
    For every $x \in \mathcal{X}$,
    let $L_{x} : \mathcal{H} \mapsto \mathbb{R}^{p}$
    designate a function, given by, for every $f \in \mathcal{H}$,
    $L_{x}\paren{f} = f(x) = \paren{f_1(x), \ldots, f_p(x)}$.
    Since $f \in \mathcal{H}$,
    for every $l \in \brac{1, \ldots, p}$
    \begin{align*}
        f_l(x) = \dotprod{e_l}{f(x)}{}
        = \dotprod{K\paren{\bullet, x} e_l}{f}{\mathcal{H}},
    \end{align*}
    where $\norm{K\paren{\bullet, x} e_l}_{\mathcal{H}} < +\infty$
    by definition of $K\paren{\bullet, x} : \mathbb{R}^{p} \to \mathcal{H}$.
    Thus, $L_x$ is a bounded linear operator.
    Under \eqref{asm.conv.loss} and \eqref{asm.lip.loss},
    for every $y \in \mathcal{Y}$,
    $\ell\paren{y, \bullet}$ is convex and continuous.
    Then, by composition,
    for every $\paren{x, y} \in \mathcal{X} \times \mathcal{Y}$,
    $f \in \mathcal{H} \mapsto \ell\paren{y, f(x)} = \paren{\ell\paren{y, \bullet} \circ L_{x}}$
    is convex and continuous.
    Therefore, by taking the expectation,
    $\mathbf{R}_{0}\paren{\bullet}$ is convex and continuous.

    For every $f \in \mathcal{H}$,
    \begin{align*}
        \mathbf{R}_{0}\paren{f}
        &
        \leq \mathbf{R}_{0}\paren{0}
        + \abss{\mathbf{R}_{0}\paren{f} - \mathbf{R}_{0}\paren{0}}
        \\
        &
        \leq C_{\ell}
        + \mathbb{E}\croch{
            \abss{
            \ell\paren{Y, f\paren{X}}
            - \ell\paren{Y, 0}}
        }
        \\
        &
        \leq C_{\ell}
        + \rho_{p}
        \mathbb{E}\croch{
            \norm{f\paren{X}}
        }
        \\
        &
        \leq C_{\ell}
        + \rho_{p}
        \mathbb{E}\croch{
            \norm{K\paren{\bullet, X}}_{\mathrm{op}}
        }
        \norm{f}_{\mathcal{H}}
        \\
        &
        \leq 
        C_{\ell}
        + \rho_{p} \norm{\Gamma}_{\mathrm{op}}^{\frac{1}{2}}
        \kappa^{\Gamma}
        \norm{f}_{\mathcal{H}} < +\infty,
    \end{align*}
    where the second inequality follows
    from \eqref{asm.bounded.expected.loss} and
    Jensen's inequality, and
    the third inequality follows from \eqref{asm.lip.loss}, and
    the last inequality follows form Lemma~\ref{lm.bounded.kernel}
    which holds true under \eqref{asm.bounded.gamma.kernel}.
    Therefore, $\mathrm{Dom}\paren{\mathbf{R}} = \brac{f \in \mathcal{H} : \mathbf{R}_{0}\paren{f}<+\infty} = \mathcal{H} \neq \emptyset$.
\end{proof}
\begin{lemma}
    \label{lm.reg.predictor.norm.bound}
    Assume \eqref{asm.conv.loss}, \eqref{asm.lip.loss},
    \eqref{asm.lower.bounded.loss}, \eqref{asm.bounded.gamma.kernel},
    \eqref{asm.bounded.expected.loss} and \eqref{asm.risk.minimizer.attained}
    hold true.
    Then, for every regularization parameter $\lambda \in \paren{0, +\infty}$,
    \begin{align*}
        \normh{f_{\lambda}} \leq \normh{f_{\mathcal{H}}},
    \end{align*}
    where the predictor $f_{\lambda} \in \mathcal{H}$ is given by Lemma~\ref{lm.reg.risk.minimizer}
    and the predictor $f_{\mathcal{H}} \in \mathcal{H}$ is given by Lemma~\ref{lm.min.norm.risk.minimizer}.
\end{lemma}
\begin{proof}
Let $\lambda \in \paren{0, +\infty}$
designate a regularization parameter.
Under \eqref{asm.conv.loss}, \eqref{asm.lip.loss},
\eqref{asm.lower.bounded.loss}, \eqref{asm.bounded.gamma.kernel},
\eqref{asm.bounded.expected.loss} and \eqref{asm.risk.minimizer.attained},
Lemma~\ref{lm.min.norm.risk.minimizer} states that
the predictor $f_{\mathcal{H}}$ is well-defined, and
for every $f \in \mathcal{H}$,
\begin{align*}
    \mathbf{R}_{0}\paren{f_{\mathcal{H}}} \leq \mathbf{R}_{0}\paren{f}.
\end{align*}
Since \eqref{asm.lip.loss} implies \eqref{asm.lsc.loss},
under \eqref{asm.conv.loss}, \eqref{asm.lip.loss}
and \eqref{asm.lower.bounded.loss},
Lemma~\ref{lm.reg.risk.minimizer} states that
the predictor $f_{\lambda} \in \mathcal{H}$ is well-defined, and
for every $f \in \mathcal{H}$,
\begin{align*}
    \mathbf{R}_{0}\paren{f_{\lambda}}
    + \lambda \normh{f_{\lambda}}^2
    \leq \mathbf{R}_{0}\paren{f}
    + \lambda \normh{f}^2.
\end{align*}
In particular, for $f = f_{\mathcal{H}}$,
\begin{align*}
    \mathbf{R}_{0}\paren{f_{\lambda}}
    + \lambda \normh{f_{\lambda}}^2
    \leq \mathbf{R}_{0}\paren{f_{\mathcal{H}}}
    + \lambda \normh{f_{\mathcal{H}}}^2.
\end{align*}
This entails that
\begin{align*}
    \lambda \normh{f_{\lambda}}^2
    \leq \mathbf{R}_{0}\paren{f_{\mathcal{H}}}
    - \mathbf{R}_{0}\paren{f_{\lambda}}
    + \lambda \normh{f_{\mathcal{H}}}^2,
    \leq \lambda \normh{f_{\mathcal{H}}}^2,
\end{align*}
where the second inequality follows from
the optimality of $f_{\mathcal{H}}$.
One can conclude by dividing on both sides
by $\lambda$ and taking the square root.
\end{proof}

\subsection{Proof of Lemma~\ref{lm.thickness.empirical.bound.maha}}
\label{proof.thickness.empirical.bound.maha}
\begin{proof}
\notparagraph{Preliminary properties}
Under \eqref{asm.conv.loss}, \eqref{asm.lip.loss} and \eqref{asm.lower.bounded.loss},
for every $y \in \mathcal{Y}$,
Lemma~\ref{lm.predictor.well.defined} ensures that
the predictors $\hat{f}_{\lambda; D^{y}}$ and $\hat{f}_{\lambda^{+}; D}$
exists and are unique.
Using the notations introduced above,
the upper-approximate non-conformity scores are bounded from above,
that is, for every $i \in \brac{1, \ldots, n}$,
and every $y \in \mathcal{Y}$,
\begin{align*}
    \widetilde{S}_{\lambda; D^{y}}^{\Gamma, \up}
    \paren{X_i, Y_i}
    &
    = 
    S_{\lambda; D}^{\Gamma}
    \paren{X_i, Y_i}
    +
    \widehat{\tau}_{\lambda}^{\Gamma}\paren{X_i}
    &&
    \leq
    S_{\lambda; D}^{\Gamma}
    \paren{X_i, Y_i}
    + \frac{
        \rho_{p} \norm{\Gamma}_{\mathrm{op}}^{\frac{1}{2}} \paren{\widehat{\kappa}^{\Gamma}}^{2}
    }{2 \lambda\paren{n+1}},
    \mbox{and }
    \\
    \widetilde{S}_{\lambda; D^{y}}^{\Gamma, \up}
    \paren{X_{n+1}, y}
    &
    = 
    S_{\lambda; D}^{\Gamma}
    \paren{X_{n+1}, y}
    + \widehat{\tau}_{\lambda}^{\Gamma}\paren{X_{n+1}}
    &&
    \leq
    S_{\lambda; D}^{\Gamma}
    \paren{X_{n+1}, y}
    + \frac{
        \rho_{p} \norm{\Gamma}_{\mathrm{op}}^{\frac{1}{2}} \paren{\widehat{\kappa}^{\Gamma}}^{2}
    }{2\lambda\paren{n+1}},
\end{align*}
where the stability-bounds $\widehat{\tau}_{\lambda}^{\Gamma}\paren{\bullet}$
is given by Eq.~\eqref{eq.score.stability.bound}
and $\widehat{\kappa}^{\Gamma}$ is given by Eq.~\eqref{eq.kernel.norm.bound}.
Similarly,
the lower-approximate non-conformity scores are bounded from below,
that is, for every $i \in \brac{1, \ldots, n}$,
and every $y \in \mathcal{Y}$,
\begin{align*}
    \widetilde{S}_{\lambda; D^{y}}^{\Gamma, \lo}\paren{X_i, Y_i}
    &
    \geq
    S_{\lambda; D}^{\Gamma}\paren{X_i, Y_i}
    - \frac{
        \rho_{p} \norm{\Gamma}_{\mathrm{op}}^{\frac{1}{2}} \paren{\widehat{\kappa}^{\Gamma}}^{2}
    }{2 \lambda\paren{n+1}},
    \\
    \widetilde{S}_{\lambda; D^{y}}^{\Gamma, \lo}\paren{X_{n+1}, y}
    &
    \geq S_{\lambda; D}^{\Gamma}\paren{X_{n+1}, y}
    - \frac{
        \rho_{p} \norm{\Gamma}_{\mathrm{op}}^{\frac{1}{2}} \paren{\widehat{\kappa}^{\Gamma}}^{2}
    }{2 \lambda\paren{n+1}}.
\end{align*}

\bigskip
\notparagraph{Simplified upper \textbf{StableCP}-region}
The upper \textbf{StableCP}-region
$\widetilde{C}_{\lambda; \alpha}^{\Gamma, \up}\paren{X_{n+1}}$
is contained in a simpler region. In fact,
for every $y \in \mathcal{Y}$,
\begin{align*}
    &
    y \in \widetilde{C}_{\lambda; \alpha}^{\Gamma, \up}\paren{X_{n+1}}
    \\
    &
    \Longleftrightarrow
    \frac{
        1 + \sum_{i = 1}^{n}
        \mathbbm{1}\brac{
            \widetilde{S}_{\lambda; D^{y}}^{\Gamma, \up}
            \paren{X_i, Y_i}
            \geq 
            \widetilde{S}_{\lambda; D^{y}}^{\Gamma, \lo}
            \paren{X_{n+1}, y}
        }
    }{n+1} > \alpha
    \\
    &
    \implies
    \frac{
        1 + \sum_{i = 1}^{n}
        \mathbbm{1}\brac{
            S_{\lambda; D}^{\Gamma}
            \paren{X_i, Y_i}
            + \frac{
                \rho_{p} \norm{\Gamma}_{\mathrm{op}}^{\frac{1}{2}} \paren{\widehat{\kappa}^{\Gamma}}^{2}
            }{2 \lambda\paren{n+1}}
            \geq 
            S_{\lambda; D}^{\Gamma}
            \paren{X_{n+1}, y}
            - \frac{
                \rho_{p} \norm{\Gamma}_{\mathrm{op}}^{\frac{1}{2}} \paren{\widehat{\kappa}^{\Gamma}}^{2}
            }{2 \lambda\paren{n+1}}
        }
    }{n+1} > \alpha
    \\
    &
    \implies
    \frac{
        1 + \sum_{i = 1}^{n}
        \mathbbm{1}\brac{
            S_{\lambda; D}^{\Gamma}
            \paren{X_i, Y_i}
            +
            \frac{
                \rho_{p} \norm{\Gamma}_{\mathrm{op}}^{\frac{1}{2}} \paren{\widehat{\kappa}^{\Gamma}}^{2}
            }{\lambda\paren{n+1}}
            \geq 
            S_{\lambda; D}^{\Gamma}
            \paren{X_{n+1}, y}
        }
    }{n+1} > \alpha
    \\
    &
    \implies
    S_{\lambda; D}^{\Gamma}
    \paren{X_{n+1}, y}
    \leq
    \widehat{Q}_{\lambda; D}^{\Gamma}(\alpha)
    + \frac{
        \rho_{p} \norm{\Gamma}_{\mathrm{op}}^{\frac{1}{2}} \paren{\widehat{\kappa}^{\Gamma}}^{2}
    }{\lambda\paren{n+1}}
    \\
    &
    \implies
    \norm{
        \Gamma^{-\frac{1}{2}}
        \paren{
            y - \hat{f}_{\lambda^{+}; D} \paren{X_{n+1}} 
        }             
    }
    \leq
    \widehat{Q}_{\lambda; D}^{\Gamma}(\alpha)
    + \frac{
        \rho_{p} \norm{\Gamma}_{\mathrm{op}}^{\frac{1}{2}} \paren{\widehat{\kappa}^{\Gamma}}^{2}
    }{\lambda\paren{n+1}},
\end{align*}
where the first implication follows
from the inequalities provided as preliminary properties, and
the third implication holds true by definition of
$\widehat{Q}_{\lambda; D}^{\Gamma}(\alpha)$ in Eq.~\eqref{eq.score.quantile}, and
the last implication follows from the definition
of $S_{\lambda; D}^{\Gamma} \paren{X_{n+1}, y}$.

\bigskip
\notparagraph{Simplified lower \textbf{StableCP}-region}
Similarly, the lower \textbf{StableCP}-region $\widetilde{C}_{\lambda; \alpha}^{\Gamma, \lo}\paren{X_{n+1}}$
contains a simpler region. In fact,
for every $y \in \mathcal{Y}$,
\begin{align*}
    &
    y \in \widetilde{C}_{\lambda; \alpha}^{\Gamma, \lo}\paren{X_{n+1}}
    \\
    &
    \Longleftrightarrow
    \frac{
        1 + \sum_{i = 1}^{n}
        \mathbbm{1}\brac{
            \widetilde{S}_{\lambda; D^{y}}^{\Gamma, \lo}
            \paren{X_i, Y_i}
            \geq 
            \widetilde{S}_{\lambda; D^{y}}^{\Gamma, \up}
            \paren{X_{n+1}, y}
        }
    }{n+1} > \alpha
    \\
    &
    \Longleftarrow 
    \frac{
        1 + \sum_{i = 1}^{n}
        \mathbbm{1}\brac{
            S_{\lambda; D}^{\Gamma}
            \paren{X_i, Y_i}
            - \frac{
                \rho_{p} \norm{\Gamma}_{\mathrm{op}}^{\frac{1}{2}} \paren{\widehat{\kappa}^{\Gamma}}^{2}
            }{2 \lambda\paren{n+1}}
            \geq 
            S_{\lambda; D}^{\Gamma}
            \paren{X_{n+1}, y}
            + \frac{
                \rho_{p} \norm{\Gamma}_{\mathrm{op}}^{\frac{1}{2}} \paren{\widehat{\kappa}^{\Gamma}}^{2}
            }{2 \lambda\paren{n+1}}
        }
    }{n+1} > \alpha
    \\
    &
    \Longleftarrow 
    \norm{
        \Gamma^{-\frac{1}{2}}
        \paren{
            y - \hat{f}_{\lambda^{+}; D} \paren{X_{n+1}}
        }               
    }
    \leq
    \widehat{Q}_{\lambda; D}^{\Gamma}(\alpha)
    - \frac{
        \rho_{p} \norm{\Gamma}_{\mathrm{op}}^{\frac{1}{2}} \paren{\widehat{\kappa}^{\Gamma}}^{2}
    }{\lambda\paren{n+1}}.
\end{align*}
 
\bigskip
\notparagraph{upper-bound on the \emph{thickness}}
Let us consider the case where
\begin{align*}
    \widehat{Q}_{\lambda; D}^{\Gamma}(\alpha)
    - \frac{
        \rho_{p} \norm{\Gamma}_{\mathrm{op}}^{\frac{1}{2}} \paren{\widehat{\kappa}^{\Gamma}}^{2}
    }{\lambda\paren{n+1}} \geq 0.
\end{align*}
Then, the \emph{thickness} $\mathrm{THK}_{\lambda; \alpha}^{\Gamma}\paren{X_{n+1}}$ of
the upper \textbf{StableCP}-region $\widetilde{C}_{\lambda; \alpha}^{\Gamma, \up}\paren{X_{n+1}}$
is bounded from above, that is,
\begin{align*}
    &\mathrm{THK}_{\lambda; \alpha}^{\Gamma}\paren{X_{n+1}}
    \leq \leb{
        \widetilde{C}_{\lambda; \alpha}^{\Gamma, \up}\paren{X_{n+1}}
        \setminus \widetilde{C}_{\lambda; \alpha}^{\Gamma, \lo}\paren{X_{n+1}}
    }
    \\
    &
    \leq \mathcal{V}\left(
        y \in \mathcal{Y}:
        \norm{
            \Gamma^{-\frac{1}{2}}
            \paren{
                y - \hat{f}_{\lambda^{+}; D} \paren{X_{n+1}}     
            }       
        }
        \leq
        \widehat{Q}_{\lambda; D}^{\Gamma}(\alpha)
        + \frac{
            \rho_{p} \norm{\Gamma}_{\mathrm{op}}^{\frac{1}{2}} \paren{\widehat{\kappa}^{\Gamma}}^{2}
        }{\lambda\paren{n+1}},
    \right.
    \\
    &
    \left.
    \qquad
    \qquad
    \qquad
    \quad
    \norm{
        \Gamma^{-\frac{1}{2}}
        \paren{
            y - \hat{f}_{\lambda^{+}; D} \paren{X_{n+1}}     
        }       
    }
    >
    \widehat{Q}_{\lambda; D}^{\Gamma}(\alpha)
    - \frac{
        \rho_{p} \norm{\Gamma}_{\mathrm{op}}^{\frac{1}{2}} \paren{\widehat{\kappa}^{\Gamma}}^{2}
    }{\lambda\paren{n+1}}
    \right)
    \\
    &
    \leq \mathcal{V}\left(
        y \in \mathbb{R}^{p}:
        \norm{
            \Gamma^{-\frac{1}{2}}
            \paren{
                y - \hat{f}_{\lambda^{+}; D} \paren{X_{n+1}}     
            }       
        }
        \leq
        \widehat{Q}_{\lambda; D}^{\Gamma}(\alpha)
        + \frac{
            \rho_{p} \norm{\Gamma}_{\mathrm{op}}^{\frac{1}{2}} \paren{\widehat{\kappa}^{\Gamma}}^{2}
        }{\lambda\paren{n+1}},
    \right.
    \\
    &
    \left.
    \qquad
    \qquad
    \qquad
    \quad
    \norm{
        \Gamma^{-\frac{1}{2}}
        \paren{
            y - \hat{f}_{\lambda^{+}; D} \paren{X_{n+1}}     
        }       
    }
    >
    \widehat{Q}_{\lambda; D}^{\Gamma}(\alpha)
    - \frac{
        \rho_{p} \norm{\Gamma}_{\mathrm{op}}^{\frac{1}{2}} \paren{\widehat{\kappa}^{\Gamma}}^{2}
    }{\lambda\paren{n+1}}
    \right)
    \\
    &
    \leq \leb{
        y \in \mathbb{R}^{p}:
        \norm{
            \Gamma^{-\frac{1}{2}}
            \paren{
                y - \hat{f}_{\lambda^{+}; D} \paren{X_{n+1}}     
            }       
        }
        \leq
        \widehat{Q}_{\lambda; D}^{\Gamma}(\alpha)
        + \frac{
            \rho_{p} \norm{\Gamma}_{\mathrm{op}}^{\frac{1}{2}} \paren{\widehat{\kappa}^{\Gamma}}^{2}
        }{\lambda\paren{n+1}}
    }
    \\
    &
    \quad
    - \leb{
        y \in \mathbb{R}^{p}:
        \norm{
            \Gamma^{-\frac{1}{2}}
            \paren{
                y - \hat{f}_{\lambda^{+}; D} \paren{X_{n+1}}         
            }   
        }
        \leq
        \widehat{Q}_{\lambda; D}^{\Gamma}(\alpha)
        - \frac{
            \rho_{p} \norm{\Gamma}_{\mathrm{op}}^{\frac{1}{2}} \paren{\widehat{\kappa}^{\Gamma}}^{2}
        }{\lambda\paren{n+1}}
    }
    \\
    &
    \leq
    \abss{\mathrm{det}\paren{\Gamma^{\frac{1}{2}}}}
    \frac{\pi^{\frac{p}{2}}}{\paren{\frac{p}{2}}!}
    \paren{
        \widehat{Q}_{\lambda; D}^{\Gamma}(\alpha)
        + \frac{
            \rho_{p} \norm{\Gamma}_{\mathrm{op}}^{\frac{1}{2}} \paren{\widehat{\kappa}^{\Gamma}}^{2}
        }{\lambda\paren{n+1}}
    }^{p}
    -\abss{\mathrm{det}\paren{\Gamma^{\frac{1}{2}}}}
    \frac{\pi^{\frac{p}{2}}}{\paren{\frac{p}{2}}!}
    \paren{
        \widehat{Q}_{\lambda; D}^{\Gamma}(\alpha)
        - \frac{
            \rho_{p} \norm{\Gamma}_{\mathrm{op}}^{\frac{1}{2}} \paren{\widehat{\kappa}^{\Gamma}}^{2}
        }{\lambda\paren{n+1}}
    }^{p}
    \\
    &
    \leq
    \abss{\mathrm{det}\paren{\Gamma^{\frac{1}{2}}}}
    \frac{\pi^{\frac{p}{2}}}{\paren{\frac{p}{2}}!}
    \croch{
        \paren{
            \widehat{Q}_{\lambda; D}^{\Gamma}(\alpha)
            + \frac{
                \rho_{p} \norm{\Gamma}_{\mathrm{op}}^{\frac{1}{2}} \paren{\widehat{\kappa}^{\Gamma}}^{2}
            }{\lambda\paren{n+1}}
        }^{p}
        - \paren{
            \widehat{Q}_{\lambda; D}^{\Gamma}(\alpha)
            - \frac{
                \rho_{p} \norm{\Gamma}_{\mathrm{op}}^{\frac{1}{2}} \paren{\widehat{\kappa}^{\Gamma}}^{2}
            }{\lambda\paren{n+1}}
        }^{p}
    }
    \\
    &
    \leq 2 \frac{
        \norm{\Gamma}_{\mathrm{op}}^{\frac{1}{2}}
        \paren{\widehat{\kappa}^{\Gamma}}^{2}
        \rho_{p}
    }{\lambda\paren{n+1}}
    \abss{\mathrm{det}\paren{\Gamma^{\frac{1}{2}}}}
    \frac{\pi^{\frac{p}{2}}}{\paren{\frac{p}{2}}!}
    p \paren{
        \widehat{Q}_{\lambda; D}^{\Gamma}(\alpha)
        + \frac{
            \rho_{p} \norm{\Gamma}_{\mathrm{op}}^{\frac{1}{2}} \paren{\widehat{\kappa}^{\Gamma}}^{2}
        }{\lambda\paren{n+1}}
    }^{p-1}
\end{align*}
where the first inequality follows from Lemma~\ref{lm.sandwiching}, and
the second, from the simpler regions provided before, and
the fifth, from a change of variable $u = \Gamma^{-\frac{1}{2}}
\paren{
    y - \hat{f}_{\lambda^{+}; D} \paren{X_{n+1}}         
}$, under the assumption that
$\widehat{Q}_{\lambda; D}^{\Gamma}(\alpha) - \frac{
    \rho_{p} \norm{\Gamma}_{\mathrm{op}}^{\frac{1}{2}} \paren{\widehat{\kappa}^{\Gamma}}^{2}
}{\lambda\paren{n+1}} \geq 0$ and the expression of the volume a $p$-ball, and
the last, from the mean value theorem
applied to the function $x \in \mathbb{R} \mapsto x^p  \in \mathbb{R}$
over the interval $\croch{
    \widehat{Q}_{\lambda; D}^{\Gamma}(\alpha)
    - \frac{
        \rho_{p} \norm{\Gamma}_{\mathrm{op}}^{\frac{1}{2}} \paren{\widehat{\kappa}^{\Gamma}}^{2}
    }{\lambda\paren{n+1}},
    \widehat{Q}_{\lambda; D}^{\Gamma}(\alpha)
    + \frac{
        \rho_{p} \norm{\Gamma}_{\mathrm{op}}^{\frac{1}{2}} \paren{\widehat{\kappa}^{\Gamma}}^{2}
    }{\lambda\paren{n+1}}
}$.

Let us now consider the case where
\begin{align*}
    \widehat{Q}_{\lambda; D}^{\Gamma}(\alpha)
    - \frac{
        \rho_{p} \norm{\Gamma}_{\mathrm{op}}^{\frac{1}{2}} \paren{\widehat{\kappa}^{\Gamma}}^{2}
    }{\lambda\paren{n+1}} < 0.
\end{align*}
Then, the \emph{thickness} $\mathrm{THK}_{\lambda; \alpha}^{\Gamma}\paren{X_{n+1}}$ of
the upper \textbf{StableCP}-region $\widetilde{C}_{\lambda; \alpha}^{\Gamma, \up}\paren{X_{n+1}}$
is bounded from above, that is,
\begin{align*}
    &
    \mathrm{THK}_{\lambda; \alpha}^{\Gamma}\paren{X_{n+1}}
    \\
    &
    \leq \leb{
        \widetilde{C}_{\lambda; \alpha}^{\Gamma, \up}\paren{X_{n+1}}
        \setminus \widetilde{C}_{\lambda; \alpha}^{\Gamma, \lo}\paren{X_{n+1}}
    }
    \\
    &
    \leq \leb{
        y \in \mathbb{R}^{p}:
        \norm{
            \Gamma^{-\frac{1}{2}}
            \paren{
                y - \hat{f}_{\lambda^{+}; D} \paren{X_{n+1}}            
            } 
        }
        \leq
        \widehat{Q}_{\lambda; D}^{\Gamma}(\alpha)
        + \frac{
            \rho_{p} \norm{\Gamma}_{\mathrm{op}}^{\frac{1}{2}} \paren{\widehat{\kappa}^{\Gamma}}^{2}
        }{\lambda\paren{n+1}}
    }
    \\
    &
    \leq
    \abss{\mathrm{det}\paren{\Gamma^{\frac{1}{2}}}}
    \frac{\pi^{\frac{p}{2}}}{\paren{\frac{p}{2}}!}
    \paren{
        \widehat{Q}_{\lambda; D}^{\Gamma}(\alpha)
        + \frac{
            \rho_{p} \norm{\Gamma}_{\mathrm{op}}^{\frac{1}{2}} \paren{\widehat{\kappa}^{\Gamma}}^{2}
        }{\lambda\paren{n+1}}
    }^{p}
    \\
    &
    \leq 
    \frac{
        \rho_{p} \norm{\Gamma}_{\mathrm{op}}^{\frac{1}{2}} \paren{\widehat{\kappa}^{\Gamma}}^{2}
    }{\lambda\paren{n+1}}
    \abss{\mathrm{det}\paren{\Gamma^{\frac{1}{2}}}}
    \frac{\pi^{\frac{p}{2}}}{\paren{\frac{p}{2}}!}
    p
    \paren{
        \widehat{Q}_{\lambda; D}^{\Gamma}(\alpha)
        + \frac{
            \rho_{p} \norm{\Gamma}_{\mathrm{op}}^{\frac{1}{2}} \paren{\widehat{\kappa}^{\Gamma}}^{2}
        }{\lambda\paren{n+1}}
    }^{p-1}
\end{align*}
where the first inequality follows from Lemma~\ref{lm.sandwiching}, and 
the second, from the simpler region provided previously, and
the third, from a change of variable $u = \Gamma^{-\frac{1}{2}}
\paren{
    y - \hat{f}_{\lambda^{+}; D} \paren{X_{n+1}}         
}$ and the expression of the volume a $p$-ball, and
the last, from
the mean value theorem applied to the function $x \in \mathbb{R} \mapsto x^{p} \in \mathbb{R}$
over the interval
$\croch{\widehat{Q}_{\lambda; D}^{\Gamma}(\alpha),
\widehat{Q}_{\lambda; D}^{\Gamma}(\alpha)
+ \frac{
    \rho_{p} \norm{\Gamma}_{\mathrm{op}}^{\frac{1}{2}} \paren{\widehat{\kappa}^{\Gamma}}^{2}
}{\lambda\paren{n+1}}}$.
\end{proof}

\subsection{Deriving an estimation error bound}
The next definition formulates a complexity measure of a class of functions,
that is the Rademacher complexity, which is classical used to derive deviation bounds
for data dependent quantities.

\bigskip
\notparagraph{Rademacher complexity}
Let $b \in \paren{0, +\infty}$
designate a scalar,
and $\mathcal{\mathcal{F}}_{\leq b} \subseteq \mathcal{H}$,
the space of function, given by
\begin{align*}
    \mathcal{F}_{\leq b} := \brac{
        \paren{x, y} \mapsto \ell\paren{y, f(x)}
        :f \in \mathcal{H}, \normh{f} \leq b 
    },
\end{align*}
and $\mathrm{R}_n\paren{\mathcal{F}_{\leq b}}$,
the Rademacher complexity
of the hypothesis space $\mathcal{F}_{\leq b}$,
given by 
\begin{align}
    \label{eq.rademacher.bounded.function}
    \mathrm{R}_n\paren{\mathcal{F}_{\leq b}}
    := \mathbb{E}\brac{
        \sup_{h \in \mathcal{F}_{\leq b}}
        \frac{1}{n}
        \sum_{i=1}^{n}
        \epsilon_i
        h\paren{X_i, Y_i}
    },
\end{align}
where $\epsilon_1, \ldots, \epsilon_n$
are independent Rademacher random variables,
that is $\mathbb{P}_{\epsilon_1} = \frac{1}{2}\delta_{-1} + \frac{1}{2}\delta_{1}$
(where $\delta_{\bullet}$ name the Dirac measure).
\begin{lemma}
\label{lm.bounded.trace}
For every $A, B \in \mathrm{Sym}_{p}^{+}\paren{\mathbb{R}}$,
\begin{align*}
    \mathrm{tr}\paren{A B} \leq \norm{A}_{\mathrm{op}} \mathrm{tr}\paren{B}.
\end{align*}    
\end{lemma}
\begin{proof}
    Let $A, B \in \mathrm{Sym}_{p}^{+}\paren{\mathbb{R}}$.
    Since $A \in \mathrm{Sym}_{p}^{+}\paren{\mathbb{R}}$,
    there exists an orthogonal matrix $U \in \mathbb{R}^{p \times p}$
    and a diagonal matrix $D \in \mathbb{R}^{p \times p}$,
    such that $A = UDU^{T}$, and
    $\norm{A}_{\mathrm{op}} = D_{1, 1} \geq \ldots \geq  D_{p, p} \geq 0$.
    Moreover, since $B \in \mathrm{Sym}_{p}^{++}\paren{\mathbb{R}}$,
    for every $l \in \brac{1, \ldots, p}$
    \begin{align*}
        \paren{U^{T}B U}_{l, l}
        = e_l^{T} U^{T} B U e_l
        = \paren{U e_l}^{T} B \paren{U e_l} \geq 0.
    \end{align*}
    It follows that
    \begin{align*}
        \mathrm{tr}\paren{AB}
        &
        = \mathrm{tr}\paren{UDU^{T}B}
        = \mathrm{tr}\paren{DU^{T}BU}
        \\
        &
        = \sum_{l = 1}^{p}
        \paren{
            DU^{T}BU
        }_{l, l}
        = \sum_{l = 1}^{p}
        \sum_{s = 1}^{p}
        D_{l, s} \paren{U^{T}B U}_{s, l}
        = \sum_{l = 1}^{p}
        D_{l, l} \paren{U^{T}B U}_{l, l}
        \\
        &
        \leq
        D_{1, 1} 
        \sum_{l = 1}^{p}
        \paren{U^{T}B U}_{l, l}
        \leq
        \norm{A}_{\mathrm{op}} \mathrm{tr}\paren{U^{T} B U}
        \leq
        \norm{A}_{\mathrm{op}} \mathrm{tr}\paren{B UU^{T}}
        \leq
        \norm{A}_{\mathrm{op}} \mathrm{tr}\paren{B}.
    \end{align*}
\end{proof}

{The next result provides an upper-bound on the Rademacher complexity
given by Eq.~\ref{eq.rademacher.bounded.function} involving quantities that
are assumed to be known.
}
\begin{lemma}
    \label{lm.bound.rademacher.complexity}
    Assume \eqref{asm.lip.loss} and
    \eqref{asm.bounded.gamma.kernel}
    hold true.
    Then, for every $b \in \paren{0, +\infty}$,
    the Rademacher complexity
    $\mathrm{R}_n\paren{\mathcal{F}_{\leq b}}$
    (see Eq.~\ref{eq.rademacher.bounded.function})
    is bounded from above, that is, 
    \begin{align*}
        \mathrm{R}_n\paren{\mathcal{F}_{\leq b}}
        \leq
        \sqrt{2} b
        \frac{ \rho_{p} \mathrm{tr}\paren{\Gamma}^{\frac{1}{2}} \kappa^{\Gamma}}{\sqrt{n}}.
    \end{align*}
\end{lemma}
\begin{proof}
    This proof is an adaptation of \citet[in Section 4.5.3]{bachLearningTheoryFirst}.
    Let $b \in \paren{0, +\infty}$ denote a bound.
    Let us consider the following term
    \begin{align*}
        \normh{
            \sum_{i=1}^{n}
            \sum_{l = 1}^{p}
            \epsilon_{i, l}
            K\paren{\bullet, X_i} e_l
        }^2
        &
        = \dotprod{
            \sum_{i=1}^{n}
            \sum_{l = 1}^{p}
            \epsilon_{i, l}
            K\paren{\bullet, X_i} e_l
        }{
            \sum_{j=1}^{n}
            \sum_{s = 1}^{p}
            \epsilon_{j, s}
            K\paren{\bullet, X_j} e_s
        }{\mathcal{H}}
        \\
        &
        =
        \sum_{i=1}^{n}
        \sum_{l = 1}^{p}
        \sum_{j=1}^{n}
        \sum_{s = 1}^{p}
        \epsilon_{i, l}
        \epsilon_{j, s}
        \dotprod{
            K\paren{\bullet, X_i} e_l
        }{
            K\paren{\bullet, X_j} e_s
        }{\mathcal{H}}
        \\
        &
        =
        \sum_{i=1}^{n}
        \sum_{l = 1}^{p}
        \sum_{j=1}^{n}
        \sum_{s = 1}^{p}
        \epsilon_{i, l}
        \epsilon_{j, s}
        \dotprod{
            e_s
        }{
            K\paren{X_j, X_i} e_l
        }{}
        \\
        &
        =
        \sum_{i=1}^{n}
        \sum_{j=1}^{n}
        \sum_{l = 1}^{p}
        \sum_{s = 1}^{p}
        \epsilon_{j, s}
        K\paren{X_j, X_i}_{s,l}
        \epsilon_{i, l}
        \\
        &
        =
        \sum_{i=1}^{n}
        \sum_{j=1}^{n}
        \epsilon_{j, \bullet}^{T}
        K\paren{X_j, X_i}
        \epsilon_{i, \bullet}
        \\
        &
        =
        \epsilon^T
        \sum_{i=1}^{n}
        \sum_{j=1}^{n}
        \paren{e_j \otimes I_{p}}
        K\paren{X_j, X_i}
        \paren{e_i^{T} \otimes I_{p}}
        \epsilon
        \\
        &
        =
        \epsilon^T
        \mathbf{K}_{\mathbf{X}}
        \epsilon
        =
        \norm{\croch{\mathbf{K}_{\mathbf{X}}}^{\frac{1}{2}} \epsilon}^2
        \\
        &
        =
        \mathrm{tr}\paren{
            \croch{\mathbf{K}_{\mathbf{X}}}^{\frac{1}{2}} \epsilon
            \epsilon^T \croch{\mathbf{K}_{\mathbf{X}}}^{\frac{1}{2}}
        }
        =\mathrm{tr}\paren{
            \mathbf{K}_{\mathbf{X}}
            \epsilon\epsilon^T
        },
    \end{align*}
    where third equality follows from
    \citet[in Proposition 2.1.a]{micchelliLearningVectorvaluedFunctions2005}, and
    the fifth equality holds with $\epsilon_{j, \bullet} = \paren{\epsilon_{j,1}, \ldots, \epsilon_{j,p}} \in \mathbb{R}^{p}$
    for every $j \in \brac{1, \ldots, n}$,
    the sixth equality holds with $\epsilon = \begin{pmatrix}
        \epsilon_{1, \bullet} & \ldots & \epsilon_{n, \bullet} 
    \end{pmatrix}^{T} \in \mathbb{R}^{np}$,
    and the seventh equality follows from
    the definition of the Gram matrix $\mathbf{K}_{\mathbf{X}} \in \mathbb{R}^{np \times np}$
    in Eq.~\eqref{eq.gram.matrix}
    which is a positive semi-definite matrix by
    \citet[in Proposition 2.1.c]{micchelliLearningVectorvaluedFunctions2005}.
    
    Then, $\mathrm{R}_n\paren{\mathcal{F}_{\leq b}}$
    is bounded from above, that is,
    \begin{align*}
        \mathrm{R}_n\paren{\mathcal{F}_{\leq b}}
        &
        = \mathbb{E}\brac{
            \sup_{h \in \mathcal{F}_{\leq b}}
            \frac{1}{n}
            \epsilon \sum_{i=1}^{n}
            h\paren{X_i, Y_i}
        }
        \\
        &
        \leq
        \sqrt{2}
        \rho_{p}
        \mathbb{E}\croch{
            \sup_{\substack{
                f \in \mathcal{H}\\
                \normh{f} \leq b
            }}
            \frac{1}{n}
            \sum_{i=1}^{n}
            \sum_{l = 1}^{p}
            \epsilon_{i, l}
            f_l\paren{X_i}
        }
        \\
        &
        \leq
        \sqrt{2}
        \frac{
        \rho_{p}}{n}
        \mathbb{E}\croch{
            \sup_{\substack{
                f \in \mathcal{H}\\
                \normh{f} \leq b
            }}
            \sum_{i=1}^{n}
            \sum_{l = 1}^{p}
            \epsilon_{i, l}
            \dotprod{e_l}{f\paren{X_i}}{}
        }
        \\
        &
        \leq
        \sqrt{2}
        \frac{
        \rho_{p}}{n}
        \mathbb{E}\croch{
            \sup_{\substack{
                f \in \mathcal{H}\\
                \normh{f} \leq b
            }}
            \sum_{i=1}^{n}
            \sum_{l = 1}^{p}
            \epsilon_{i, l}
            \dotprod{K\paren{\bullet, X_i} e_l}{f}{\mathcal{H}}
        }
        \\
        &
        \leq
        \sqrt{2}
        \frac{
        \rho_{p}}{n}
        \mathbb{E}\croch{
            \sup_{\substack{
                f \in \mathcal{H}\\
                \normh{f} \leq b
            }}
            \dotprod{
                \sum_{i=1}^{n}
                \sum_{l = 1}^{p}
                \epsilon_{i, l}
                K\paren{\bullet, X_i} e_l
            }{f}{\mathcal{H}}
        }
        \\
        &
        \leq
        \sqrt{2} b
        \frac{
        \rho_{p}}{n}
        \mathbb{E}\croch{
            \normh{
                \sum_{i=1}^{n}
                \sum_{l = 1}^{p}
                \epsilon_{i, l}
                K\paren{\bullet, X_i} e_l
            }
        }
        \\
        &
        \leq
        \sqrt{2} b
        \frac{
        \rho_{p}}{n}
        \sqrt{
            \mathbb{E}\croch{
            \normh{
                \sum_{i=1}^{n}
                \sum_{l = 1}^{p}
                \epsilon_{i, l}
                K\paren{\bullet, X_i} e_l
            }^2
            }
        }
        \\
        &
        \leq
        \sqrt{2} b
        \frac{
        \rho_{p}}{n}
        \sqrt{
            \mathbb{E}\croch{
            \mathrm{tr}\paren{
                \mathbf{K}_{\mathbf{X}}\epsilon\epsilon^T
            }
            }
        }
        \\
        &
        \leq
        \sqrt{2} b
        \frac{
        \rho_{p}}{n}
        \sqrt{
            \mathbb{E}\croch{
            \mathrm{tr}\paren{
                \mathbf{K}_{\mathbf{X}}
                }
            }
        }
        \\
        &
        \leq
        \sqrt{2} b
        \frac{\rho_{p}}{\sqrt{n}}
        \sqrt{
            \frac{1}{n}
            \sum_{i=1}^{n}
            \mathbb{E}\croch{
                \mathrm{tr}\paren{
                    K\paren{X_i, X_i}
                }
            }
        }
        \\
        &
        \leq
        \sqrt{2} b
        \frac{\rho_{p}}{\sqrt{n}}
        \sqrt{
            \frac{1}{n}
            \sum_{i=1}^{n}
            \mathbb{E}\croch{
                \mathrm{tr}\paren{
                    \Gamma^{-\frac{1}{2}} K\paren{X_i, X_i} \Gamma^{-\frac{1}{2}} \Gamma
                }
            }
        }
        \\
        &
        \leq
        \sqrt{2} b
        \frac{\rho_{p}}{\sqrt{n}}
        \sqrt{
            \frac{1}{n}
            \sum_{i=1}^{n}
            \mathbb{E}\croch{
                \norm{\Gamma^{-\frac{1}{2}} K\paren{X_i, X_i} \Gamma^{-\frac{1}{2}}}_{\mathrm{op}}
                \mathrm{tr}\paren{
                     \Gamma
                }
            }
        }
        \\
        &
        \leq
        \sqrt{2} b
        \frac{ \rho_{p} \mathrm{tr}\paren{\Gamma}^{\frac{1}{2}} \kappa^{\Gamma}}{\sqrt{n}},
    \end{align*}
    where the first inequality follows from
    the vector-contraction inequality
    \citep[see Corollary 4]{maurerVectorcontractionInequalityRademacher2016}
    which holds true under \eqref{asm.lip.loss}, and
    the third, from
    the reproducing property of the hypothesis space
    $\mathcal{H}$, and
    the fifth, from
    Cauchy-Schwarz inequality, and
    the sixth, from
    Jensen's inequality, and
    the eight, from
    the linearity of the expectation
    and the trace, the independence
    of the Rademacher vector $\epsilon$ and
    the data set $D$, and
    the fact that $\mathbb{E}\croch{\epsilon\epsilon^{T}} = I_{np}$, and
    the eleventh, from Lemma~\ref{lm.bounded.trace}
    since for every $x \in \mathcal{X}$,
    $\Gamma^{-\frac{1}{2}}K\paren{x, x} \Gamma^{-\frac{1}{2}}, \Gamma \in \mathrm{Sym}_{p}^{++}\paren{\mathbb{R}}$, and
    the last inequality holds true under \eqref{asm.bounded.gamma.kernel}.
\end{proof}

This bound is akin to the one provided by \cite{bachLearningTheoryFirst}
but in the more general setting of vector-valued outputs.
Hence, the effect of the dimension is reflected in the term $\sqrt{2}\mathrm{tr}\paren{\Gamma}$.
It is important to observe that this effect goes through the inter-task covariance-matrix
and not directly from the ambiant dimension $p$.

\begin{lemma}
    \label{lm.bounded.kernel}
    Assume \eqref{asm.bounded.gamma.kernel} hold true.
    Then, for every $x \in \mathcal{X}$,
    \begin{align*}
        \norm{K\paren{\bullet, x}}_{\mathrm{op}}
        = \norm{K\paren{x, x}}_{\mathrm{op}}^{\frac{1}{2}}
        \leq \norm{\Gamma}_{\mathrm{op}}^{\frac{1}{2}} \kappa^{\Gamma}.
    \end{align*}
\end{lemma}
\begin{proof}
    Let $x \in \mathcal{X}$,
    \begin{align*}
        \norm{K\paren{\bullet, x}}_{\mathrm{op}}
        &
        = \norm{K\paren{x, x}}_{\mathrm{op}}^{\frac{1}{2}}
        \\
        &
        = \norm{
            \Gamma^{\frac{1}{2}} \Gamma^{-\frac{1}{2}}
            K\paren{x, x}
            \Gamma^{-\frac{1}{2}} \Gamma^{\frac{1}{2}}
        }_{\mathrm{op}}^{\frac{1}{2}}
        \\
        &
        \leq
        \norm{\Gamma^{\frac{1}{2}}}_{\mathrm{op}}^{\frac{1}{2}}
        \norm{
            \Gamma^{-\frac{1}{2}}
            K\paren{x, x}
            \Gamma^{-\frac{1}{2}}
        }_{\mathrm{op}}^{\frac{1}{2}}
        \norm{\Gamma^{\frac{1}{2}}}_{\mathrm{op}}^{\frac{1}{2}}
        \\
        &
        \leq
        \norm{\Gamma}_{\mathrm{op}}^{\frac{1}{2}}
        \kappa^{\Gamma},
    \end{align*}
    where the first equality follows from
    \citet[in Proposition 2.1.d]{micchelliLearningVectorvaluedFunctions2005}, and
    the first inequality follows from the submultiplicativity of the operator norm, and
    the last inequality holds true under \eqref{asm.bounded.gamma.kernel}.
\end{proof}

{Harnessing the upper-bound on the Rademacher complexity,
the next result provides an upper-bound on the regularized risk of
the predictor $\hat{f}_{\lambda; D}$
provided that of its population counterpart $f_{\lambda}$.
}
\begin{lemma}
    \label{lm.estimation.error.bound}
    Assume \eqref{asm.conv.loss},
    \eqref{asm.lip.loss}, \eqref{asm.lower.bounded.loss} and
    \eqref{asm.bounded.gamma.kernel} hold true.
    Then, for any $\delta \in \paren{0, 1}$,
    with probability greater than $1 - \delta$,
    \begin{align*}
        \mathbf{R}_{\lambda}\paren{\hat{f}_{\lambda; D}}
        \leq \mathbf{R}_{\lambda}\paren{f_{\lambda}}
        + \frac{
            \rho_{p}^2 \norm{\Gamma}_{\mathrm{op}} \paren{\kappa^{\Gamma}}^{2}
        }{
            \lambda n
        }
        \paren{
            2\sqrt{2}\paren{
                \frac{
                    \mathrm{tr}\paren{\Gamma}
                }{
                    \norm{\Gamma}_{\mathrm{op}}
                }
            }^{\frac{1}{2}}
            + \sqrt{2 \log \frac{1}{\delta}}
        }^2,
    \end{align*}
    where $\hat{f}_{\lambda; D}$ is given by Eq.~\eqref{def.predictor}
    and $f_{\lambda}$ is given in Lemma~\ref{lm.reg.risk.minimizer},
\end{lemma}
\begin{proof}
    The following proof is an adaptation of
    \citet[see Proposition 4.6]{bachLearningTheoryFirst}.
    Let $\delta \in \paren{0, 1}$
    stand for a risk level,
    and for every $i \in \brac{1, \ldots, n}$,
    $D^{i}$, the data set given by
    \begin{align}
        \label{eq.modified.data.set}
        D^{i} := \brac{
            \paren{X_1, Y_1}, \ldots,\paren{X_{i-1}, Y_{i-1}},\paren{X_i^{\prime}, Y_i^{\prime}}, \paren{X_{i+1}, Y_{i+1}}, \ldots \paren{X_n, Y_n}
        } 
    \end{align}

    Since \eqref{asm.lip.loss} implies \eqref{asm.lsc.loss}.
    Then, under \eqref{asm.conv.loss}, \eqref{asm.lip.loss},
    and \eqref{asm.lower.bounded.loss},
    Lemma~\ref{lm.predictor.well.defined}
    states that $\widehat{\mathbf{R}}_{\lambda; D}\paren{\bullet}$
    admits a unique minimizer $\hat{f}_{\lambda; D} \in \mathcal{H}$, and
    Lemma~\ref{lm.min.norm.risk.minimizer}
    states that $\mathbf{R}_{\lambda}\paren{\bullet}$
    admits a unique minimizer $f_{\lambda} \in \mathcal{H}$.

    Let $\varepsilon \in \paren{0, +\infty}$ and let us consider the set
    \begin{align*}
        C_{\varepsilon}
        := \brac{
            f \in \mathcal{H}:
            \mathbf{R}_{\lambda}\paren{f}
            - \mathbf{R}_{\lambda}\paren{f_{\lambda}}
            \leq \varepsilon
        } \subseteq \mathcal{H}.
    \end{align*}
    Since \eqref{asm.conv.loss} holds true,
    then, this set is convex.
    Moreover, since \eqref{asm.lip.loss} holds true,
    then, this set is closed.
    Since \eqref{asm.conv.loss} holds true, 
    $\mathbf{R}_{\lambda}\paren{\bullet}$
    is $2\lambda$-strongly convex, and
    by Lemma~\ref{lm.minimum.gap},
    for every $f \in \mathcal{H}$,
    $\mathbf{R}_{\lambda}\paren{f} - \mathbf{R}_{\lambda}\paren{f_{\lambda}} \geq \lambda \normh{f - f_{\lambda}}^2$.
    It follows that $C_{\varepsilon}\subseteq \mathcal{H}$ is included in
    the ball $B_{\varepsilon}\subseteq \mathcal{H}$, centred
    around $f_{\lambda}$ with radius
    $\sqrt{\frac{\varepsilon}{\lambda}}$, that is,
    \begin{align*}
        C_{\varepsilon}\subseteq B_{\varepsilon}
        := \brac{
            f \in \mathcal{H} : 
            \normh{f - f_{\lambda}} \leq \sqrt{\frac{\varepsilon}{\lambda}}
        }.
    \end{align*}
    Since \eqref{asm.lip.loss} and \eqref{asm.bounded.gamma.kernel} hold true,
    Lemma~\ref{lm.bound.rademacher.complexity}
    and \citet[see Proposition 4.2]{bachLearningTheoryFirst}
    imply that
    \begin{align*}
        \mathbb{E}\croch{
            \sup_{\xi \in B_{\epsilon}}
            \brac{
                \mathbf{R}_{0}\paren{\xi}
                - \widehat{\mathbf{R}}_{0; D}\paren{\xi}
            }
        }
        \leq 2 \mathrm{R}_n\paren{\mathcal{F}_{\leq \sqrt{\frac{\varepsilon}{\lambda}}}}
        \leq 2 \sqrt{2} \frac{\rho_p \mathrm{tr}\paren{\Gamma}^{\frac{1}{2}} \kappa^{\Gamma}}{\sqrt{n}} \sqrt{\frac{\varepsilon}{\lambda}}.
    \end{align*}
    Let $A_{\varepsilon; D}$ designate
    the scalar-valued random variable,
    given by
    \begin{align*}
        A_{\varepsilon; D} := \sup_{\xi \in B_{\varepsilon}}
        \brac{
            \paren{
                \mathbf{R}_{\lambda}\paren{\xi}
                - \mathbf{R}_{\lambda}\paren{f_{\lambda}}
            }
            - \paren{
                \widehat{\mathbf{R}}_{\lambda; D}\paren{\xi}
                - \widehat{\mathbf{R}}_{\lambda; D}\paren{f_{\lambda}}
            }
        }.
    \end{align*}
    The expectation of $A_{\varepsilon; D}$ is bounded from above, that is,
    \begin{align*}
        \mathbb{E}\croch{A_{\varepsilon; D}}
        &
        = \mathbb{E}\croch{
            \sup_{\xi \in B_{\varepsilon}}
            \brac{
                \paren{
                    \mathbf{R}_{\lambda}\paren{\xi}
                    - \mathbf{R}_{\lambda}\paren{f_{\lambda}}
                }
                - \paren{
                    \widehat{\mathbf{R}}_{\lambda; D}\paren{\xi}
                    - \widehat{\mathbf{R}}_{\lambda; D}\paren{f_{\lambda}}
                }
            }
        }
        \\
        &
        \leq \mathbb{E}\croch{
            \sup_{\xi \in B_{\varepsilon}}
            \brac{
                \mathbf{R}_{\lambda}\paren{\xi}
                - \widehat{\mathbf{R}}_{\lambda; D}\paren{\xi}
            }
        }
        - \paren{
            \mathbf{R}_{\lambda}\paren{f_{\lambda}}
            - \mathbb{E}\croch{\widehat{\mathbf{R}}_{\lambda; D}\paren{f_{\lambda}}}
        }
        \\
        &
        \leq
        \mathbb{E}\croch{
            \sup_{\xi \in B_{\varepsilon}}
            \brac{
                \mathbf{R}_{0}\paren{\xi}
                + \lambda \normh{\xi}^2
                - \widehat{\mathbf{R}}_{0; D}\paren{\xi}
                - \lambda \normh{\xi}^2
            }
        }
        - \paren{
            \mathbf{R}_{\lambda}\paren{f_{\lambda}}
            - \mathbb{E}\croch{\widehat{\mathbf{R}}_{\lambda; D}\paren{f_{\lambda}}}
        }
        \\
        &
        \leq
        \mathbb{E}\croch{
            \sup_{\xi \in B_{\varepsilon}}
            \brac{
                \mathbf{R}_{0}\paren{\xi}
                - \widehat{\mathbf{R}}_{0; D}\paren{\xi}
            }
        }
        - \paren{
            \mathbf{R}_{\lambda}\paren{f_{\lambda}}
            - \mathbb{E}\croch{\widehat{\mathbf{R}}_{\lambda; D}\paren{f_{\lambda}}}
        }
        \\
        &
        \leq
        2 \sqrt{2} \frac{\rho_p \mathrm{tr}\paren{\Gamma}^{\frac{1}{2}} \kappa^{\Gamma}}{\sqrt{n}} \sqrt{\frac{\varepsilon}{\lambda}} - 0.
    \end{align*}
    Moreover, for every $i \in \brac{1, \ldots, n}$,
    \begin{align*}
        \abss{
            A_{\varepsilon; D} - A_{\varepsilon; D^{i}}
        }
        &
        \leq
        \sup_{\xi \in B_{\varepsilon}}
        \abss{
            \croch{
                \widehat{\mathbf{R}}_{\lambda; D}\paren{\xi}
                - \widehat{\mathbf{R}}_{\lambda; D}\paren{f_{\lambda}}
            } - \croch{
                \widehat{\mathbf{R}}_{\lambda; D^{i}}\paren{\xi}
                - \widehat{\mathbf{R}}_{\lambda; D^{i}}\paren{f_{\lambda}}
            }
        }
        \\
        &
        \leq
        \sup_{\xi \in B_{\varepsilon}}
        \abss{
            \croch{
                \widehat{\mathbf{R}}_{\lambda; D}\paren{\xi}
                - \widehat{\mathbf{R}}_{\lambda; D^{i}}\paren{\xi}
            } - \croch{
                \widehat{\mathbf{R}}_{\lambda; D}\paren{f_{\lambda}}
                - \widehat{\mathbf{R}}_{\lambda; D^{i}}\paren{f_{\lambda}}
            }
        }
        \\
        &
        \leq
        \sup_{\xi \in B_{\varepsilon}}
        \abss{
            \frac{1}{n}
            \croch{
                \ell\paren{Y_i, \xi\paren{X_i}}
                - \ell\paren{Y_i^{\prime}, \xi\paren{X_i^{\prime}}}
            }
            - \frac{1}{n}
            \croch{
                \ell\paren{Y_i, f_{\lambda}\paren{X_i}}
                - \ell\paren{Y_i^{\prime}, f_{\lambda}\paren{X_i^{\prime}}}
            }
        }
        \\
        &
        \leq
        \frac{1}{n}
        \sup_{\xi \in B_{\varepsilon}}
        \abss{
            \croch{
                \ell\paren{Y_i, \xi\paren{X_i}}
                - \ell\paren{Y_i, f_{\lambda}\paren{X_i}}
            }
            - \croch{
                \ell\paren{Y_i^{\prime}, \xi\paren{X_i^{\prime}}}
                - \ell\paren{Y_i^{\prime}, f_{\lambda}\paren{X_i^{\prime}}}
            }
        }
        \\
        &
        \leq
        \frac{\rho_{p}}{n}
        \sup_{\xi \in B_{\varepsilon}}
        \paren{
            \norm{K\paren{\bullet, X_i}}_{\mathrm{op}}
            + \norm{K\paren{\bullet, X_i^{\prime}}}_{\mathrm{op}}
        }\normh{\xi - f_{\lambda}}
        \\
        &
        \leq
        2\frac{\rho_{p}\norm{\Gamma}_{\mathrm{op}}^{\frac{1}{2}} \kappa^{\Gamma}}{n}
        \sqrt{\frac{\varepsilon}{\lambda}},
    \end{align*}
    where the fourth inequality follows from \eqref{asm.lip.loss}
    the reproducing property of the hypothesis space $\mathcal{H}$, and
    the last inequality follows from Lemma~\ref{lm.bounded.kernel}
    which holds true under \eqref{asm.bounded.gamma.kernel}.
    By McDiarmid's inequality
    \citep[see Proposition 1.3]{bachLearningTheoryFirst},
    \begin{align*}
        1 - \delta
        &
        \leq \mathbb{P}\croch{
            A_{\varepsilon; D}
            \leq
            \mathbb{E}\croch{A_{\varepsilon; D}}
            + 2\frac{\rho_{p}\norm{\Gamma}_{\mathrm{op}}^{\frac{1}{2}} \kappa^{\Gamma}}{n}
            \sqrt{\frac{\varepsilon}{\lambda}}
            \sqrt{\frac{n}{2}}
            \sqrt{\log \frac{1}{\delta}}
        }
        \\
        &
        \leq
        \mathbb{P}\croch{
            A_{\varepsilon; D}
            \leq \frac{
                \rho_{p}\norm{\Gamma}_{\mathrm{op}}^{\frac{1}{2}} \kappa^{\Gamma} \sqrt{\varepsilon/\lambda}
            }{
                \sqrt{n}
            }
            \paren{
                2\sqrt{2} \paren{
                    \frac{
                        \mathrm{tr}\paren{\Gamma}
                    }{
                        \norm{\Gamma}_{\mathrm{op}}
                    }
                }^{\frac{1}{2}}
                + \sqrt{2 \log \frac{1}{\delta}}
            }
        }.
    \end{align*}

    Let us consider the case where $\hat{f}_{\lambda; D}\notin C_{\varepsilon}$.
    Since \eqref{asm.lip.loss} holds true, $\mathbf{R}_{\lambda}\paren{\bullet}$
    is continuous and since $C_{\varepsilon}$ is closed,
    one can choose $\eta \in \croch{f_{\lambda}, \hat{f}_{\lambda; D}}$
    on the boundary of $C_{\varepsilon} \subseteq B_{\varepsilon}$, that is,
    $\mathbf{R}_{\lambda}\paren{\eta} - \mathbf{R}_{\lambda}\paren{f_{\lambda}} = \varepsilon$.
    Moreover, since \eqref{asm.conv.loss} holds true,
    $\widehat{\mathbf{R}}_{\lambda; D}\paren{\bullet}$
    is convex, then,
    \begin{align*}
        \widehat{\mathbf{R}}_{\lambda; D}\paren{\eta}
        \leq \max\paren{
        \widehat{\mathbf{R}}_{\lambda; D}\paren{f_{\lambda}},
        \widehat{\mathbf{R}}_{\lambda; D}\paren{\hat{f}_{\lambda; D}}
        } \leq  \widehat{\mathbf{R}}_{\lambda; D}\paren{f_{\lambda}}.
    \end{align*}
    It follows that
    \begin{align*}
        \mathbf{R}_{\lambda}\paren{\eta}
        - \widehat{\mathbf{R}}_{\lambda; D}\paren{\eta}
        + \widehat{\mathbf{R}}_{\lambda; D}\paren{f_{\lambda}}
        - \mathbf{R}_{\lambda}\paren{f_{\lambda}}
        &
        \geq
        \paren{
            \mathbf{R}_{\lambda}\paren{\eta}
            - \mathbf{R}_{\lambda}\paren{f_{\lambda}}
        }
        - \paren{
            \widehat{\mathbf{R}}_{\lambda; D}\paren{\eta}
            - \widehat{\mathbf{R}}_{\lambda; D}\paren{f_{\lambda}}
        }
        \\
        &
        \geq \varepsilon.
    \end{align*}
    Thus, $\hat{f}_{\lambda; D} \notin C_{\varepsilon} \implies A_{\varepsilon, D} \geq \varepsilon$.
    As a consequence,
    for $\varepsilon_0>0$ such that,
    \begin{align*}
        \varepsilon_{0}
        = \frac{
            \rho_{p}\norm{\Gamma}_{\mathrm{op}}^{\frac{1}{2}} \kappa^{\Gamma} \sqrt{\varepsilon_{0}/\lambda}
        }{
            \sqrt{n}
        }
        \paren{
            2\sqrt{2} \paren{
                \frac{
                    \mathrm{tr}\paren{\Gamma}
                }{
                    \norm{\Gamma}_{\mathrm{op}}
                }
            }^{\frac{1}{2}}
            + \sqrt{2 \log \frac{1}{\delta}}
        },
    \end{align*}
    the next result holds true
    \begin{align*}
        \mathbb{P}\croch{
            \hat{f}_{\lambda; D} \notin C_{\varepsilon_{0}}
        }
        &
        =
        \mathbb{P}\croch{
            \hat{f}_{\lambda; D} \notin C_{\varepsilon_{0}},
            A_{\varepsilon_{0}, D} \geq \varepsilon_{0}
        }
        +
        \mathbb{P}\croch{
            \hat{f}_{\lambda; D} \notin C_{\varepsilon_{0}},
            A_{\varepsilon_{0}, D} < \varepsilon_{0}
        }
        \\
        &
        \leq 
        \mathbb{P}\croch{
            A_{\varepsilon_{0}, D} \geq \varepsilon_{0}
        }
        + 0
        \\
        &
        \leq \delta,
    \end{align*}
    where the last inequality follows from 
    the control on $A_{\varepsilon_{0}, D}$
    and on the definition of $\varepsilon_{0}$.
    One concludes by solving for $\varepsilon_{0}$,
    which results in the expression given by
    \begin{align*}
        \varepsilon_{0}
        =
        \croch{
            \frac{
                \rho_{p} \norm{\Gamma}_{\mathrm{op}}^{\frac{1}{2}} \kappa^{\Gamma} \sqrt{1/\lambda}
            }{
                \sqrt{n}
            }
            \paren{
                2\sqrt{2}\paren{
                    \frac{
                        \mathrm{tr}\paren{\Gamma}
                    }{
                        \norm{\Gamma}_{\mathrm{op}}
                    }
                }^{\frac{1}{2}}
                + \sqrt{2 \log \frac{1}{\delta}}
            }
        }^2.
    \end{align*}
\end{proof}
Again, this upper-bound is akin to one derived by \citet[see Proposition 4.6]{bachLearningTheoryFirst}.
The main difference is the effect of output dimension reflected
by the term $2\frac{\mathrm{tr}\paren{\Gamma}}{\norm{\Gamma}_{\mathrm{op}}}$
which is twice the effective rank \citep[see Definition 3]{bartlettBenignOverfittingLinear2020}
of the inter-task covariance-matrix $\Gamma$, and its operator norm $\norm{\Gamma}_{\mathrm{op}}$.
Yet again, this effect goes through the inter-task covariance-matrix.

\subsection{Upper-bounding the quantile value}

The next result provides a control on the deviation of $\hat{f}_{\lambda; D}$ around $f_{\lambda}$.
\begin{corollary}
    \label{cor.estimation.error.bound}
    Assume \eqref{asm.conv.loss},
    \eqref{asm.lip.loss}, \eqref{asm.lower.bounded.loss} and
    \eqref{asm.bounded.gamma.kernel} hold true.
    Then, for every $\delta \in \paren{0, 1}$, with probability greater than $1 - \delta$,
    \begin{align*}
        \normh{
            \hat{f}_{\lambda^{+}; D} - f_{\lambda}
        }
        \leq
        \frac{
            \rho_{p}\norm{\Gamma}_{\mathrm{op}}^{\frac{1}{2}} \kappa^{\Gamma}
        }{
            \lambda \sqrt{n}
        }
        \paren{
            \frac{\sqrt{n}}{2\paren{n+1}}
            +
            2\sqrt{2}\paren{
                \frac{
                    \mathrm{tr}\paren{\Gamma}
                }{
                    \norm{\Gamma}_{\mathrm{op}}
                }
            }^{\frac{1}{2}}
            + \sqrt{2 \log \frac{1}{\delta}}
        }.
    \end{align*}
\end{corollary}
\begin{proof}
Let $\delta \in \paren{0, 1}$ designate a risk level.
Since \eqref{asm.conv.loss} holds true, 
$\mathbf{R}_{\lambda}\paren{\bullet}$
is $2\lambda$-strongly convex, and
by Lemma~\ref{lm.minimum.gap},
for every $f \in \mathcal{H}$,
$\mathbf{R}_{\lambda}\paren{f} - \mathbf{R}_{\lambda}\paren{f_{\lambda}} \geq \lambda \normh{f - f_{\lambda}}^2$.
Thus,  $\normh{f - f_{\lambda}} \leq \frac{1}{\sqrt{\lambda}}\sqrt{\mathbf{R}_{\lambda}\paren{f} - \mathbf{R}_{\lambda}\paren{f_{\lambda}}}$.
For $f = \hat{f}_{\lambda; D}$ and using the inequality provided in Lemma~\ref{lm.estimation.error.bound},
with probability greater than $1 - \delta$,
\begin{align*}
    \normh{
        \hat{f}_{\lambda; D} - f_{\lambda}
    }
    \leq
    \frac{
        \rho_{p}\norm{\Gamma}_{\mathrm{op}}^{\frac{1}{2}} \kappa^{\Gamma}
    }{
        \lambda \sqrt{n}
    }
    \paren{
        2\sqrt{2}\paren{
            \frac{
                \mathrm{tr}\paren{\Gamma}
            }{
                \norm{\Gamma}_{\mathrm{op}}
            }
        }^{\frac{1}{2}}
        + \sqrt{2 \log \frac{1}{\delta}}
    }.
\end{align*}
Conjoining this with Lemma~\ref{lm.regularization.stability}
yields the desired result.
\end{proof}
%
%
%
%
%
%

%
%
{%
The upper bound improves at a rate of $O\paren{\frac{1}{\lambda\sqrt{n}}}$ for any fixed $p$.
For $p=1$, \citet[see Theorem 1]{smaleLearningTheoryEstimates2007},
provided an upper-bound on a similar deviation
when the loss-function is the quadratic loss-function.
Just like the one above, their upper bound improves at a rate of $O\paren{\frac{1}{\lambda\sqrt{n}}}$.
}

The effect of output-space dimension $p$ is reflected by the term $\frac{\mathrm{tr}\paren{\Gamma}}{\norm{\Gamma}_{\mathrm{op}}}$,
which is the effective rank \citep[see Definition 3]{bartlettBenignOverfittingLinear2020}
of the inter-task covariance-matrix $\Gamma$, and its operator norm $\norm{\Gamma}_{\mathrm{op}}$.
Although $p$ can be very large, as long as the effective rank of $\Gamma$ and its operator norm are bounded, then the upper-bound is informative.
This additionally emphasizes the importance of the inter-task relatedness,
since a strong relatedness effectively means a lower effective rank and, therefore, a tighter bound.

\medskip
Now that the randomness due to the training data set is controlled,
the next step is to deal with $\lambda$.
One must then consider a minimizer of the risk (which does not depend on $\lambda$).
Assuming such minimizer exists,
the next result ensures that the minimum-norm minimizer of the risk exists and is unique.
\begin{lemma}
    \label{lm.min.norm.risk.minimizer}
    Assume \eqref{asm.conv.loss}, \eqref{asm.lip.loss},
    \eqref{asm.lower.bounded.loss},
    \eqref{asm.bounded.expected.loss},
    \eqref{asm.bounded.gamma.kernel}, and
    \eqref{asm.risk.minimizer.attained} hold true.
    Then, the following minimization problem admits a unique
    solution $f_{\mathcal{H}} \in \mathcal{H}$,
    \begin{align*}
        \min_{
            \substack{
                f \in \mathcal{H}\\
                \mathbf{R}_{0}\paren{f} = \inf_{g\in \mathcal{H}} \mathbf{R}_{0}\paren{g}
            }
        } \norm{f}_{\mathcal{H}}.
    \end{align*}
\end{lemma}
\begin{proof}
Under \eqref{asm.conv.loss}, \eqref{asm.lip.loss},
\eqref{asm.bounded.expected.loss} and \eqref{asm.bounded.gamma.kernel},
Lemma~\ref{lm.risk.function.proper} states
that the risk function $\mathbf{R}_{0}\paren{\bullet}$
is a proper convex continuous function.
Under \eqref{asm.lower.bounded.loss},
\begin{align*}
    \inf_{g \in \mathcal{H}} \mathbf{R}_{0}\paren{g}
    \geq c_{\ell} > -\infty.
\end{align*}
It follows that the set of solution
$\brac{
    f \in \mathcal{H} :
    \mathbf{R}_{0}\paren{f} = \inf_{g\in \mathcal{H}} \mathbf{R}_{0}\paren{g}
}$ is convex, and since $\mathbf{R}_{0}\paren{\bullet}$
is continuous, this set is closed.
Moreover, under \eqref{asm.risk.minimizer.attained},
this set is non-empty.
Since the objective function
$f \in \mathcal{H} \mapsto \norm{f}_{\mathcal{H}}$
is convex, continuous and coercive,
then it admits a unique minimizer
over the non-empty
closed, convex set $\brac{
    f \in \mathcal{H},
    \mathbf{R}_{0}\paren{f} = \inf_{g\in \mathcal{H}} \mathbf{R}_{0}\paren{g}
}$.
\end{proof}

%
{Conjoining all the above results,
the next proposition provides a uniform upper-bound on the non-conformity scores given by Eq.~\eqref{eq.score},
thus controling the randomness of the quantile value.
}
\begin{lemma}
    \label{lm.bounded.score}
    %
    %
    %
    Assume \eqref{asm.conv.loss}, \eqref{asm.lip.loss},
    \eqref{asm.lower.bounded.loss}, \eqref{asm.bounded.expected.loss},
    \eqref{asm.bounded.gamma.kernel}, \eqref{asm.risk.minimizer.attained} and
    \eqref{asm.bounded.gamma.y} hold true.
    For every risk level $\delta \in \paren{0, 1}$, with probability greater than
    $1 - \delta$,
    \begin{align*}
        &
        \sup_{\paren{x, u} \in \mathcal{X} \times \mathcal{Y}}
        \norm{\Gamma^{-\frac{1}{2}}\paren{u - \hat{f}_{\lambda^{+}; D}(x)}}
        \\
        &
        \leq C_{\mathcal{Y}}\paren{p}
        + \frac{
            \rho_{p} \norm{\Gamma}_{\mathrm{op}}^{\frac{1}{2}} \kappa^{\Gamma}
        }{
            \lambda \sqrt{n}
        }
        \paren{
            \frac{\sqrt{n}}{2\paren{n+1}}
            + 2\sqrt{2}\paren{
                \frac{
                    \mathrm{tr}\paren{\Gamma}
                }{
                    \norm{\Gamma}_{\mathrm{op}}
                }
            }^{\frac{1}{2}}
            + \sqrt{2 \log \frac{1}{\delta}}
        }
        + \kappa^{\Gamma} 
        \normh{f_{\mathcal{H}}}.
    \end{align*}
\end{lemma}
\begin{proof}
Let $\delta \in \paren{0, 1}$ designate a risk level.
Since \eqref{asm.lip.loss} implies \eqref{asm.lsc.loss},
under \eqref{asm.conv.loss}, \eqref{asm.lip.loss},
and \eqref{asm.lower.bounded.loss},
Lemma~\ref{lm.predictor.well.defined}
states that $\widehat{\mathbf{R}}_{\lambda; D}\paren{\bullet}$
admits a unique minimizer $\hat{f}_{\lambda^{+}; D} \in \mathcal{H}$, and
Lemma~\ref{lm.min.norm.risk.minimizer}
states that $\mathbf{R}_{\lambda}\paren{\bullet}$
admits a unique minimizer $f_{\lambda} \in \mathcal{H}$.
Moreover, since \eqref{asm.conv.loss}, \eqref{asm.lip.loss},
\eqref{asm.lower.bounded.loss}, \eqref{asm.bounded.expected.loss},
\eqref{asm.bounded.gamma.kernel} and \eqref{asm.risk.minimizer.attained} hold true,
Lemma~\ref{lm.min.norm.risk.minimizer} states that $f_{\mathcal{H}}$
is well-defined.
\begin{align*}
    &
    \sup_{\paren{x, u} \in \mathcal{X} \times \mathcal{Y}}
    \norm{\Gamma^{-\frac{1}{2}}\paren{u - \hat{f}_{\lambda^{+}; D}(x)}}
    \\
    &
    \leq \sup_{u \in \mathcal{Y}}\norm{\Gamma^{-\frac{1}{2}}u}
    + \sup_{x \in \mathcal{X}}\norm{\Gamma^{-\frac{1}{2}}\hat{f}_{\lambda^{+}; D}(x)}
    \\
    &
    \leq C_{\mathcal{Y}}\paren{p}
    + \sup_{x \in \mathcal{X}}\norm{\Gamma^{-\frac{1}{2}} K\paren{x, x}\Gamma^{-\frac{1}{2}}}_{\mathrm{op}}^{\frac{1}{2}}
    \normh{\hat{f}_{\lambda^{+}; D}}
    \\
    &
    \leq C_{\mathcal{Y}}\paren{p}
    + \kappa^{\Gamma}
    \paren{
        \normh{\hat{f}_{\lambda^{+}; D} - f_{\lambda}}
        + \normh{f_{\lambda}}
    }
    \\
    &
    \leq C_{\mathcal{Y}}\paren{p}
    + \frac{
        \rho_{p} \norm{\Gamma}_{\mathrm{op}}^{\frac{1}{2}} \paren{\kappa^{\Gamma}}^{2}
    }{
        \lambda \sqrt{n}
    }
    \paren{
        \frac{\sqrt{n}}{2\paren{n+1}}
        + 2\sqrt{2}\paren{
            \frac{
                \mathrm{tr}\paren{\Gamma}
            }{
                \norm{\Gamma}_{\mathrm{op}}
            }
        }^{\frac{1}{2}}
        + \sqrt{2 \log \frac{1}{\delta}}
    }
    + \kappa^{\Gamma} 
    \normh{f_{\mathcal{H}}},
\end{align*}
where the first inequality follows from
the triangle inequality, and
the second inequality follows from
\eqref{asm.bounded.gamma.y} and Lemma~\ref{lm.gamma.prediction.bound}, and
the third inequality follows from \eqref{asm.bounded.gamma.kernel},
and the last inequality follows from Lemma~\ref{lm.reg.predictor.norm.bound}, and
Corollary~\ref{cor.estimation.error.bound}
which holds with probability greater than $1 - \delta$.
\end{proof}

\subsection{Proof of Theorem~\ref{thm.thickness.upper.bound}}
\label{proof.thickness.upper.bound}
\begin{proof}
Since \eqref{asm.eigenvalue.decay} holds true,
then $\norm{\Gamma}_{\mathrm{op}}^{\frac{1}{2}} \leq C_{\Gamma}^{\frac{1}{2}}$.
Moreover,
the determinant of $\Gamma^{\frac{1}{2}}$
is bounded from above, that is,
\begin{align*}
    \mathrm{det}\paren{\Gamma^{\frac{1}{2}}}
    = \mathrm{det}\paren{\Gamma}^{\frac{1}{2}}
    = \paren{
        \prod_{l=1}^{p} \mu_l\paren{\Gamma}
    }^{\frac{1}{2}}
    \leq
    C_{\Gamma}^{\frac{p}{2}}
    \paren{
        \prod_{l=1}^{p} \frac{1}{l^{\gamma}}
    }^{\frac{1}{2}}
    &
    \leq
    C_{\Gamma}^{\frac{p}{2}}
    \frac{1}{
        \paren{p!}^{\gamma/2}
    }
    \\
    &
    \leq
    \paren{2\pi}^{-\frac{\gamma}{4}}
    e^{-\frac{\gamma}{24p + 2}}
    p^{-\frac{\gamma}{4}}
    C_{\Gamma}^{\frac{p}{2}}
    e^{\frac{\gamma}{2}p}
    p^{-\frac{\gamma}{2}p},
\end{align*}
where the third line follows from Stirling's approximation of the factorial.
Using said approximation again,
\begin{align*}
    \frac{\pi^{\frac{p}{2}}}{\paren{\frac{p}{2}}!}
    \leq \pi^{-\frac{1}{2}}
    e^{-\frac{1}{6p + 1}}
    p^{-\frac{1}{2}}
    \paren{2 \pi}^{\frac{p}{2}}
    e^{\frac{p}{2}}
    p^{- \frac{p}{2}}.
\end{align*}
Since \eqref{asm.conv.loss}, \eqref{asm.lip.loss} and \eqref{asm.lower.bounded.loss}
hold true, Lemma~\ref{lm.thickness.empirical.bound.maha}
combined with the upper-bounds imply that
the \emph{thickness} is bounded from above, that is
\begin{align*}
    \mathrm{THK}_{\lambda; \alpha}^{\Gamma}\paren{X_{n+1}}
    &
    \leq
    \frac{
        \rho_{p} \paren{\kappa^{\Gamma}}^{2}
    }{\lambda \paren{n+1}}
    \paren{
        2C_{\Gamma}^{\frac{1}{2}}
        \paren{2\pi}^{-\frac{\gamma}{4}}
        e^{-\frac{\gamma}{24p + 2}}
        \pi^{-\frac{1}{2}}
        e^{-\frac{1}{6p + 1}}
    }\paren{
        p
        p^{-\frac{\gamma}{4}}
        p^{-\frac{1}{2}}
    }
    \\
    &
    \quad
    \times
    \paren{
        C_{\Gamma}^{\frac{p}{2}}
        e^{\frac{\gamma}{2}p}
        \paren{2 \pi}^{\frac{p}{2}}
        e^{\frac{p}{2}}
    }
    \paren{
        p^{-\frac{\gamma}{2}p}
        p^{- \frac{p}{2}}
    }
    \\
    &
    \quad
    \times
    \paren{
        \widehat{Q}_{\lambda; D}^{\Gamma}(\alpha)
        + \frac{
            \rho_{p} \paren{\kappa^{\Gamma}}^{2}
        }{\lambda \paren{n+1}}
        C_{\Gamma}^{\frac{1}{2}}
    }^{p-1}
    \\
    &
    \leq
    \frac{
        \rho_{p} \paren{\kappa^{\Gamma}}^{2}
    }{\lambda \paren{n+1}}
    \paren{
        2C_{\Gamma}^{\frac{1}{2}}
        2^{-\frac{\gamma}{4}}
        \pi^{-\frac{\gamma}{4} - \frac{1}{2}}
        e^{-\frac{\gamma}{24p + 2}
        - \frac{1}{6p + 1}}
    }p^{\frac{1}{2} - \frac{\gamma}{4}}
    \paren{
        \paren{2 \pi C_{\Gamma}}^{\frac{p}{2}}
        e^{\frac{\gamma+1}{2}p}
    }
    p^{-\frac{\gamma+1}{2}p}
    \\
    &
    \quad
    \times
    \paren{
        \widehat{Q}_{\lambda; D}^{\Gamma}(\alpha)
        + \frac{
            \rho_{p} \paren{\kappa^{\Gamma}}^{2}
        }{\lambda \paren{n+1}}
        C_{\Gamma}^{\frac{1}{2}}
    }^{p-1}.
\end{align*}

Let $\delta \in \paren{0, 1}$ designate a risk level.
Since \eqref{asm.eigenvalue.decay} holds true,
the trace of $\Gamma$
is bounded above, that is,
\begin{align*}
    \mathrm{tr}\paren{\Gamma}
    = \sum_{l=1}^{p}\mu_{l}\paren{\Gamma}
    &
    \leq \sum_{l=1}^{p} C_{\Gamma} \frac{1}{l^{\gamma}}
    \leq C_{\Gamma} \paren{1 + \sum_{l=2}^p \int_{l-1}^{l} \frac{1}{x^{\gamma}}dx}
    \leq C_{\Gamma} \paren{1 + \int_{1}^{l} \frac{1}{x^{\gamma}}dx}
    \\
    &
    \leq C_{\Gamma} \paren{1 + \croch{
        \frac{x^{-\gamma + 1}}{- \gamma + 1}
    }_{1}^{p}}
    \leq
    C_{\Gamma} \paren{1 + \croch{
        \frac{l^{-\gamma + 1}}{- \gamma + 1}
        - \frac{1}{- \gamma + 1}
    }
    }
    \\
    &
    \leq
    C_{\Gamma}
    \paren{\frac{\gamma}{\gamma - 1}},
\end{align*}
where the last inequality follows from the fact that $\gamma > 1$.    

Conjoined with \eqref{asm.conv.loss}, \eqref{asm.lip.loss},
\eqref{asm.lower.bounded.loss}, \eqref{asm.bounded.expected.loss},
\eqref{asm.risk.minimizer.attained}, \eqref{asm.bounded.gamma.y}
and \eqref{asm.not.source.condition},
Lemma~\ref{lm.bounded.score} implies that with probability greater than
$1 - \delta$,
the quantile $\widehat{Q}_{\lambda; D}^{\Gamma}(\alpha)$
(see Eq.~\ref{eq.score.quantile}) is bounded from above, that is,
\begin{align*}
    &
    \widehat{Q}_{\lambda; D}^{\Gamma}(\alpha)
    \\
    &
    \leq
    C_{\mathcal{Y}}\paren{p}
    + \frac{
        \rho_{p} \norm{\Gamma}_{\mathrm{op}}^{\frac{1}{2}} \paren{\kappa^{\Gamma}}^{2}
    }{
        \lambda \sqrt{n}
    }
    \paren{
        \frac{\sqrt{n}}{2\paren{n+1}}
        + 2\sqrt{2}\paren{
            \frac{
                \mathrm{tr}\paren{\Gamma}
            }{
                \norm{\Gamma}_{\mathrm{op}}
            }
        }^{\frac{1}{2}}
        + \sqrt{2 \log \frac{1}{\delta}}
    } + \kappa^{\Gamma} \normh{f_{\mathcal{H}}}
    \\
    &
    \leq
    C_{\mathcal{Y}}\paren{p}
    + \kappa^{\Gamma} C_{\mathcal{H}}
    + \frac{
        \rho_{p} \paren{\kappa^{\Gamma}}^{2}
    }{
        \lambda \sqrt{n}
    }
    \paren{
        C_{\Gamma}^{\frac{1}{2}}
        \frac{\sqrt{n}}{2\paren{n+1}}
        + 2^{\frac{3}{2}}
        C_{\Gamma}^{\frac{1}{2}}\paren{
            \frac{\gamma}{\gamma - 1}
        }^{\frac{1}{2}}
        + C_{\Gamma}^{\frac{1}{2}}\sqrt{2 \log \frac{1}{\delta}}
    }
    \\
    &
    \leq
    C_{\mathcal{Y}}\paren{p}
    + \kappa^{\Gamma} C_{\mathcal{H}}
    + \frac{
        \rho_{p} \paren{\kappa^{\Gamma}}^{2}
    }{
        \lambda \sqrt{n}
    }C_{\Gamma}^{\frac{1}{2}}
    \paren{
        \frac{\sqrt{n}}{2\paren{n+1}}
        + 2^{\frac{3}{2}}
        \paren{
            \frac{\gamma}{\gamma - 1}
        }^{\frac{1}{2}}
        + \sqrt{2 \log \frac{1}{\delta}}
    }.
\end{align*}
Combining all the above,
the \emph{thickness} is bounded from above
with probability greater than $1 - \delta$, that is
\begin{align*}
    &
    \mathrm{THK}_{\lambda; \alpha}^{\Gamma}\paren{X_{n+1}}
    \\
    &
    \leq
    \frac{
        \rho_{p} \paren{\kappa^{\Gamma}}^{2}
    }{\lambda \paren{n+1}}
    \paren{
        2C_{\Gamma}^{\frac{1}{2}}
        2^{-\frac{\gamma}{4}}
        \pi^{-\frac{\gamma}{4} - \frac{1}{2}}
        e^{-\frac{\gamma}{24p + 2}
        - \frac{1}{6p + 1}}
    }p^{\frac{1}{2} - \frac{\gamma}{4}}
    \paren{
        \paren{2 \pi C_{\Gamma}}^{\frac{p}{2}}
        e^{\frac{\gamma+1}{2}p}
    }
    p^{-\frac{\gamma+1}{2}p}
    \\
    &
    \quad
    \times
    \paren{
        \widehat{Q}_{\lambda; D}^{\Gamma}(\alpha)
        + \frac{
            \rho_{p} \paren{\kappa^{\Gamma}}^{2}
        }{\lambda \paren{n+1}}
        C_{\Gamma}^{\frac{1}{2}}
    }^{p-1}
    \\
    &
    \leq
    \frac{
        \rho_{p} \paren{\kappa^{\Gamma}}^{2}
    }{\lambda \paren{n+1}}
    \paren{
        2C_{\Gamma}^{\frac{1}{2}}
        2^{-\frac{\gamma}{4}}
        \pi^{-\frac{\gamma}{4} - \frac{1}{2}}
        e^{-\frac{\gamma}{24p + 2}
        - \frac{1}{6p + 1}}
    }p^{\frac{1}{2} - \frac{\gamma}{4}}
    \paren{
        \paren{2 \pi C_{\Gamma}}^{\frac{p}{2}}
        e^{\frac{\gamma+1}{2}p}
    }
    p^{-\frac{\gamma+1}{2}p}
    \\
    &
    \quad
    \times
    \paren{
        C_{\mathcal{Y}}\paren{p}
        + \kappa^{\Gamma} C_{\mathcal{H}}
        + \frac{
            \rho_{p} \paren{\kappa^{\Gamma}}^{2}
        }{
            \lambda \sqrt{n}
        }C_{\Gamma}^{\frac{1}{2}}
        \paren{
            \frac{\sqrt{n}}{2\paren{n+1}}
            + 2^{\frac{3}{2}}
            \paren{
                \frac{\gamma}{\gamma - 1}
            }^{\frac{1}{2}}
            + \sqrt{2 \log \frac{1}{\delta}}
        }
        + \frac{
            \rho_{p} \paren{\kappa^{\Gamma}}^{2}
        }{\lambda \paren{n+1}}
        C_{\Gamma}^{\frac{1}{2}}
    }^{p-1}
    \\
    &
    \leq
    \frac{
        \rho_{p} \paren{\kappa^{\Gamma}}^{2}
    }{\lambda \paren{n+1}}
    \paren{
        2C_{\Gamma}^{\frac{1}{2}}
        2^{-\frac{\gamma}{4}}
        \pi^{-\frac{\gamma}{4} - \frac{1}{2}}
        e^{-\frac{\gamma}{24p + 2}
        - \frac{1}{6p + 1}}
    }p^{\frac{1}{2} - \frac{\gamma}{4}}
    \paren{
        \paren{2 \pi C_{\Gamma}}^{\frac{p}{2}}
        e^{\frac{\gamma+1}{2}p}
    }
    p^{-\frac{\gamma+1}{2}p}
    \\
    &
    \quad
    \times
    \paren{
        C_{\mathcal{Y}}\paren{p}
        + \kappa^{\Gamma} C_{\mathcal{H}}
        + \frac{
            \rho_{p} \paren{\kappa^{\Gamma}}^{2}
        }{
            \lambda \sqrt{n}
        }C_{\Gamma}^{\frac{1}{2}}
        \paren{
            \frac{3}{2}
            \frac{\sqrt{n}}{n+1}
            +
            2^{\frac{3}{2}}
            \paren{
                \frac{\gamma}{\gamma - 1}
            }^{\frac{1}{2}}
            + \sqrt{2 \log \frac{1}{\delta}}
        }
    }^{p-1}
    \\
    &
    \leq
    \frac{
        \rho_{p} \paren{\kappa^{\Gamma}}^{2}
    }{\lambda \paren{n+1}}
    a^{\Gamma}\paren{p}
    C_{\lambda; \theta}^{\Gamma, \gamma} \paren{p}
    p^{-\frac{\gamma+1}{2}p},
\end{align*}
where $a^{\Gamma}\paren{p}$ and $C_{\lambda; \theta}^{\Gamma, \gamma} \paren{p}$ are given by Eq.~\eqref{eq.constant.term}.
\end{proof}

\subsection{Proof of Proposition~\ref{prop.generalization.upper.bound}}
\label{proof.generalization.upper.bound}

\begin{proof}
    Since \eqref{asm.lip.loss} implies \eqref{asm.lsc.loss},
    then, under \eqref{asm.conv.loss}, \eqref{asm.lip.loss}
    and \eqref{asm.lower.bounded.loss}, Lemma~\ref{lm.predictor.well.defined}
    states that $\hat{f}_{\lambda; D}$ exists and is unique,
    so does $\hat{f}_{\lambda; D^{i}}$,
    for every $i \in \brac{1, \ldots, n}$ (see Eq.~\ref{eq.modified.data.set} for $D^{i}$).
    Moreover, $\widehat{\mathbf{R}}_{\lambda; D}\paren{\cdot}$,
    and $\widehat{\mathbf{R}}_{\lambda; D^{i}}\paren{\cdot}$
    are $2\lambda$-strongly convex.
    Then, from Lemma~\eqref{lm.minimum.gap},
    \begin{align*}
        \lambda \normh{
            \hat{f}_{\lambda; D}
            - \hat{f}_{\lambda; D^{i}}
        }^2
        &
        \leq \widehat{\mathbf{R}}_{\lambda; D}\paren{\hat{f}_{\lambda; D^{i}}}
        - \widehat{\mathbf{R}}_{\lambda; D}\paren{\hat{f}_{\lambda; D}}
        \\
        &
        \leq \widehat{\mathbf{R}}_{\lambda; D^{i}}\paren{\hat{f}_{\lambda; D^{i}}}
        - \widehat{\mathbf{R}}_{\lambda; D^{i}}\paren{\hat{f}_{\lambda; D}}
        \\
        &
        \quad
        + \frac{1}{n}
        \paren{
            \ell\paren{Y_{i}^{\prime}, \hat{f}_{\lambda; D}\paren{X_i^{\prime}}}
            - \ell\paren{Y_{i}^{\prime}, \hat{f}_{\lambda; D^{i}}\paren{X_i^{\prime}}}
        }
        \\
        &
        \quad
        + \frac{1}{n}
        \paren{
            \ell\paren{Y_{i}, \hat{f}_{\lambda; D^{i}}\paren{X_i}}
            - \ell\paren{Y_{i}, \hat{f}_{\lambda; D}\paren{X_i}}
        }.
    \end{align*}
    On the other hand,
    \begin{align*}
        \lambda \normh{
            \hat{f}_{\lambda; D^{i}}
            - \hat{f}_{\lambda; D}
        }^2
        \leq \widehat{\mathbf{R}}_{\lambda; D^{i}}\paren{\hat{f}_{\lambda; D}}
        - \widehat{\mathbf{R}}_{\lambda; D^{i}}\paren{\hat{f}_{\lambda; D^{i}}},
    \end{align*}
    which entails that
    \begin{align*}
        \widehat{\mathbf{R}}_{\lambda; D^{i}}\paren{\hat{f}_{\lambda; D^{i}}}
        - \widehat{\mathbf{R}}_{\lambda; D^{i}}\paren{\hat{f}_{\lambda; D}} 
        \leq - \lambda \normh{
            \hat{f}_{\lambda; D^{i}}
            - \hat{f}_{\lambda; D}
        }^2.
    \end{align*}
    Summing this with the first sets of inequalities yields
    \begin{align*}
        &
        2
        \lambda \normh{
            \hat{f}_{\lambda; D}
            - \hat{f}_{\lambda; D^{i}}
        }^2
        \\ 
        &
        \leq \frac{1}{n}
        \paren{
            \ell\paren{Y_{i}^{\prime}, \hat{f}_{\lambda; D}\paren{X_i^{\prime}}}
            - \ell\paren{Y_{i}^{\prime}, \hat{f}_{\lambda; D^{i}}\paren{X_i^{\prime}}}
        }
        + \frac{1}{n}
        \paren{
            \ell\paren{Y_{i}, \hat{f}_{\lambda; D^{i}}\paren{X_i}}
            - \ell\paren{Y_{i}, \hat{f}_{\lambda; D}\paren{X_i}}
        }
        \\
        &
        \leq
        \frac{\rho_{p}}{n}
        \norm{K\paren{X_{i}^{\prime}, X_{i}^{\prime}}}_{\mathrm{op}}^{\frac{1}{2}}
        \normh{\hat{f}_{\lambda; D} - \hat{f}_{\lambda; D^{i}}}
        + \frac{\rho_{p}}{n}
        \norm{K\paren{X_{i}, X_{i}}}_{\mathrm{op}}^{\frac{1}{2}}
        \normh{\hat{f}_{\lambda; D} - \hat{f}_{\lambda; D^{i}}}
        \\
        &
        \leq
        \frac{\rho_{p}}{n}
        \croch{
            \norm{K\paren{X_{i}^{\prime}, X_{i}^{\prime}}}_{\mathrm{op}}^{\frac{1}{2}}
            + \norm{K\paren{X_{i}, X_{i}}}_{\mathrm{op}}^{\frac{1}{2}}
        }\normh{
            \hat{f}_{\lambda; D}
            - \hat{f}_{\lambda; D^{i}}
        },
    \end{align*}
    where the second inequality follows from \citet[see Lemma 3]{audiffrenStabilityMultitaskKernel2013}
    which holds true under \eqref{asm.lip.loss}.
    Dividing both sides by $2
    \lambda \normh{
        \hat{f}_{\lambda; D}
        - \hat{f}_{\lambda; D^{i}}
    }^2$ yields
    \begin{align*}
        \normh{
            \hat{f}_{\lambda; D}
            - \hat{f}_{\lambda; D^{i}}
        } \leq\frac{\rho_{p}}{2\lambda n}
        \croch{
            \norm{K\paren{X_{i}^{\prime}, X_{i}^{\prime}}}_{\mathrm{op}}^{\frac{1}{2}}
            + \norm{K\paren{X_{i}, X_{i}}}_{\mathrm{op}}^{\frac{1}{2}}
        } 
    \end{align*}
    Then, by \citet[see Lemma 7]{bousquetStabilityGeneralization2002},
    for every $i \in \brac{1, \ldots, n}$,
    \begin{align*}
        &
        \mathbb{E}_{D}\croch{
            \mathbf{R}_{0}\paren{\hat{f}_{\lambda; D}}
            - \widehat{\mathbf{R}}_{0; D}\paren{\hat{f}_{\lambda; D}}
        }
        \\
        &
        = \mathbb{E}_{D, \paren{X_i^{\prime}, Y_i^{\prime}}}\croch{
            \ell\paren{Y_{i}^{\prime}, \hat{f}_{\lambda; D}\paren{X_{i}^{\prime}}}
            - \ell\paren{Y_{i}^{\prime}, \hat{f}_{\lambda; D^{i}}\paren{X_{i}^{\prime}}}
        }
        \\
        &
        \leq
        \rho_{p}
        \mathbb{E}_{D, \paren{X_i^{\prime}, Y_i^{\prime}}}\croch{
            \normh{\hat{f}_{\lambda; D} - \hat{f}_{\lambda; D^{i}}}
        }
        \\
        &
        \leq
        \frac{\rho_{p}}{2\lambda n}
            \mathbb{E}_{D, \paren{X_i^{\prime}, Y_i^{\prime}}}\croch{
            \norm{K\paren{X_{i}^{\prime}, X_{i}^{\prime}}}_{\mathrm{op}}^{\frac{1}{2}}
            + \norm{K\paren{X_{i}, X_{i}}}_{\mathrm{op}}^{\frac{1}{2}}
        }
        \\
        &
        \leq
        \frac{\rho_{p}}{\lambda n}
            \mathbb{E}\croch{
            \norm{K\paren{X, X}}_{\mathrm{op}}^{\frac{1}{2}}
        }.
    \end{align*}
\end{proof}
\subsection{Additional numerical experiment on the evolution of the thickness}
\label{sec.evolution.thickness.fixed}
\begin{figure}[H]
    \centering
    \includegraphics[width=0.65\textwidth]{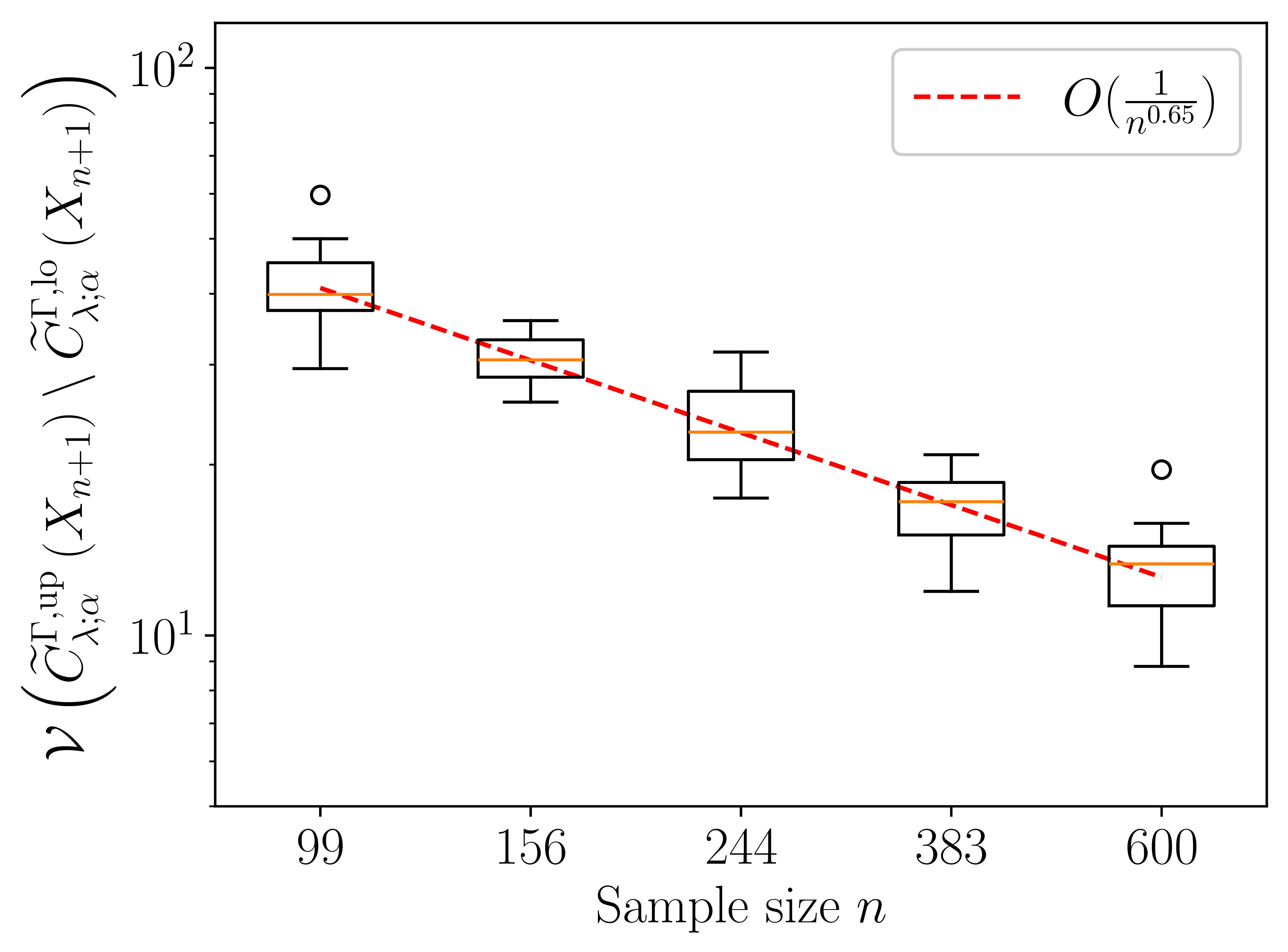}
    \caption{
        Evolution of the computable empirical upper-bound $\leb{\widetilde{C}_{\lambda; \alpha}^{\Gamma, \mathrm{up}}\paren{X_{n+1}} \setminus \widetilde{C}_{\lambda; \alpha}^{\Gamma, \lo}\paren{X_{n+1}}}$
        for the \emph{thickness} $\thicc{\Gamma}$ across 20 repetitions for $\alpha = 0.1$,
        and for $\lambda \propto \frac{1}{\sqrt{n}}$.
    }
    \label{fig.thickness.stableCP.1}
\end{figure}

Since $\lambda \propto \frac{1}{\sqrt{n}}$ then, $\frac{1}{\lambda n} = O\paren{\frac{1}{\sqrt{n}}}$,
that is to say, by Theorem~\ref{thm.thickness.upper.bound}, the theoretical rate should be $O\paren{\frac{1}{\sqrt{n}}}$.
On average, $\leb{\widetilde{C}_{\lambda; \alpha}^{\Gamma, \mathrm{up}}\paren{X_{n+1}} \setminus \widetilde{C}_{\lambda; \alpha}^{\Gamma, \lo}\paren{X_{n+1}}}$ gets smaller as the training-sample size grows larger.
The estimated rate of improvement (represented by the dashed red-line) is a bit faster than the theoretical one, that is, $O\paren{\frac{1}{\sqrt{n}}}$.

\section{Concerning the estimated inter-task covariance-matrix case}
This section gives the proofs of the results in Section~\ref{sec.estimated.inter.task.covariance}.
\subsection{Preliminary properties}
\begin{lemma}
    \label{lm.weighted.covariance.matrix.spd}
    Let $Y_{1}$, \dots, $Y_{n}$ stand for output-vectors,
    and $\widetilde{w} \in \mathbb{R}^{n}_{+}$, a weight vector
    with only non-negative coordinates.
    Then,
    \begin{align*}
        &
        \sum_{i=1}^{n}
        \widehat{w}_{i}
        \paren{
            Y_i - 
            \paren{
                \sum_{j=1}^{n}\widehat{w}_{j} Y_j
            }
        }\paren{
            Y_i - \paren{
                \sum_{j=1}^{n}\widehat{w}_{j} Y_j
            }
        }^{T}
        + a I_{p}
        \\
        &
        =
        \sum_{i=1}^{n}
        \widehat{w}_{i} Y_i Y_i^{T}
        + \paren{
            \sum_{j=1}^{n}\widehat{w}_{j} Y_j
        }\paren{
            \sum_{j=1}^{n}\widehat{w}_{j} Y_j
        }^{T}
        + a I_{p},
    \end{align*}
    where for every $i \in \brac{1, \ldots, n}$,
    $\widehat{w}_i := \frac{
        \widetilde{w}_i
    }{
        \sum_{j = 1}^{n} \widetilde{w}_{j}
    }$.
    Therefore, the above matrix is symmetric positive define.
\end{lemma}
\begin{proof}
    \begin{align*}
        &
        \sum_{i=1}^{n}
        \widehat{w}_{i}
        \paren{
            Y_i - \paren{
                \sum_{j=1}^{n}\widehat{w}_{j} Y_j
            }
        }\paren{
            Y_i - \paren{
                \sum_{j=1}^{n}\widehat{w}_{j} Y_j
            }
        }^{T}
        + a I_{p}
        \\
        &
        = 
        \sum_{i=1}^{n}
        \widehat{w}_{i}
        Y_i Y_i^{T}
        + \sum_{i=1}^{n}
        \widehat{w}_{i}
        \paren{
            \sum_{j=1}^{n}\widehat{w}_{j} Y_j
        }
        \paren{
            \sum_{j=1}^{n}\widehat{w}_{j} Y_j
        }^{T}
        \\
        &
        \quad
        - \sum_{i=1}^{n}
        \widehat{w}_{i}
        \paren{
            \sum_{j=1}^{n}\widehat{w}_{j} Y_j
        } Y_i^{T}
        - \sum_{i=1}^{n}
        \widehat{w}_{i}
        Y_i \paren{
            \sum_{j=1}^{n}\widehat{w}_{j} Y_j
        }^{T}
        + a I_{p}
        \\
        &
        = 
        \sum_{i=1}^{n}
        \widehat{w}_{i}
        Y_i Y_i^{T}
        + \paren{
            \sum_{j=1}^{n}\widehat{w}_{j} Y_j
        }
        \paren{
            \sum_{j=1}^{n}\widehat{w}_{j} Y_j
        }^{T}
        \\
        &
        \quad
        - \paren{
            \sum_{j=1}^{n}\widehat{w}_{j} Y_j
        }
        \paren{
            \sum_{i=1}^{n}
            \widehat{w}_{i}
            Y_i
        }^{T}
        - \paren{
            \sum_{i=1}^{n}
            \widehat{w}_{i}
            Y_i
        } \paren{
            \sum_{j=1}^{n}\widehat{w}_{j} Y_j
        }^{T}
        + a I_{p}
        \\
        &
        = 
        \sum_{i=1}^{n}
        \widehat{w}_{i}
        Y_i Y_i^{T}
        - \paren{
            \sum_{j=1}^{n}\widehat{w}_{j} Y_j
        }
        \paren{
            \sum_{j=1}^{n}\widehat{w}_{j} Y_j
        }^{T}
        + a I_{p},
    \end{align*}
    where the second equality follows from the fact that
    $\sum_{i=1}^{n} \widehat{w}_{j} = 1$.
\end{proof}
\begin{lemma}
    \label{lm.estimated.global.covariance}
    Let $a \in \paren{0, +\infty}$
    designate a regularization parameter,
    and $y \in \mathcal{Y}$, a test output-value.
    \begin{align*}
        \widehat{\Gamma}_{a; D^{y}}
        =
        \frac{n}{n+1}
        \paren{
            \widehat{\Gamma}_{a^{+}}
            + \paren{
                \frac{y - \widehat{\mu}_{D}}{\sqrt{n+1}}
            }\paren{
                \frac{y - \widehat{\mu}_{D}}{\sqrt{n+1}}
            }^{T}
        },
    \end{align*}
    where $\widehat{\Gamma}_{a^{+}}$ is given by Eq.~\eqref{eq.global.covariance}.
\end{lemma}
\begin{proof}
    Let us notice the following,
    \begin{align*}
        y y^{T}
        &
        = \paren{y - \widehat{\mu}_{D} + \widehat{\mu}_{D}}
        \paren{y - \widehat{\mu}_{D} + \widehat{\mu}_{D}}^{T}
        \\
        &
        = \paren{
            y - \widehat{\mu}_{D}
        }\paren{
            y - \widehat{\mu}_{D}
        }^{T}
        + \widehat{\mu}_{D}
        \widehat{\mu}_{D}^{T}
        + \paren{
            y - \widehat{\mu}_{D}
        }\widehat{\mu}_{D}^{T}
        + \widehat{\mu}_{D}^{T}
        \paren{
            y - \widehat{\mu}_{D}
        }.
    \end{align*}
    It follows that
    \begin{align*}
        \frac{1}{n+1}
        \paren{
            yy^{T}
            + \sum_{i = 1}^{n}
            Y_i Y_i^{T}    
        }
        &
        =
        \frac{1}{n+1}
        \paren{
            \widehat{\mu}_{D}
            \widehat{\mu}_{D}^{T}
            + \sum_{i = 1}^{n}
            Y_i Y_i^{T}
        }
        + \frac{1}{n+1}
        \paren{
            y - \widehat{\mu}_{D}
        }\paren{
            y - \widehat{\mu}_{D}
        }^{T}
        \\
        &
        \quad
        + \frac{1}{n+1}
        \paren{
            y - \widehat{\mu}_{D}
        }\widehat{\mu}_{D}^{T}
        + \frac{1}{n+1}
        \widehat{\mu}_{D}
        \paren{
            y - \widehat{\mu}_{D}
        }^{T}.
    \end{align*}
    Moreover,
    \begin{align*}
        \frac{1}{n+1}
        \paren{
            y 
            + \sum_{i=1}^{n}
            Y_i
        }
        &
        =
        \frac{1}{n+1}
        y
        + \frac{1}{n+1} \sum_{i=1}^{n} Y_i
        =
        \frac{1}{n+1}
        \paren{
            y - \widehat{\mu}_{D}
        }
        + \frac{\widehat{\mu}_{D}}{n+1}
        + \frac{n}{n+1} \widehat{\mu}_{D}
        \\
        &
        =
        \frac{1}{n+1}
        \paren{
            y - \widehat{\mu}_{D}
        }
        + \widehat{\mu}_{D}. 
    \end{align*}
    It follows that
    \begin{align*}
        &
        \paren{
            \frac{1}{n+1}
            \paren{
                y 
                + \sum_{i=1}^{n}
                Y_i
            }
        }\paren{
            \frac{1}{n+1}
            \paren{
                y 
                + \sum_{i=1}^{n}
                Y_i
            }
        }^{T}
        \\
        &
        =
        \paren{
            \frac{1}{n+1}
            \paren{
                y - \widehat{\mu}_{D}
            }
            + \widehat{\mu}_{D}
        }\paren{
            \frac{1}{n+1}
            \paren{
                y - \widehat{\mu}_{D}
            }
            + \widehat{\mu}_{D}
        }^{T}
        \\
        &
        =
        \frac{1}{\paren{n+1}^2}
        \paren{
            y - \widehat{\mu}_{D}
        }\paren{
            y - \widehat{\mu}_{D}
        }^{T}
        + \widehat{\mu}_{D} \widehat{\mu}_{D}^{T}
        + \frac{1}{n+1}
        \paren{
            y - \widehat{\mu}_{D}
        }
        \widehat{\mu}_{D}^{T}
        + \frac{1}{n+1}
        \widehat{\mu}_{D}
        \paren{
            y - \widehat{\mu}_{D}
        }^{T}.
    \end{align*}
    As a result,
    \begin{align*}
        \widehat{\Gamma}_{a; D^{y}}
        &
        =
        \frac{1}{n+1}
        \paren{
            \widehat{\mu}_{D}
            \widehat{\mu}_{D}^{T}
            + \sum_{i = 1}^{n}
            Y_i Y_i^{T}
        }
        + \frac{1}{n+1}
        \paren{
            y - \widehat{\mu}_{D}
        }\paren{
            y - \widehat{\mu}_{D}
        }^{T}
        + a I_{p}
        \\
        &
        \quad
        + \frac{1}{n+1}
        \paren{
            y - \widehat{\mu}_{D}
        }\widehat{\mu}_{D}^{T}
        + \frac{1}{n+1}
        \widehat{\mu}_{D}
        \paren{
            y - \widehat{\mu}_{D}
        }^{T}
        \\
        &
        \quad
        -
        \frac{1}{\paren{n+1}^2}
        \paren{
            y - \widehat{\mu}_{D}
        }\paren{
            y - \widehat{\mu}_{D}
        }^{T}
        - \widehat{\mu}_{D} \widehat{\mu}_{D}^{T}
        \\
        &
        \quad
        - \frac{1}{n+1}
        \paren{
            y - \widehat{\mu}_{D}
        }
        \widehat{\mu}_{D}^{T}
        - \frac{1}{n+1}
        \widehat{\mu}_{D}
        \paren{
            y - \widehat{\mu}_{D}
        }^{T}
        + a I_{p}
        \\
        &
        =
        \frac{1}{n+1}
        \sum_{i = 1}^{n}
        Y_i Y_i^{T}
        +
        \paren{
            \frac{1}{n+1} - 1
        } \widehat{\mu}_{D} \widehat{\mu}_{D}^{T}
        + \frac{n}{\paren{n+1}^2}\paren{
            y - \widehat{\mu}_{D}
        }\paren{
            y - \widehat{\mu}_{D}
        }^{T}
        + a I_{p}
        \\
        &
        =
        \frac{n}{n+1}
        \paren{
            \frac{1}{n}
            \sum_{i = 1}^{n}
            Y_i Y_i^{T}
            - \widehat{\mu}_{D} \widehat{\mu}_{D}^{T}
            + a \paren{\frac{n+1}{n}} I_{p}
        } + \frac{n}{\paren{n+1}^2}\paren{
            y - \widehat{\mu}_{D}
        }\paren{
            y - \widehat{\mu}_{D}
        }^{T}
        \\
        &
        =
        \frac{n}{n+1}
        \paren{
            \widehat{\Gamma}_{a^{+}}
            + \paren{
                \frac{y - \widehat{\mu}_{D}}{\sqrt{n+1}}
            }\paren{
                \frac{y - \widehat{\mu}_{D}}{\sqrt{n+1}}
            }^{T}
        }.
    \end{align*}
\end{proof}

\begin{lemma}
    \label{lm.global.norm.bound}
    Let $a \in \paren{0, +\infty}$
    designate a regularization parameter,
    and $y \in \mathcal{Y}$, a test output-value,
    and $v \in \mathbb{R}^{p}$ a vector of residual.
    Then, 
    \begin{align*}
        \frac{
            \paren{
                \frac{n+1}{n}
            }^{\frac{1}{2}}
            \norm{
                \widehat{\Gamma}_{a^{+}}
                v
            }
        }{
            \paren{
                1 + \frac{1}{n+1}
                \paren{
                    \norm{
                        \widehat{\Gamma}_{a^{+}}
                        \paren{y - \hat{f}_{\lambda^{+}; D}\paren{X_{n+1}}}
                    } + \widehat{t}_{\lambda; a}
                }^2
            }^{\frac{1}{2}}
        }
        \leq
        \norm{
            \widehat{\Gamma}_{a; D^{y}}^{-\frac{1}{2}} v
        }
        \leq
        \paren{
            \frac{n+1}{n}
        }^{\frac{1}{2}}
        \norm{
            \widehat{\Gamma}_{a^{+}}
            v
        },
    \end{align*}
    where $\widehat{t}_{\lambda; a}$ is given by Eq.~\eqref{eq.correction}.
\end{lemma}
\begin{proof}
    Since $\widehat{\Gamma}_{a; D^{y}}$ and 
    $\widehat{\Gamma}_{a^{+}}$ are non-singular by Lemma~\ref{lm.weighted.covariance.matrix.spd},
    then, by the Sherman-Morisson-Woodbury formula
    along with Lemma~\ref{lm.estimated.global.covariance},
    \begin{align*}
        \widehat{\Gamma}_{a; D^{y}}^{-1}
        &
        =
        \paren{
            \frac{n+1}{n}
        }
        \paren{
            \widehat{\Gamma}_{a^{+}}^{-1}
            - \frac{
                \widehat{\Gamma}_{a^{+}}^{-1}
                \paren{
                    \frac{y - \widehat{\mu}_{D}}{\sqrt{n+1}}
                }
                \paren{
                    \frac{y - \widehat{\mu}_{D}}{\sqrt{n+1}}
                }^{T}
                \widehat{\Gamma}_{a^{+}}^{-1}
            }{
                1 + \paren{
                    \frac{y - \widehat{\mu}_{D}}{\sqrt{n+1}}
                }^{T} \widehat{\Gamma}_{a^{+}}^{-1} \paren{
                    \frac{y - \widehat{\mu}_{D}}{\sqrt{n+1}}
                }
            }
        }
        \\
        &
        =
        \paren{
            \frac{n+1}{n}
        }
        \paren{
            \widehat{\Gamma}_{a^{+}}^{-1}
            - \frac{
                \frac{1}{n+1}
                \widehat{\Gamma}_{a^{+}}^{-1}
                \paren{y - \widehat{\mu}_{D}}
                \paren{y - \widehat{\mu}_{D}}^{T}
                \widehat{\Gamma}_{a^{+}}^{-1}
            }{
                1 + \frac{1}{n+1}
                \norm{
                    \widehat{\Gamma}_{a^{+}}^{-\frac{1}{2}}
                    \paren{y - \widehat{\mu}_{D}}
                }^{2}
            }
        }.
    \end{align*}
    Let $v \in \mathbb{R}^{p \times 1}$
    designate a vector.
    \begin{align*}
        \norm{
            \widehat{\Gamma}_{a; D^{y}}^{-\frac{1}{2}} v
        }^2 
        &
        =
        \paren{
            \frac{n+1}{n}
        }
        v^{T} \widehat{\Gamma}_{a^{+}}^{-1} v
        - \paren{
            \frac{n+1}{n}
        }
        \frac{
            \frac{1}{n+1}
            \paren{
                v^T\widehat{\Gamma}_{a^{+}}^{-1}
                \paren{y - \widehat{\mu}_{D}}
            }^2
        }{
            1 + \frac{1}{n+1}
            \norm{
                \widehat{\Gamma}_{a^{+}}^{-\frac{1}{2}}
                \paren{y - \widehat{\mu}_{D}}
            }^{2}
        }.
    \end{align*}
    On the one hand,
    \begin{align*}
        \norm{
            \widehat{\Gamma}_{a; D^{y}}^{-\frac{1}{2}} v
        }^2 \leq \paren{
            \frac{n+1}{n}
        }\norm{
            \widehat{\Gamma}_{a^{+}}^{-\frac{1}{2}} v
        }^2,
    \end{align*}
    and by taking the square root on both sides,
    \begin{align*}
        \norm{
            \widehat{\Gamma}_{a; D^{y}}^{-\frac{1}{2}} v
        } \leq 
        \paren{
            \frac{n+1}{n}
        }^{\frac{1}{2}}
        \norm{
            \widehat{\Gamma}_{a^{+}}^{-\frac{1}{2}} v
        }.
    \end{align*}
    On the other hand,
    by the Cauchy-Schwarz inequality,
    and the triangle inequality,
    \begin{align*}
        \norm{
            \widehat{\Gamma}_{a; D^{y}}^{-\frac{1}{2}} v
        }^2
        &
        \geq
        \paren{
            \frac{n+1}{n}
        }
        v^{T} \widehat{\Gamma}_{a^{+}}^{-1} v
        - \paren{
            \frac{n+1}{n}
        }
        \frac{
            \frac{1}{n+1}
            \norm{
                \widehat{\Gamma}_{a^{+}}^{-\frac{1}{2}}
                \paren{y - \widehat{\mu}_{D}}
            }^{2}
            \norm{
                \widehat{\Gamma}_{a^{+}}^{-\frac{1}{2}}
                v
            }^{2}
        }{
            1 + \frac{1}{n+1}
            \norm{
                \widehat{\Gamma}_{a^{+}}^{-\frac{1}{2}}
                \paren{y - \widehat{\mu}_{D}}
            }^{2}
        }
        \\
        &
        = \frac{
            \paren{
                \frac{n+1}{n}
            }
            \norm{
                \widehat{\Gamma}_{a^{+}}^{-\frac{1}{2}}
                v
            }^{2}
        }{
            1 + \frac{1}{n+1}
            \norm{
                \widehat{\Gamma}_{a^{+}}^{-\frac{1}{2}}
                \paren{y - \widehat{\mu}_{D}}
            }^{2}
        }
        \\
        &
        = \frac{
            \paren{
                \frac{n+1}{n}
            }
            \norm{
                \widehat{\Gamma}_{a^{+}}^{-\frac{1}{2}}
                v
            }^{2}
        }{
            1 + \frac{1}{n+1}
            \paren{
                \norm{
                    \widehat{\Gamma}_{a^{+}}^{-\frac{1}{2}}
                    \paren{y - \hat{f}_{\lambda^{+}; D}\paren{X_{n+1}}}
                } + \widehat{t}_{\lambda; a}
            }^{2}
        },
    \end{align*}
    where $\widehat{t}_{\lambda; a}$ given by Eq.~\eqref{eq.correction},
    and by taking the square root on both sides
    \begin{align*}
        \norm{
            \widehat{\Gamma}_{a; D^{y}}^{-\frac{1}{2}} v
        }
        \geq \frac{
            \paren{
                \frac{n+1}{n}
            }^{\frac{1}{2}}
            \norm{
                \widehat{\Gamma}_{a^{+}}^{-\frac{1}{2}}
                v
            }
        }{
            \paren{
                1 + \frac{1}{n+1}
                \paren{
                    \norm{
                        \widehat{\Gamma}_{a^{+}}^{-\frac{1}{2}}
                        \paren{y - \hat{f}_{\lambda^{+}; D}\paren{X_{n+1}}}
                    } + \widehat{t}_{\lambda; a}
                }^{2}
            }^{\frac{1}{2}}
        }.
    \end{align*}
\end{proof}
\begin{lemma}
    \label{lm.variation.lower.transformation.function}
    Let $t, \tau \in \mathbb{R}_{+}$ and $w \in \paren{0, +\infty}$ name parameters,
    and $L_{t}^{\tau} : \mathbb{R}_{+} \to \mathbb{R}$,
    the function given by, for every $x \in \mathbb{R}_{+}$,
    \begin{align}
        \label{eq.lower.transformation.function}
        L_{w}^{t}\paren{x; \tau}
        := \frac{
            x - \tau
        }{
            \paren{
                1 + w\paren{
                    x + t
                }^2
            }^{\frac{1}{2}}
        }.
    \end{align}
    Then, $L_{w}^{t}\paren{\bullet; \tau} : \mathbb{R}_{+} \to \left[- \frac{\tau}{\paren{1 + w t^2}^{\frac{1}{2}}}, w^{-\frac{1}{2}}\right)$
    is an increasing bijection, with an inverse designated by
    $\paren{L_{w}^{t}}^{-1}\paren{\bullet; \tau} : \left[- \frac{\tau}{\paren{1 + w t^2}^{\frac{1}{2}}}, w^{-\frac{1}{2}}\right) \to \mathbb{R}_{+}$.
\end{lemma}
\begin{proof}
    Let $x \in \mathbb{R}_+$ stand for a scalar input.
    \begin{align*}
        \paren{L_{w}^{t}}^{\prime}\paren{x; \tau}
        &
        = \frac{\partial}{\partial x}
        \paren{
            \frac{
                x - \tau
            }{
                \paren{
                    1 + w\paren{x + t}^2
                }^{\frac{1}{2}}
            }
        }
        \\
        &
        = \frac{
            \croch{
                \frac{\partial}{\partial x}\paren{x - \tau}
            }\paren{
                1 + w\paren{x + t}^2
            }^{\frac{1}{2}}
            - \paren{x - \tau}
            \croch{
                \frac{\partial}{\partial x}\paren{
                    1 + w\paren{x + t}^2
                }^{\frac{1}{2}}
            }
        }{
            1 + w\paren{x + t}^2
        }
        \\
        &
        = \frac{
            \paren{
                1 + w\paren{x + t}^2
            }^{\frac{1}{2}}
            - \paren{x - \tau}
            \frac{
                2w \paren{x + t}
            }{
                2 \paren{
                    1 + w\paren{x + t}^2
                }^{\frac{1}{2}}
            }
        }{
            1 + w\paren{x + t}^2
        }
        \\
        &
        =
        \frac{
            1 + w\paren{x + t}^2
            - w\paren{x - \tau}\paren{x + t}
        }{
            \paren{
                1 + w\paren{x + t}^2
            }^{\frac{3}{2}}            
        }
        =
        \frac{
            1 + w\paren{t + \tau}\paren{x + t}
        }{
            \paren{
                1 + w\paren{x + t}^2
            }^{\frac{3}{2}}            
        }
        .
    \end{align*}
    It follows that for every $x \in \mathbb{R}_{+}$,
    $\paren{L_{w}^{t}}^{\prime}\paren{x; \tau} > 0$, and thus,
    $L_{w}^{t}\paren{\bullet; \tau}$ is increasing over $\mathbb{R}_{+}$.
    Moreover, for every $x > 0$,
    \begin{align*}
        L_{w}^{t}\paren{x; \tau}
        = 
        \frac{
            x - \tau
        }{
            \paren{
                1 + w\paren{
                    x + t
                }^2
            }^{\frac{1}{2}}
        }
        = \frac{
            1 - \frac{\tau}{x}
        }{
            \paren{
                \frac{1}{x^2} + w\paren{
                    1 + \frac{t}{x}
                }^2
            }^{\frac{1}{2}}
        }
        \xrightarrow[x \rightarrow +\infty]{}
        \frac{1}{\paren{w}^{\frac{1}{2}}}
        = w^{-\frac{1}{2}},
    \end{align*}
    and $L_{w}^{t}\paren{0; \tau} = - \frac{\tau}{\paren{1 + w t^2}^{\frac{1}{2}}}$.
    This concludes the proof.
\end{proof}
\begin{lemma}
    \label{lm.variation.upper.transformation.function}
    Let $t, \tau, c \in \mathbb{R}_{+}$
    and $w \in \paren{0, +\infty}$ designate constants,
    $U_{w; c}^{t}\paren{\bullet; \tau} : \mathbb{R}_{+} \to \mathbb{R}$,
    the function given by, for every $x \in \mathbb{R}_{+}$,
    \begin{align}
        \label{eq.upper.transformation.function}
        U_{w; c}^{t}\paren{x; \tau}
        := \paren{x + \tau}\paren{1 + w\paren{cx + t}^2}^{\frac{1}{2}}.
    \end{align}
    Therefore, $U_{w; c}^{t}\paren{\bullet; \tau} : \mathbb{R}_+ \to \left[\tau \paren{1 + w t^{2}}^{\frac{1}{2}}, +\infty\right)$
    is an increasing bijection, with an inverse designated by
    $\paren{U_{w; c}^{t}}^{-1}\paren{\bullet; \tau} : \left[\tau \paren{1 + w t^{2}}^{\frac{1}{2}}, +\infty\right) \to \mathbb{R}_{+}$.
\end{lemma}
\begin{proof}
    Let $x \in \mathbb{R}_{+}$.
    \begin{align*}
        \paren{
            U_{w; c}^{t}
        }^{\prime}\paren{x; \tau}
        &
        = \frac{\partial}{\partial x}
        \croch{
            \paren{x + \tau}\paren{1 + w\paren{cx + t}^2}^{\frac{1}{2}}
        }
        \\
        &
        = \croch{
            \frac{\partial}{\partial x}
            \paren{x + \tau}
        }\paren{1 + w\paren{cx + t}^2}^{\frac{1}{2}}
        + \paren{x + \tau}
        \croch{
            \frac{\partial}{\partial x}\paren{1 + w\paren{cx + t}^2}^{\frac{1}{2}}
        }
        \\
        &
        =
        \paren{1 + w\paren{cx + t}^2}^{\frac{1}{2}}
        + \paren{x + \tau}
        \frac{1}{2}
        \frac{2wc \paren{cx + t}}{
            \paren{1 + w\paren{cx + t}^2}^{\frac{1}{2}}
        }
        \\
        &
        = 
        \frac{
            1 + w\paren{cx + t}^2
            + wc \paren{x + \tau}\paren{cx + t}
        }{
            \paren{1 + w\paren{cx + t}^2}^{\frac{1}{2}}
        }
        \\
        &
        =
        \frac{
            1 + w\paren{cx + t}
            \paren{
                cx + t + c\paren{x + \tau}
            }
        }{
            \paren{1 + w\paren{cx + t}^2}^{\frac{1}{2}}
        }
        \\
        &
        =
        \frac{
            1 + w\paren{cx + t}
            \paren{
                2cx + t + c\tau
            }
        }{
            \paren{1 + w\paren{cx + t}^2}^{\frac{1}{2}}
        }
        .
    \end{align*}
    It follows that for every $x \in \mathbb{R}_{+}$,
    $\paren{
        U_{w; c}^{t}
    }^{\prime}\paren{x; \tau} > 0$, and thus,
    $U_{w; c}^{t}\paren{\bullet; \tau}$ is increasing over $\mathbb{R}_{+}$.
    Moreover,
    \begin{align*}
        U_{w; c}^{t}\paren{x; \tau} = \paren{x + \tau}\paren{1 + w\paren{cx + t}^2}^{\frac{1}{2}}
        \xrightarrow[x \rightarrow + \infty]{} +\infty,
    \end{align*}
    and $U_{w; c}^{t}\paren{0; \tau} = \tau \paren{1 + w t^{2}}^{\frac{1}{2}}$.
    This concludes the proof.
\end{proof}

\subsection{Proof of Lemma~\ref{lm.stable.score.maha}}
\label{proof.stable.score.maha}
\begin{proof}
Let $\paren{x, u}$ designate a data point.
Since \eqref{asm.lip.loss} implies Assumption~\eqref{asm.lsc.loss},
then, under \eqref{asm.conv.loss}, \eqref{asm.lip.loss}
and \eqref{asm.lower.bounded.loss},
Lemma~\ref{lm.predictor.well.defined},
$\hat{f}_{\lambda; D^{y}}$ and $\hat{f}_{\lambda^{+}; D}$ are well-defined.
By Lemma~\ref{lm.global.norm.bound}, on the one hand,
\begin{align*}
    S_{\lambda; D^{y}}^{\widehat{\Gamma}_{a}}\paren{x, u}
    &
    \geq \frac{
        \paren{
            \frac{n+1}{n}
        }^{\frac{1}{2}}
        \norm{
            \widehat{\Gamma}_{a^{+}}^{-\frac{1}{2}}
            \paren{u - \hat{f}_{\lambda; D^{y}}(x)}
        }
    }{
        \paren{
            1 + \frac{1}{n+1}
            \norm{
                \widehat{\Gamma}_{a^{+}}^{-\frac{1}{2}}
                \paren{y - \widehat{\mu}_{D}}
            }^{2}
        }^{\frac{1}{2}}
    }
    \\
    &
    \geq \frac{
        \paren{
            \frac{n+1}{n}
        }^{\frac{1}{2}}
        \norm{
            \widehat{\Gamma}_{a^{+}}^{-\frac{1}{2}}
            \paren{u - \hat{f}_{\lambda; D^{y}}(x)}
        }
    }{
        \paren{
            1 + \frac{1}{n+1}
            \paren{
                \norm{
                    \widehat{\Gamma}_{a^{+}}^{-\frac{1}{2}}
                    \paren{y - \hat{f}_{\lambda^{+}; D}\paren{X_{n+1}}}
                } + \widehat{t}_{\lambda; a}
            }^{2}
        }^{\frac{1}{2}}
    },
\end{align*}
and on the other hand,
\begin{align*}
    S_{\lambda; D^{y}}^{\widehat{\Gamma}_{a}}\paren{x, u}
    \leq \paren{
        \frac{n+1}{n}
    }^{\frac{1}{2}}
    \norm{
        \widehat{\Gamma}_{a^{+}}^{-\frac{1}{2}}
        \paren{u - \hat{f}_{\lambda; D^{y}}(x)}
    }.
\end{align*}
Under \eqref{asm.conv.loss} and \eqref{asm.lip.loss},
\begin{align*}
    &
    \abss{
        \norm{
            \widehat{\Gamma}_{a^{+}}^{-\frac{1}{2}}
            \paren{u - \hat{f}_{\lambda; D^{y}}(x)}
        } - \norm{
            \widehat{\Gamma}_{a^{+}}^{-\frac{1}{2}}
            \paren{u - \hat{f}_{\lambda^{+}; D}(x)}
        }
    }
    \\
    &
    \leq
    \norm{
        \widehat{\Gamma}_{a^{+}}^{-\frac{1}{2}}
        \paren{
            \hat{f}_{\lambda; D^{y}}(x)
            - \hat{f}_{\lambda^{+}; D}(x)
        }
    }
    \\
    &
    \leq
    \norm{
        \widehat{\Gamma}_{a^{+}}^{-\frac{1}{2}}
        \paren{
            \hat{f}_{\lambda; D^{y}} - \hat{f}_{\lambda^{+}; D}
        }(x)
    }
    \\
    &
    \leq
    \norm{
        \widehat{\Gamma}_{a^{+}}^{-\frac{1}{2}}
        K\paren{x, x}
        \widehat{\Gamma}_{a^{+}}^{-\frac{1}{2}}
    }_{\mathrm{op}}^{\frac{1}{2}}
    \normh{\hat{f}_{\lambda; D^{y}} - \hat{f}_{\lambda^{+}; D}}
    \\
    &
    \leq
    \widehat{\tau}_{\lambda}^{\widehat{\Gamma}_{a^{+}}}(x)
\end{align*}
where the first inequality follows from the triangle inequality, and
the third inequality follows from Lemma~\ref{lm.gamma.prediction.bound}, and
the last inequality follows from Lemma~\ref{lm.stability.bounds}
and $\widehat{\tau}_{\lambda}^{\widehat{\Gamma}_{a^{+}}}(x)$
designates the stability-bound given by Eq.~\eqref{eq.global.score.stability.bound}.

One concludes by applying triangle inequality and the fact that the scores
are non-negative.
\end{proof}

\subsection{Proof of Proposition~\ref{prop.expression.upper.GlobalEllipsoidCP.region}}
\label{proof.expression.upper.GlobalEllipsoidCP.region}
\begin{proof}
Let $\alpha \in \left[\frac{1}{n+1}, 1\right)$ designate a control-level, and
$y \in \mathcal{Y}$, a test output-value.
\begin{align*}
    &
    y \in \widetilde{C}_{\lambda; \alpha}^{\widehat{\Gamma}_{a}, \up}\paren{X_{n+1}}
    \\
    &
    \Longleftrightarrow
    \frac{
        1 + \sum_{i=1}^{n}
        \mathbbm{1}\brac{
            \widetilde{S}_{\lambda; D^{y}}^{\widehat{\Gamma}_{a}, \up}\paren{X_i, Y_i}
            \geq \widetilde{S}_{\lambda; D^{y}}^{\widehat{\Gamma}_{a}, \lo}\paren{X_{n+1}, y}
        }
    }{
        n+1
    } > \alpha
    \\
    &
    \Longleftrightarrow
    \widetilde{S}_{\lambda; D^{y}}^{\widehat{\Gamma}_{a}, \lo}\paren{X_{n+1}, y}
    \leq \widetilde{S}_{\lambda; D^{y}}^{\widehat{\Gamma}_{a}, \up}\paren{X_{\paren{i_{n, \alpha}^{n}}}, Y_{\paren{i_{n, \alpha}^{n}}}}
    \\
    &
    \Longleftrightarrow
    \frac{
        \norm{
            \widehat{\Gamma}_{a^{+}}^{-\frac{1}{2}}
            \paren{y - \hat{f}_{\lambda^{+}; D}\paren{X_{n+1}}}
        }
        - \widehat{\tau}_{\lambda}^{\widehat{\Gamma}_{a^{+}}}\paren{X_{n+1}}
    }{
        \paren{
            1 +
            \frac{1}{n+1}
            \paren{
                \norm{
                    \widehat{\Gamma}_{a^{+}}^{-\frac{1}{2}}
                    \paren{
                        y - \hat{f}_{\lambda^{+}; D}\paren{X_{n+1}}
                    }
                }
                + \widehat{t}_{\lambda; a}
            }^2 
        }^{\frac{1}{2}}
    }
    \leq \sqrt{\frac{n}{n+1}}\widetilde{S}_{\lambda; D^{y}}^{\widehat{\Gamma}_{a}, \up}\paren{X_{\paren{i_{n, \alpha}^{n}}}, Y_{\paren{i_{n, \alpha}^{n}}}}
    \\
    &
    \Longleftrightarrow
    L_{\frac{1}{n+1}}^{\widehat{t}_{\lambda; a}}
    \paren{
        \norm{
            \widehat{\Gamma}_{a^{+}}^{-\frac{1}{2}}
            \paren{y - \hat{f}_{\lambda^{+}; D}\paren{X_{n+1}}}
        }; \widehat{\tau}_{\lambda}^{\widehat{\Gamma}_{a^{+}}}\paren{X_{n+1}}
    } \leq \widehat{Q}_{\lambda; D^{+}}^{\widehat{\Gamma}_{a}, \up}(\alpha)
\end{align*}
where the second equivalence follows from Lemma~\ref{lm.quantile} with $m = n$,
and the third equivalence, from the definition of the upper and lower
approximate non-conformity scores (see Lemma~\ref{lm.stable.score.maha}),
and the last equivalence follows from the definition of the function
\begin{align*}
    L_{\frac{1}{n+1}}^{\widehat{t}_{\lambda; a}}
    \paren{\bullet; \widehat{\tau}_{\lambda}^{\widehat{\Gamma}_{a^{+}}}\paren{X_{n+1}}} :
    \mathbb{R}_{+} \to
    \left[
        -\frac{\widehat{\tau}_{\lambda}^{\widehat{\Gamma}_{a^{+}}}\paren{X_{n+1}}}{\paren{1 + \frac{\widehat{t}_{\lambda; a}^{2}}{n+1}}^{\frac{1}{2}}},
        \sqrt{n+1}
    \right),
\end{align*}
by Eq.~\eqref{eq.lower.transformation.function},
and the quantile value $\widehat{Q}_{\lambda; D^{+}}^{\widehat{\Gamma}_{a}, \up}(\alpha)$ given by Eq.~\eqref{eq.upper.quantile.global.covariance}.
Since $\widehat{Q}_{\lambda; D^{+}}^{\widehat{\Gamma}_{a}, \up}(\alpha) \geq 0$, then
$\widehat{Q}_{\lambda; D^{+}}^{\widehat{\Gamma}_{a}, \up}(\alpha) \geq -\frac{\widehat{\tau}_{\lambda}^{\widehat{\Gamma}_{a^{+}}}\paren{X_{n+1}}}{\paren{1 + \frac{\widehat{t}_{\lambda; a}^{2}}{n+1}}^{\frac{1}{2}}}$.
Thus,  Lemma~\ref{lm.variation.lower.transformation.function} implies
that if $\widehat{Q}_{\lambda; D^{+}}^{\widehat{\Gamma}_{a}, \up}(\alpha) \geq \sqrt{n+1}$, then,
\begin{align*}
    y \in \widetilde{C}_{\lambda; \alpha}^{\widehat{\Gamma}_{a}, \up}\paren{X_{n+1}}
    \Longleftrightarrow y \in \mathcal{Y}.
\end{align*}
Otherwise, that is,
if $\widehat{Q}_{\lambda; D^{+}}^{\widehat{\Gamma}_{a}, \up}(\alpha) < \sqrt{n+1}$,
then,
\begin{align*}
    &
    y \in \widetilde{C}_{\lambda; \alpha}^{\widehat{\Gamma}_{a}, \up}\paren{X_{n+1}}
    \\
    &
    \Longleftrightarrow \norm{
        \widehat{\Gamma}_{a^{+}}^{-\frac{1}{2}}
        \paren{y - \hat{f}_{\lambda^{+}; D}\paren{X_{n+1}}}
    } \leq \paren{
        L_{\frac{1}{n+1}}^{\widehat{t}_{\lambda; a}}
    }^{-1}
    \paren{
        \widehat{Q}_{\lambda; D^{+}}^{\widehat{\Gamma}_{a}, \up}(\alpha);
        \widehat{\tau}_{\lambda}^{\widehat{\Gamma}_{a^{+}}}\paren{X_{n+1}}
    }.
\end{align*} 
\end{proof}

\subsection{Proof of Proposition~\ref{prop.expresion.lower.GlobalEllipsoidCP.region}}
\label{proof.expresion.lower.GlobalEllipsoidCP.region}
\begin{proof}
Let $\alpha \in \left[\frac{1}{n+1}, 1\right)$ stand for a control-level, and
$y \in \mathcal{Y}$, a test output-value.
\begin{align*}
    &
    y \in \widetilde{C}_{\lambda; \alpha}^{\widehat{\Gamma}_{a}, \lo}\paren{X_{n+1}}
    \\
    &
    \Longleftrightarrow
    \frac{
        1 + \sum_{i=1}^{n}
        \mathbbm{1}\brac{
            \widetilde{S}_{\lambda; D^{y}}^{\widehat{\Gamma}_{a}, \lo}\paren{X_i, Y_i}
            \geq \widetilde{S}_{\lambda; D^{y}}^{\widehat{\Gamma}_{a}, \up}\paren{X_{n+1}, y}
        }
    }{
        n+1
    } > \alpha
    \\
    &
    \Longleftrightarrow
    \widetilde{S}_{\lambda; D^{y}}^{\widehat{\Gamma}_{a}, \up}\paren{X_{n+1}, y}
    \leq \widetilde{S}_{\lambda; D^{y}}^{\widehat{\Gamma}_{a}, \lo}\paren{X_{\paren{i_{n, \alpha}^{n}}}, Y_{\paren{i_{n, \alpha}^{n}}}}
    \\
    &
    \Longleftrightarrow
    \paren{
        \frac{n+1}{n}
    }^{\frac{1}{2}}
    \croch{
        \norm{
            \widehat{\Gamma}_{a^{+}}^{-\frac{1}{2}}
            \paren{y - \hat{f}_{\lambda^{+}; D}\paren{X_{n+1}}}
        }
        + \widehat{\tau}_{\lambda}^{\widehat{\Gamma}_{a^{+}}}\paren{X_{n+1}}
    }
    \\
    &
    \qquad
    \leq \frac{
        \paren{
            \frac{n+1}{n}
        }^{\frac{1}{2}}
        \widehat{Q}_{\lambda; D^{+}}^{\widehat{\Gamma}_{a}, \lo}(\alpha)
    }{
        \paren{
            1 +
            \frac{1}{n+1}
            \paren{
                \norm{
                    \widehat{\Gamma}_{a^{+}}^{-\frac{1}{2}}
                    \paren{
                        y - \hat{f}_{\lambda^{+}; D}\paren{X_{n+1}}
                    }
                } +  \widehat{t}_{\lambda; a}
            }^2
        }^{\frac{1}{2}}
    },
    \\
    &
    \Longleftrightarrow
    U_{\frac{1}{n+1}; 1}^{\widehat{t}_{\lambda; a}}
    \paren{
        \norm{
            \widehat{\Gamma}_{a^{+}}^{-\frac{1}{2}}
            \paren{y - \hat{f}_{\lambda^{+}; D}\paren{X_{n+1}}}
        }; \widehat{\tau}_{\lambda}^{\widehat{\Gamma}_{a^{+}}}\paren{X_{n+1}}
    }
    \leq
    \widehat{Q}_{\lambda; D^{+}}^{\widehat{\Gamma}_{a}, \lo}(\alpha),
\end{align*}
where the second equivalence follows from Lemma~\ref{lm.quantile} with $m = n$,
and the third equivalence, from the definition of the upper and lower
approximate non-conformity scores (see Lemma~\ref{lm.stable.score.maha}),
and the last equivalence, from the definition of the function
\begin{align*}
    U_{\frac{1}{n+1}; 1}^{\widehat{t}_{\lambda; a}}\paren{\bullet; \widehat{\tau}_{\lambda}^{\widehat{\Gamma}_{a^{+}}}\paren{X_{n+1}}} : \mathbb{R}_{+} \to \left[\widehat{\tau}_{\lambda}^{\widehat{\Gamma}_{a^{+}}}\paren{X_{n+1}}\paren{1 + \frac{\widehat{t}_{\lambda; a}^2}{n+1}}^{\frac{1}{2}}, +\infty\right),
\end{align*}
given by Eq.~\eqref{eq.upper.transformation.function}, and
the quantile value $\widehat{Q}_{\lambda; D^{+}}^{\widehat{\Gamma}_{a}, \lo}(\alpha)$ given by Eq.~\eqref{eq.lower.quantile.global.covariance}.
Thus, it follows from Lemma~\ref{lm.variation.upper.transformation.function}
that if $\widehat{Q}_{\lambda; D^{+}}^{\widehat{\Gamma}_{a}, \lo}(\alpha)
< \widehat{\tau}_{\lambda}^{\widehat{\Gamma}_{a^{+}}}\paren{X_{n+1}}\paren{1 + \frac{\widehat{t}_{\lambda; a}^2}{n+1}}^{\frac{1}{2}}$, then,
\begin{align*}
    y \in \widetilde{C}_{\lambda; \alpha}^{\widehat{\Gamma}_{a}, \lo}\paren{X_{n+1}}
    \Longleftrightarrow 
    y \in \emptyset.
\end{align*}
Otherwise, that is, if $\widehat{Q}_{\lambda; D^{+}}^{\widehat{\Gamma}_{a}, \lo}(\alpha)
\geq \widehat{\tau}_{\lambda}^{\widehat{\Gamma}_{a^{+}}}\paren{X_{n+1}}\paren{1 + \frac{\widehat{t}_{\lambda; a}^2}{n+1}}^{\frac{1}{2}}$, then,
\begin{align*}
    &
    y \in \widetilde{C}_{\lambda; \alpha}^{\widehat{\Gamma}_{a}, \lo}\paren{X_{n+1}}
    \\
    &
    \Longleftrightarrow \norm{
        \widehat{\Gamma}_{a^{+}}^{-\frac{1}{2}}
        \paren{y - \hat{f}_{\lambda^{+}; D}\paren{X_{n+1}}}
    } \leq \paren{
        U_{\frac{1}{n+1}; 1}^{\widehat{t}_{\lambda; a}}
    }^{-1}\paren{
        \widehat{Q}_{\lambda; D^{+}}^{\widehat{\Gamma}_{a}, \lo}(\alpha);
        \widehat{\tau}_{\lambda}^{\widehat{\Gamma}_{a^{+}}}\paren{X_{n+1}}
    }.
\end{align*}
\end{proof}

\subsection{Additional numerical experiment on the evolution of the thickness}
\label{sec.evolution.thickness.global}
\begin{figure}[H]
    \centering
    \includegraphics[width=0.60\textwidth]{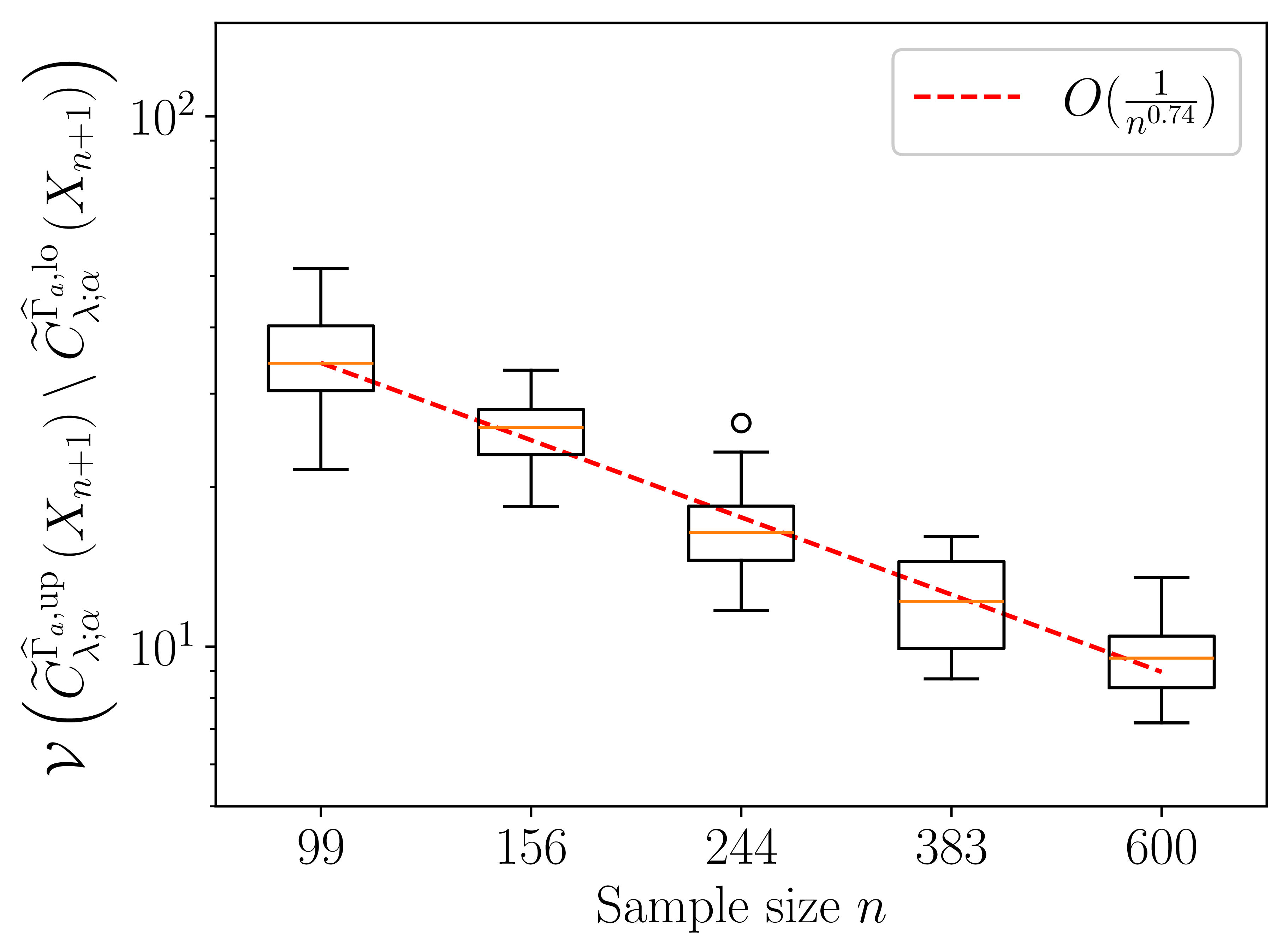}
    \caption{
        Evolution of the computable empirical upper-bound $\leb{\widetilde{C}_{\lambda; \alpha}^{\widehat{\Gamma}_{a}, \mathrm{up}}\paren{X_{n+1}} \setminus \widetilde{C}_{\lambda; \alpha}^{\widehat{\Gamma}_{a}, \lo}\paren{X_{n+1}}}$
        for the \emph{thickness} $\thicc{\widehat{\Gamma}_{a}}$ across 20 repetitions for $\alpha = 0.1$
        and for $\lambda \propto \frac{1}{\sqrt{n}}$.
    }
    \label{fig.thickness.GlobalEllipsoidCP.1}
\end{figure}

On average, $\leb{\widetilde{C}_{\lambda; \alpha}^{\widehat{\Gamma}_{a}, \mathrm{up}}\paren{X_{n+1}} \setminus \widetilde{C}_{\lambda; \alpha}^{\widehat{\Gamma}_{a}, \lo}\paren{X_{n+1}}}$
gets smaller as the training-sample size $n$ gets larger.
Additionally, the estimated rate of improvement (represented by the dashed red-line) is
about the same as the theoretical one, that is, $O\paren{\frac{1}{\sqrt{n}}}$
since $\lambda \propto \frac{1}{\sqrt{n}}$.
That is to say, even when $\lambda \propto \frac{1}{\sqrt{n}}$,
integrating a global inter-task covariance estimator
does not seem to affect the rate of improvement of the \emph{thickness}.

\section{Other conformal strategies}
The present section presents the first alternative to \textbf{FullCP},
that is, \textbf{SplitCP}, and a prediction-region of reference
used in place of the \textbf{FullCP}-region, that is, the \textbf{OracleCP}-region.

\subsection{Split conformal prediction (\textbf{SplitCP})}
\label{sec.split.cp}
Given a random partition of the initial data set $D$ into
two disjoint data set, $D_{\mathrm{train}}$ and $D_{\mathrm{calib}}$,
the split conformal prediction (\textbf{SplitCP}) region \citep{papadopoulosInductiveConformalPrediction2008} is given by
\begin{align*}
    \widehat{C}_{\lambda; \alpha}^{\mathrm{split}}\paren{X_{n+1}}
    := \brac{
        y \in \mathcal{Y}:
        \widehat{\pi}_{\lambda; D}^{\mathrm{split}}\paren{X_{n+1}, y} > \alpha
    },
\end{align*}
where $\widehat{\pi}_{\lambda; D}^{\mathrm{split}}\paren{X_{n+1}, \bullet} : \mathcal{Y} \to \left[\frac{1}{n+1}, 1\right]$
denote the split conformal p-value function, given by,
for every $y \in \mathcal{Y}$,
\begin{align*}
    \widehat{\pi}_{\lambda; D}^{\mathrm{split}}\paren{X_{n+1}, y}
    := \frac{
        1 + \sum_{\paren{x, u} \in D_{\mathrm{cal}}}
        \mathbbm{1}\brac{
            \widehat{s}_{D_{\mathrm{train}}}
            \paren{u, \hat{f}_{\lambda; D_{\mathrm{train}}}\paren{x}}
            \geq \widehat{s}_{D_{\mathrm{train}}}
            \paren{y, \hat{f}_{\lambda; D_{\mathrm{train}}}\paren{X_{n+1}}}
        }
    }{\abss{D_{\mathrm{cal}}}+1}.
\end{align*}
As a result, by Lemma~\ref{lm.quantile},
\begin{align*}
    \widehat{C}_{\lambda; \alpha}^{\mathrm{split}}\paren{X_{n+1}}
    = \brac{
        y \in \mathcal{Y}:
        \widehat{s}_{D_{\mathrm{train}}}
        \paren{y, \hat{f}_{\lambda; D_{\mathrm{train}}}\paren{X_{n+1}}}
        \leq \widehat{Q}_{\lambda}^{\mathrm{split}}\paren{\alpha},
    },
\end{align*}
where the quantile value $\widehat{Q}_{\lambda}^{\mathrm{split}}\paren{\alpha}$
is given by,
\begin{align*}
    \widehat{Q}_{\lambda}^{\mathrm{split}}\paren{\alpha}
    := \widehat{s}_{D_{\mathrm{train}}}
    \paren{Y_{\paren{i_{n_{\mathrm{cal}}, \alpha}^{n_{\mathrm{cal}}}}}, \hat{f}_{\lambda; D_{\mathrm{train}}}\paren{X_{\paren{i_{n_{\mathrm{cal}}, \alpha}^{n_{\mathrm{cal}}}}}}},
\end{align*}
where $n_{\mathrm{cal}} := \abss{D_{\mathrm{cal}}}$,
and index $i_{n_{\mathrm{cal}}, \alpha}^{n_{\mathrm{cal}}} \in \brac{1, \ldots, n_{\mathrm{cal}}}$
is given by, $i_{n_{\mathrm{cal}}, \alpha}^{n_{\mathrm{cal}}} := \ceil{\paren{n_{\mathrm{cal}+1}}\paren{1 - \alpha}}$.
\begin{corollary}
    For any confidence control-level $\alpha \in \left[\frac{1}{n_{\mathrm{cal}}+ 1}, 1\right)$,
    the \textbf{SplitCP}-region $\widehat{C}_{\lambda; \alpha}^{\mathrm{split}}\paren{X_{n+1}}$ enjoys
    the following guarantee,
    \begin{align*}
        \mathbb{P}\croch{Y_{n+1} \in \widehat{C}_{\lambda; \alpha}^{\mathrm{split}}\paren{X_{n+1}}} \geq 1 - \alpha,
    \end{align*}
    as such it is a confidence prediction-region. Furthermore, if the non-conformity scores
    $\widehat{s}_{D_{\mathrm{train}}}\paren{Y_1, \hat{f}_{\lambda; D_{\mathrm{train}}}\paren{X_1}}$,
    \ldots, $\widehat{s}_{D_{\mathrm{train}}}\paren{Y_{n_{\mathrm{cal}}}, \hat{f}_{\lambda; D_{\mathrm{train}}}\paren{X_{n_{\mathrm{cal}}}}}$
    and\\
    $\widehat{s}_{D_{\mathrm{train}}}\paren{Y_{n+1}, \hat{f}_{\lambda; D_{\mathrm{train}}}\paren{X_{n+1}}}$,
    are almost distinct, then,
    \begin{align*}
        \mathbb{P}\croch{Y_{n+1} \in \widehat{C}_{\lambda; \alpha}^{\mathrm{split}}\paren{X_{n+1}}} \leq 1 - \alpha + \frac{1}{n_{\mathrm{cal}} + 1}.
    \end{align*}
\end{corollary}
\begin{proof} \textbf{SplitCP} can be seen as special case of \textbf{FullCP},
conditional on the proper training-data points,
the calibration data points corresponds to the training data points in \textbf{FullCP},
and the non-conformity score function and the predictor are deterministic.
\end{proof}

\subsection{Oracle conformal prediction (\textbf{OracleCP})}
\label{sec.oracle.cp}
Given the unknown output-vector $Y_{n+1}$,
the oracle conformal prediction (\textbf{OracleCP}) region is given by
\begin{align*}
    \widehat{C}_{\lambda; \alpha}^{\mathrm{oracle}}\paren{X_{n+1}}
    := \brac{
        y \in \mathcal{Y}:
        \widehat{\pi}_{\lambda; D}^{\mathrm{oracle}}\paren{X_{n+1}, y} > \alpha
    },
\end{align*}
where $\widehat{\pi}_{\lambda; D}^{\mathrm{oracle}}\paren{X_{n+1}, \bullet} : \mathcal{Y} \to \left[\frac{1}{n+1}, 1\right]$
denote the split conformal p-value function, given by,
for every $y \in \mathcal{Y}$,
\begin{align*}
    \widehat{\pi}_{\lambda; D}^{\mathrm{oracle}}\paren{X_{n+1}, y}
    := \frac{
        1 + \sum_{i=1}^{n}
        \mathbbm{1}\brac{
            \widehat{s}_{D^{Y_{n+1}}}
            \paren{Y_i, \hat{f}_{\lambda; D^{Y_{n+1}}}\paren{X_i}}
            \geq \widehat{s}_{D^{Y_{n+1}}}
            \paren{y, \hat{f}_{\lambda; D^{Y_{n+1}}}\paren{X_{n+1}}}
        }
    }{n+1}.
\end{align*}
As a result, by Lemma~\ref{lm.quantile},
\begin{align*}
    \widehat{C}_{\lambda; \alpha}^{\mathrm{oracle}}\paren{X_{n+1}}
    = \brac{
        y \in \mathcal{Y}:
        \widehat{s}_{D^{Y_{n+1}}}
        \paren{y, \hat{f}_{\lambda; D^{Y_{n+1}}}\paren{X_{n+1}}}
        \leq \widehat{Q}_{\lambda}^{\mathrm{oracle}}\paren{\alpha},
    },
\end{align*}
where the quantile value $\widehat{Q}_{\lambda}^{\mathrm{oracle}}\paren{\alpha}$
is given by,
\begin{align*}
    \widehat{Q}_{\lambda}^{\mathrm{oracle}}\paren{\alpha}
    := \widehat{s}_{D^{Y_{n+1}}}
    \paren{Y_{\paren{i_{n, \alpha}^{n}}},
    \hat{f}_{\lambda; D^{Y_{n+1}}}\paren{X_{\paren{i_{n, \alpha}^{n}}}}},
\end{align*}
where the index $i_{n, \alpha}^{n} \in \brac{1, \ldots, n}$
is given by, $i_{n, \alpha}^{n} := \ceil{\paren{n+1}\paren{1 - \alpha}}$.
\begin{corollary}
    The \textbf{OracleCP}-region enjoys all guarantees enjoyed by the \textbf{FullCP}-region,
    stated in Theorem~\ref{thm.coverage}.
\end{corollary}

\newpage
\bibliography{sample}

\end{document}